\documentclass[10pt,a4paper,final,reqno]{amsart}
\usepackage[utf8]{inputenc}
\usepackage{amsmath}
\usepackage{amsfonts}
\usepackage{amssymb}
\usepackage{graphicx}
\usepackage{pifont}
\usepackage{amssymb,amsmath,amscd, amsthm,stmaryrd,verbatim,
	mathrsfs,tikz,color , amsfonts, }
\usetikzlibrary{patterns}
\usetikzlibrary{fit}
\usepackage[colorlinks,linkcolor=blue,urlcolor=cyan,citecolor=green
]{hyperref}
\usepackage{tikz}
\usepackage{caption}
\usepackage{subcaption}
\usepackage{etoolbox}
\usepackage{enumitem}

\newcommand{\circled}[1]{\begin{tikzpicture}[inner sep=0.001pt,baseline=-3pt]
\node (a) at (0,0){{\normalfont#1}};
\node[draw,circle,fit=(a)]{};
\end{tikzpicture}}
\robustify{\circled}

\allowdisplaybreaks
\newcommand{\af}{ \mathbf{a} }
\newcommand{\mb}[1]{\mathbf{#1}}
\newcommand{\mc}[1]{\mathcal{#1}}

\newcommand{\mf}[1]{\mathfrak{#1}}
\renewcommand{\subset}{\subseteq}

\newcommand{\nt}{{}^{{ N}}\! T}

\newcommand{\nf}{{}^{{ N}}\! f}

\newcommand{\nc}{{}^{{ N}}\! C}

\newcommand{\wn}{W_{\leq N}}

\newcommand{\hn}{\mc{H}_{\leq N} }

\newtheorem{thm}{Theorem}[section]
\newtheorem{prop}[thm]{Proposition}
\newtheorem{lem}[thm]{Lemma}
\newtheorem{cor}[thm]{Corollary}

\newtheorem{rem}[thm]{Remark}
\newtheorem{deff}[thm]{Definition}
\newtheorem{ex}[thm]{Example}
\newtheorem{conj}[thm]{Conjecture}
\numberwithin{equation}{section}

\begin{document}
\bibliographystyle{alpha}

\title[Hyperbolic Coxeter groups of rank 3]{Conjectures P1-P15 for hyperbolic Coxeter groups of rank 3}
	
\author[J. Gao]{Jianwei Gao}
\address{Beijing International Center for Mathematical Research, Peking University, Beijing 100871, China}
\email{gaojianwei827@126.com}
\author[X. Xie]{Xun Xie}
\address{School of Mathematics and Statistics, Beijing Institute of Technology, Beijing 100081, China}
\email{xieg7@163.com}

\date{\today}
	
\subjclass[2010]{Primary 20C08; Secondary  20F55}
\keywords{$ \mathbf{a} $-functions, two-sided cells, Conjectures P1-P15, Coxeter groups of rank 3}

\maketitle

\begin{abstract}
We prove  Lusztig's conjectures P1-P15 for hyperbolic Coxeter groups of rank 3. 
Our proof enables us to give a description of the $ \af $-function and Kazhdan-Lusztig cells for these Coxeter groups.
\end{abstract}
\setcounter{tocdepth}{1}
\tableofcontents

\section{Introduction}\label{sec:intr}

Lusztig proposed a series of conjectures, called P1-P15, on Hecke algebras with unequal parameters in \cite[\S14]{lusztig2003}. They are stated in a very general form, i.e. for any finitely generated Coxeter group and any positive weight function. These conjectures predict that there are nice relations between $ \af $-functions and cells, and one can define a kind of asymptotic rings.
%
Lusztig proved these conjectures for constant weight functions (the equal parameter case) by assuming the boundedness conjecture about the  $ \af $-function  \cite[\S13.4]{lusztig2003} and the positivity of the Kazhdan-Lusztig basis.  The positivity has been proved by \cite{elias-williamson2014positivity}, using some deep ideas from the Hodge theory. However, there is no positivity for the unequal parameter case, and the problem becomes mysterious. Up to now we have no efficient way to prove P1-P15 for arbitrary parameters.
The main difficulty probably lies in computing $ \af $-functions.

In \cite{lusztig2003}, Lusztig proved P1-P15 for infinite dihedral groups and the quasi-split case.
For P1-P15 of finite Weyl groups, only the case of type $ B_n $ with ``non-asymptotic parameters" remains open, see \cite{geck2011rank2} and references therein. The universal Coxeter group case is proved in \cite{Shi-Yang2015universal}. Recently,  the case of affine Weyl groups of rank 2 is proved in \cite{guilhot-parkinson2019C2,guilhot-parkinson2019G2}, by studying cell representations 
and a connection with the Plancherel formula.

This paper is a subsequent work of \cite{xie2019}, where the second named author proved P1-P15  for Coxeter groups with complete graph and right-angled Coxeter groups. The aim of this paper is to complete the proof of conjectures P1-P15  for all Coxeter groups of rank 3. The cases of finite and affine types are known, see \cite[25A.1]{bonnafe2017book} and \cite{guilhot-parkinson2019C2,guilhot-parkinson2019G2}.
Thus we  focus on  hyperbolic Coxeter groups of rank 3.
Since the cases of Coxeter groups with complete graph and right-angled Coxeter groups have been proved in \cite{xie2019},  we only need to consider   Coxeter groups $ (W,S) $
such that $ S=\{r,s,t\} $,  $ m_{rt}=2 $, $ \frac1{m_{rs}}+\frac1{m_{rs}}<\frac12 $, and at least one of $ m_{rs} $, $ m_{st} $ is not $  \infty$.
These groups are divided into three classes in our proofs:
\begin{itemize}
\item [(1)]$\infty=m_{rs}>m_{st}\geq3$;
\item [(2)] $ \infty> m_{rs}, m_{st}\geq 4$ but $(m_{rs},m_{st})\neq (4,4)$;
\item [(3)] $ \infty>m_{rs}\geq 7, m_{st}=3$.
\end{itemize}
 Their Coxeter graphs look like
\begin{tikzpicture}[scale=1.5,baseline=0]
\draw (0,0.05) node[above=1pt]{$ r $} circle [radius=0.05];\draw (0.05,0.05)--++(0.65,0); \draw (0.75,0.05) node[above=1pt]{$ s $} circle[radius=0.05];  \draw (0.8,0.05)--++(0.65,0); \draw  (1.45,0) ++(0.05,0.05) node[above=1pt]{$ t $} circle[radius=0.05];
\draw (0.35,0.05)node [below=0.1pt]{$ m_{rs} $}  ++(0.8,0) node[below=0.1pt]{$ m_{st} $};
\end{tikzpicture}.

The main  methods of proving P1-P15 in \cite{xie2019}  are  applying decreasing induction on $ \af $-values and proving a kind of decomposition formula for some Kazhdan-Lusztig basis elements in a quotient algebra.
In this paper, we prove P1-P15 for Coxeter groups in (1) (2) and (3) by similar methods. 
However, the proofs of some key properties become complicated. In fact all the sections \ref{sec:dih}--\ref{sec:est} of this paper are devoted to prove them. Once we obtain these properties, we can repeat arguments in \cite{xie2019} to prove P1-P15, see \cite[\S9]{xie2019}.
%
 The proofs in sections \ref{sec:take}-\ref{sec:est} depend on some explicit expansions of  products in the Hecke algebra (section \ref{sec:exp}).  Since results from section \ref{sec:exp} are discrete, we have to verify case by case in sections \ref{sec:take}-\ref{sec:est}.
The ideas of the proofs  are simple, and for most cases the verification is easy.

The $ \af $-functions and  cells are often not easy to determine, even for equal parameter cases. A benefit of our proof of P1-P15 is that we have a description of the $ \af $-functions and left (right) cells, see Theorem \ref{thm:92}. In the complete graph case, each $ W_N $ (the set of elements of $ W $ with $ \af $-value $ N $) is either empty or a two-sided cell. However this is not  correct in our present situation. We give a complete list of cases  where $ W_N $ contains more than one two-sided cells, see Theorem \ref{thm:twocells}. As a corollary, we conclude that, in the equal parameter case,  a non-empty $ W_N $ is always a two-sided cell. This confirms \cite[Conj. 3.1]{Belolipetsky2014hyperbolic_cells} in the case of rank 3. At last, we give some examples to show how the cell partitions depend on the parameters, which should be helpful for understanding the semi-continuity conjecture formulated in 
\cite{bonnafe2009semi-cont}.

This article is organized as follows.
In section \ref{sec:pre}, we prepare some notations and basic facts.
In section \ref{sec:3}, we state Conjecture \ref{conj:dim2} for Coxeter groups of dimension 2, and then claim that it implies P1-P15  (Theorem \ref{thm:main}).
A more general form of  Conjecture \ref{conj:dim2}  is given in \cite[\S 9]{xie2019}.
In section  \ref{sec:dih}, we prepare some facts about dihedral groups that we will use. In section \ref{sec:exp}, we give explicit expansions of some products, which are fundamental  in our proof of key properties. The main goal of sections \ref{sec:take}-\ref{sec:est} is to prove Conjecture \ref{conj:dim2} for those Coxeter groups listed above, see Proposition \ref{prop:bound}, Lemma \ref{lem:length}, and Proposition \ref{prop:strictineq}. In the section \ref{sec:9}, we arrive at our main result and determine two-sided cells.

\section{Preliminaries}\label{sec:pre}

Let us fix some notations. Let $ (W,S) $ be a Coxeter group with $ |S|<\infty $. For $ s,t\in S $, $ m_{st}\in \mathbb{N}\cup\{\infty\} $ is the order of $ st $ in $ W $. The neutral element of $ W $ is denoted by $ e $. We have the length function $ l:W\to \mathbb{N} $. For $ I\subset S $, the parabolic subgroup generated by $ I $ is denoted by $ W_I $. If $ I $ is finite, the longest element is denoted by $ w_I $. For $s,t\in S$ with $s\neq t$, we use notation $W_{st}$ instead of $W_{\{s,t\}}$ and $w_{st}$ instead of $w_{\{s,t\}}$. If $ x\in W $ can be expressed as  a product  $ x_1x_2\ldots x_k $ of some elements $ x_i\in W $ with $ l(x)=\sum_{1\leq i\leq k}l(x_i) $, then we say $ x_1x_2\ldots x_k $  is a \textit{reduced product}. We use the notation $ x_1\cdot x_2 \cdot \ldots \cdot x_k  $ to indicate that $ x_1x_2\ldots x_k $ is a reduced product. For $x,y\in W$, we say $x$ appears in $y$, if there exist $w,z\in W$ such that $y=w\cdot x\cdot z$. For  $ x\in W $,
\(\mc{L}(x)=\{s\in S\mid sx<x \}, \mc{R}(x)=\{s\in S\mid xs<x \}.\)

Let $ L:W\to \mathbb{N} $ be a  weight function on $ W $. In other words, $L$ satisfies $L(xy)=L(x)+L(y)$ for any $x,y\in W$ with $l(xy)=l(x)+l(y)$. Unless otherwise stated, $ L $ is assumed to be positive, i.e.  $L(s)>0$ for any $s\in S$. We call $(W,S,L)$ a (positively) weighted Coxeter group. Let $ \mc{A}=\mathbb{Z}[q,q^{-1}] $ be the ring of Laurent polynomials in  indeterminate  $ q $. The Hecke algebra $ \mc{H} $ of $(W,S,L)$ is a unital associative algebra over  $ \mc{A} $ with $ \mc{A} $-basis $ \{T_w\mid w\in W \} $ subject to relations: $$  T_{ww'}=T_wT_{w'} \text{ if } l(ww')=l(w)+l(w'),  $$ $$ \text{ and } T_s^2=1+\xi_s T_s \text{ with } \xi_s=q^{L(s)}-q^{-L(s)} \in\mc{A}.$$
For $ 0\neq a=\sum_i\alpha_i q^i\in \mc{A} $ with $ \alpha_i \in \mathbb{Z}$, we define $ \deg a=\max\{i\mid \alpha_i\neq 0 \} $. For $ 0\in \mc{A} $, we define $ \deg 0=-\infty $. For $  h=\sum_{w\in W} a_w T_w\in \mc{H}$ with $ a_w\in\mc{A} $, we define $ \deg h=\max\{ \deg a_w \mid w\in W\} $. This gives a function $ \deg : \mc{H}\to \mathbb{Z}\cup \{-\infty\} $.

Denote by $ C_w $, $ w\in W $ the Kazhdan-Lusztig basis of $ \mc{H} $. We have $ C_w=\sum_{y\leq w}p_{y,w} T_y $ with $ p_{y,w}\in \mc{A}_{< 0}=q^{-1}\mathbb{Z}[q^{-1}] $ for $ y<w $ and $ p_{w,w}=1 $. Moreover, $ C_w $ is invariant under the bar involution $ \overline{\phantom{q}} $ of $ \mc{H} $ such that $ \overline{q}=q^{-1} $ and $ \overline{T_w}=T_{w^{-1}}^{-1} $. Using Kazhdan-Lusztig basis, one can define preorders  $  \prec_{ {L}}$, $ \prec_{ {R}} $, $ \prec_{ {LR}} $ and  equivalence relations $ \sim_{ {L}} $, $ \sim_{ {R}} $, $ \sim_{ {LR}} $ on $ W $. The associated equivalence classes are called respectively left cells, right cells and two-sided cells. Let $ f_{x,y,z} $ and $ h_{x,y,z} $ be elements of $ \mc{A} $ given by \[
T_xT_y=\sum_{z\in W} f_{x,y,z}T_z,\quad C_xC_y=\sum_{z\in W} h_{x,y,z}C_z.
\]
For $ w\in W $, define $$  \af(w):=\max\{\deg h_{x,y,w}\mid x,y\in W \}  .$$ Then $ \af:W\to \mathbb{N}\cup\{\infty \} $ is called Lusztig's $ \af $-function. Define $ \gamma_{x,y,z^{-1}}\in \mathbb{Z}$ to be the coefficient of $ q^{\af(z)} $ in $ h_{x,y,z} $.
For $ w\in W $, integers $ \Delta(w) $ and $ n_w $ are
defined by\[
p_{e,w}=n_wv^{-\Delta(w)}+\text{ terms of lower degree, with } n_w\neq0.
\]
Let $ \mc{D}=\{z\in W \mid \af(z)=\Delta(z) \} $.
For $ N\in\mathbb{N} $, we set
\[
W_{\geq N}:=\{w\in W\mid \af(w)\geq N \},
\]
\[
W_{>N}:=W_{\geq (N+1)},\quad W_{N}:=W_{\geq N}\setminus W_{>N},
\]
and similarly  define  $ W_{\leq N} $,  $ W_{<N} $. Let  $ \mc{D}_{\geq N}=\mc{D}\cap W_{\geq N} $, and similarly define $  \mc{D}_{N}  $ etc. The boundedness conjecture says that \begin{equation}\label{eq:bound-conj}
 W_{>N_0} =\emptyset\text{ for } N_0=\max\{L(w_I)\mid I\subset S\text{ with }W_I\text{ finite}\}.
\end{equation}  This conjecture holds for Coxeter groups of rank 3 by \cite{zhou,gao}.
\begin{conj}\label{conj15} Let $ N\in\mathbb{N} $.
	
	$ (\text{P1})_{\geq N} $. For any  $w\in W_{\geq N} $, we have $\af(w)\leq\Delta(w)$.
	
$ (\text{P2})_{\geq N} $. If $ z\in\mc{D}_{\geq N} $ and $ x,y\in W $ such that $ \gamma_{x,y,z}\neq0 $, then $ x=y^{-1} $.
	
$ (\text{P3})_{\geq N} $. If $ y\in W_{\geq N} $, there exists a unique $ z\in \mc{D} $ such that $ \gamma_{y^{-1},y,z} \neq0$.
	
$ (\text{P4})_{\geq N} $. If $ w'\prec_{LR} w $ with $ w\in W_{\geq N} $, then $ \af(w')\geq\af(w) $.
	
$ (\text{P5})_{\geq N} $. If $z\in\mathcal{D}_{\geq N}$, $ y\in W $, $\gamma_{y^{-1},y,z}\neq0$, then $\gamma_{y^{-1},y,z}=n_z=\pm1$.
	
$ (\text{P6})_{\geq N} $. For any $z\in\mathcal{D}_{\geq N}$, we have  $z^2=e$.
	
$ (\text{P7})_{\geq N} $.  For any  $x,y,z\in W$ with one of them belonging to $ W_{\geq N} $, we have $\gamma_{x,y,z}=\gamma_{y,z,x}=\gamma_{z,x,y}$.
	
$ (\text{P8})_{\geq N} $.    For any $x,y,z\in W$ with one of them belonging to $ W_{\geq N} $, then $\gamma_{x,y,z}\neq0$ implies that  $x\sim_{L}y^{-1}$, $y\sim_{ L}z^{-1}$, $z\sim_{ L}x^{-1}$.
	
$ (\text{P9})_{\geq N} $.  If  $w'\prec_{ {L}}w$ with $ w\in W_{\geq N} $ and $\af(w')=\af(w)$, then $w'\sim_{ {L}}w$.
	
$ (\text{P10})_{\geq N} $. If  $w'\prec_{ {R}}w$ with $ w\in W_{\geq N} $ and $\af(w')=\af(w)$, then $w'\sim_{ {R}}w$.
	
$ (\text{P11})_{\geq N} $.  If  $w'\prec_{ {LR}}w$ with $ w\in W_{\geq N} $ and $\af(w')=\af(w)$, then $w'\sim_{ {LR}}w$.
	
$ (\text{P12})_{\geq N} $. For any $I\subset S$ and $y\in W_I\cap W_{\geq N}$,  the $ \mb{a} $-value of $ y $ in  $W_I$ is equal to that in $W$.
	
$ (\text{P13})_{\geq N} $. Any left cell  $\Gamma\subset W_{\geq N}$ contains  a unique element $ z$ in  $\mathcal{D}$. And for such  $ z $, $ \Gamma $, and  any $y\in\Gamma$, we have $\gamma_{y^{-1},y,z}\neq0$.
	
$ (\text{P14})_{\geq N} $. For any $w\in W_{\geq N}$,  we have  $w\sim_{ {LR}}w^{-1}$.
	
$ (\text{P15})_{\geq N} $. For   $w,w'\in W$ and  $ x,y\in  W_{\geq N} $ such that  $ \af(x)=\af(y)$, we have$$
\sum_{z\in W}h_{w,x,z}\otimes h_{z,w',y}  =\sum_{z\in W} h_{w,z,y}\otimes h_{x,w',z} \in  \mc{A}\mathop{\otimes}_\mathbb{Z}\mc{A}.$$
\end{conj}

These conjectures are referred to as $ (\text{P1-P15})_{\geq N} $. Similarly, by replacing $\geq N $ by $   >N $ (resp. $  N  $) in  Conjecture \ref{conj15}, we get $ (\text{P1-P15})_{>N} $ (resp. $ (\text{P1-P15})_{N} $).
Since $ W_{\geq 0}=W $,  $  (\text{P1-P15})_{\geq 0}  $ is just Lusztig's conjectures P1-P15 from \cite[\S14.2]{lusztig2003}.

Assume that $ N $ is an integer such that $ W_{>N} $ is $ \prec_{ {LR}} $-closed. Then we can form  a quotient algebra $ \mc{H}_{\leq N} $ of $ \mc{H} $ by the subspace  $ \mc{H}_{>N} $ spanned by $ \{ C_w\mid w\in W_{>N} \} $ over $ \mc{A} $. For any $ w\in W $, the image of $ T_w $ (resp. $ C_w $) in $\hn  $ is denoted by $ \nt_w $ (resp. $ \nc_w $), and $ \{\nt_w\mid w\in\wn \} $ and $ \{\nc_w\mid w\in\wn \} $ form two $ \mc{A} $-basis of $ \hn $, see \cite[Lem. 3.2 and 3.3]{xie2019}. For $ h=\sum_{z\in\wn}b_z\nt_z \in\hn$, we define $$  \deg(h):= \max\{\deg b_{z}\mid z\in\wn \}.  $$ This gives rise to a function $ \deg:\hn\to \mathbb{N}\cup\{-\infty\}$. Note that $ \nc_z=0 $ and
$ \deg(\nt_z)<0 $  for $ z\in W_{>N} $.
For $ x,y,z\in W_{\leq N} $, let $ \nf_{x,y,z} $ be the element  of $ \mc{A} $ given by
\begin{equation*}\label{eq:fxyz}
\nt_x\nt_y =\sum_{z\in W_{\leq N} } \nf_{x,y,z}\nt_z.
\end{equation*}
By \cite[Lem. 3.4]{xie2019}, for any $ x,y\in \wn $, we always have\begin{equation}\label{eq:bound}
 \deg (\nt_x\nt_y)\leq N ,
\end{equation}
and $$  W_N=\{z\in \wn\mid \deg (\nf_{x,y,z})=N \text{ for some } x,y\in \wn \} .$$ If we further assume $ (\text{P1})_{>N} $, $ (\text{P4})_{>N} $, $ (\text{P8})_{>N} $ hold, then $$  W_N=\{ x\in\wn\mid \deg \nt_x\nt_y=N \text{ for some }y\in\wn \}.  $$
\section{Coxeter groups of dimension 2}\label{sec:3}

In this section, $ (W,S) $ is a Coxeter group such that all its finite parabolic subgroups have rank 1 or 2, and at least one of them has rank 2.
We say that this kind of Coxeter group has dimension 2,  because its Davis complex has dimension 2.  For example,   infinite Coxeter groups of rank 3 have dimension 2.

Let 
$$D=\{w_{J}\mid J\subseteq S, |W_{J}|<\infty\}\ \sqcup\ \{sw_{st}\mid s,t\in S, \infty>m_{st}\geq 4, L(t)>L(s)\}.$$
Let $ \af':D\to \mathbb{N} $ be the function given by
$$
\af'(w_J)=L(w_J) \ \mbox{ and }\ \af'(sw_{st})=L(t)+(\frac{m_{st}}{2}-1)(L(t)-L(s)).
$$
For $N\in \mathbb{N}$, we define
$$\begin{aligned}
D_{\geq N}&=\{d\in D \mid \af'(d)\geq N\},\quad D_N=D_{\geq N}\setminus D_{\geq N+1} ,\\
\Omega_{\geq N} &=\left\{x\cdot d\cdot y\mid x,y\in W, d\in D_{\geq N}
\right\},\\
\Omega_{ N}&=\Omega_{\geq N}\setminus \Omega_{\geq N+1},\quad \Omega_{<N}=W\setminus \Omega_{\geq N}.
\end{aligned}
$$
For $N\in \mathbb{N}$ and $d\in D_{N}$, we define
$$\begin{aligned}
U_d&=\{y\in W\mid d\cdot y\in \Omega_{N}\},\\
B_d&=\{x\in U_d^{-1}\mid \mbox{\ if\ }  x\cdot d=w\cdot v  \mbox{\ and\ }  v\neq e , \mbox{\ then\ }  w\in \Omega_{<N}\}.
\end{aligned}
$$

\begin{conj}\label{conj:dim2}
Assume that $ W $ is a Coxeter group of dimension 2, and $ N $ is an integer such that $ W_{>N}=\Omega_{>N} $ and $ W_{>N} $ is $ \prec_{ {LR}} $-closed. Then we have the following properties. \begin{itemize}
\item [(1)] The equality in \eqref{eq:bound} holds only if $ y\in\Omega_{\geq N} $.
\item [(2)] For any $ d\in D_N $, $ b\in B_d $, $ y\in U_d $, we have
\(	l(bdy)=l(b)+l(d)+l(y).\)
\item [(3)]	If  $ d\in D_N $, $ x\in U_d^{-1} $, $ y\in U_d $, $ w\leq d $, then
\begin{equation}\label{eq:leq}
	\deg(\nt_{xw}\nt_y) \leq-\deg p_{w,d}.
\end{equation}
If moreover $ x\in B_d $, $ w<d $, then
	\begin{equation}\label{eq:lin}
\deg(\nt_{xw}\nt_y) <-\deg p_{w,d}.
\end{equation}
\end{itemize}
\end{conj}

\begin{thm}\label{thm:main}
If Conjecture \ref{conj:dim2} and the boundedness conjecture \eqref{eq:bound-conj} hold for $ (W,S) $  and its all  parabolic subgroups, then  P1-P15  hold for $ (W,S) $.
\end{thm}
\begin{proof}
See  \cite[\S 9]{xie2019}.
\end{proof}

Thus for the aim of this paper it suffices to  prove Conjecture \ref{conj:dim2} for those Coxeter groups that are listed in section \ref{sec:intr}. We  first prepare some results in the following two sections.


\section{Dihedral groups}\label{sec:dih}
In this section, $ (W_I,I,L) $ is a weighted dihedral group with $ I=\{s,t\} $ and $ 2\leq m_{st}\leq \infty $.

If $ m_{st}\geq 4$ is even and $ L(s)\neq L(t) $, we set $I= \{s_1,s_2\}$ with $ L(s_1)>L(s_2) $,  $ d_I =s_{2}w_{I}$, $ d_I' =s_{1}w_{I}$, and define a (nonpositive) weight function  $ L':W_I\to\mathbb{Z} $ by $ L'(s_1)=L(s_1) $ and $ L'(s_2)=-L(s_2) $.

\subsection{Possible monomials in $ f_{u,v,w} $}\label{subsec:pterm}

For $ u,v,w\in W_I$, $ f_{u,v,w} $ is a polynomial of $ \xi_s $ and $ \xi_t $ with nonnegative coefficients. We view $ \xi_s $ and $ \xi_t $ as variables, and say $ \xi_s^m \xi_t^n$ appears in $ f_{u,v,w} $ if the coefficient of $ \xi_s^m \xi_t^n$ in $ f_{u,v,w} $ is nonzero. (It is possible that $ \xi_s=\xi_t $, but this does not affect the following statements.)

\begin{lem}\label{lem:fuvw} Assume $ \infty\neq m_{st}\geq 3 $.
For $ u,v\in W_I\setminus\{w_{I}\} $, consider possible monomials that appear in $ f_{u,v,sts} $. At least one of the following (maybe overlapped) situations happens:
\begin{itemize}
\item [(1)] $ f_{u,v,sts}=0 $ or has a nonzero constant term;
\item [(2)] $ \xi_s $ appears in $f_{u,v,sts}   $, and   $(u,v)$ is one of the following pairs:
\begin{itemize}
	\item $ (s\cdot u', u'^{-1}\cdot sts) $ for some $ u'\in W_{I} $;
    \item $  (tw_I, w_Is) $;
	\item $ (sts\cdot u',u'^{-1}\cdot s) $ for some $ u'\in W_{I} $;
	\item $  (sw_I, w_It) $;
\end{itemize}
\item [(3)] $ \xi_t $ appears in $f_{u,v,sts}   $, and   $(u,v)$ is one of the following pairs:
\begin{itemize}
	\item $( st\cdot u' , u'^{-1}\cdot ts)$ for some $ u'\in W_{I} $;
    \item $(sw_I, w_I s)$.
\end{itemize}
\end{itemize}
\end{lem}

\begin{proof}
Assume first $ \xi_s\neq \xi_t $.
We consider the product $ T_{sts}T_u $ since $ f_{u,v,sts}=f_{sts,u, v^{-1}} $. If  $ \xi_s^2\xi_t $ or $ \xi_s\xi_t $ or $ \xi_s^2 $ appears in $f_{u,v,w}   $, it is easy to see $u=w_{I}$ or $v=w_{I}$, which contradicts with the assumption $ u,v\neq w_I $.

Suppose $ \xi_s $ appears in $f_{u,v,sts}=f_{sts,u, v^{-1}}$ and consider the product $ T_{sts}T_u $. If the third factor of $ sts $ gives  $ \xi_s $, then $\mathcal{L}(u)=\{s\}$, $ stu=v^{-1} $. Thus,
\begin{itemize}
\item if $ t\cdot u<w_I $, then $(u,v)=(s\cdot u', u'^{-1}\cdot sts) $ for some $ u'\in W_{I} $;
\item if $ t\cdot u=w_I $, then $(u,v)=(tw_I, w_Is) $.
\end{itemize}
Similarly if the first factor of $ sts $ gives $ \xi_s $, then $ s\in\mc{L}(tsu) $ and $ tsu=v^{-1} $. Since $ u,v\neq w_I $, we have two cases:
\begin{itemize}
\item if $\mathcal{L}(u)=\{s\}$, then $(u,v)=( sts\cdot u' , u'^{-1}\cdot s)$ for some $ u'\in W_{I} $;
\item if $\mathcal{L}(u)=\{t\}$, then $(u,v)=(sw_I, w_It) $ for some $ u'\in W_{I} $.
\end{itemize}

Suppose $ \xi_t $ appears in $f_{u,v,sts} =f_{sts,u, v^{-1}}$. Then we have $ t\in\mc{L}(su) $ and $ u=v^{-1} $.
 If $\mathcal{L}(u)=\{s\}$, then $u=st\cdot u'=v^{-1}$ for some $ u'\in W_{I}$. If $\mathcal{L}(u)=\{t\}$,
then  $ su=w_I $, and  $ u=sw_I=v^{-1}$.

Now the lemma follows for $\xi_s \neq \xi_t $. But for  $\xi_s = \xi_t $ similar arguments shows that the same statement is valid: at least one of (1)(2)(3) happens.
\end{proof}

Similarly, we have following three lemmas.

\begin{lem}\label{lem:infinite2}
Assume that $m_{st}=\infty$. For $ u,v\in W_I $, consider possible monomials that appear in $ f_{u,v,sts} $.  We must be in (at least) one of the following situations.
\begin{itemize}
\item[(1)] $ f_{u,v,sts}=0 $ or has a nonzero constant term;
\item [(2)]$ \xi_s $ appears in $f_{u,v,sts}   $, $su<u$ and $vs<v$;
\item [(3)]$ \xi_t $ appears in $f_{u,v,sts}   $, and $(u,v)=( st\cdot u' , u'^{-1}\cdot ts)$ for some $u'\in W_{I}$.
\end{itemize}
\end{lem}

\begin{lem}\label{lem:fuvw2}
 Assume $ \infty\neq m_{st}\geq 2 $.
For $ u,v\in W_I $, consider possible monomials that appear in $ f_{u,v,st} $. We must be in one of the following situations:
\begin{itemize}
\item[(1)] $ f_{u,v,st}=0 $ or has a nonzero constant term;
\item [(2)] $ \xi_s\xi_t $ appears in $f_{u,v,st}   $, and   $(u,v)=(w_I, w_I) $;
\item [(3)] $ \xi_s $ appears in $f_{u,v,st}   $, and  $(u,v)$ is one of the following pairs:
\begin{itemize}
\item $ (u,v)=(s\cdot u', u'^{-1}\cdot st)$,
\item $ (u,v)=(w_I, w_It) $;
\end{itemize}
\item [(4)] $ \xi_t $ appears in $f_{u,v,st}   $, and  $(u,v)$ is one of the following pairs:
\begin{itemize}
\item $ (u,v)=(st\cdot v', v'^{-1}\cdot t)$,
\item $ (u,v)=(sw_I, w_I) $.
\end{itemize}
\end{itemize}
\end{lem}

\begin{lem}\label{lem:infinite1}
Assume that $m_{st}=\infty$. For $ u,v\in W_I $, consider possible monomials that appear in $ f_{u,v,st} $.  We must be in one of the following situations.
\begin{itemize}
\item[(1)] $ f_{u,v,st}=0 $ or has a nonzero constant term;
\item [(2)] $ \xi_s $ appears in $f_{u,v,st}   $, and $(u,v)=(s\cdot u', u'^{-1}\cdot st)$ for some $u'\in W_{I}$;
\item [(3)] $ \xi_t $ appears in $f_{u,v,st}   $, and $(u,v)=(st\cdot u', u'^{-1}\cdot t)$ for some $u'\in W_{I}$.
\end{itemize}
\end{lem}

\subsection{Possible degrees of $\delta$}\label{subsec:delta}
In this subsection, $ 3\leq m_{st}<\infty $.

\begin{lem}\label{lem:plusdeg}
Let $ u,v\in W_I $. If   $ f_{u,v,w_I}\neq 0 $, then $ \deg f_{u,v,w_I}=L(u)+L(v)-L(w_I)  $.
\end{lem}
\begin{proof}
See \cite[Lem.4.6]{xie2019}.
\end{proof}

\begin{lem}\label{lem:degfp}
Assume that $ u,v\in W_I\setminus\{w_I \} $. For  $ \delta=\deg f_{u,v,w_I}p_{sts,w_I} $, we must be in one of the following situations.
\begin{itemize}
\item [(1)] $ \delta\leq 0 $;
\item [(2)] $ L(s)=L(t) $, $ \delta=L(s) $, and $ l(u)=l(v)=m_{st}-1 $;
\item[(3)] $ L(s)\neq L(t) $, $ \delta=L(t) $,  and $ u=v=sw_{I}$;
\item [(4)]$ L(s)\neq L(t) $, $ \delta=L(s) $,  and $ \{u,v\}=\{d_I,d'_I\} $;
\item [(5)]$ L(s)\neq L(t) $, $ \delta=2L(s)-L(t) >0$,  and $ u=v=tw_{I}$;
\item [(6)] $ L(s)>L(t) $, $ \delta=L(s)-L(t) $,  and $ u=d_I,$  $ L(v)=L(w_I)-L(st) $;
\item [(7)] $ L(s)>L(t) $, $ \delta=L(s)-2L(t) $, and $ u=d_I,$  $ L(v)=L(w_I)-L(tst) $.
\end{itemize}
\end{lem}
\begin{proof}
Assume $ \delta> 0$. By Lemma \ref{lem:plusdeg}, we have $$  \delta=2L(s)+L(t)-(L(w_I)-L(u))-(L(w_I)-L(v)) .$$

Assume $ L(s)=L(t) $. Then the possible values of $ L(w_I)-L(u) $ are $ L(s)$, $2L(s)$, $3L(s),\cdots $. Thus $ \delta=L(s) $ and $ l(u)=l(v)=l(w_I)-1 $.

Assume $ L(t)>L(s) $. Then the possible values of $ L(w_I)-L(u) $ (resp. $ L(w_I)-L(v) $) are
\[
L(s),L(t),L(s)+L(t),2L(s)+L(t),L(s)+2L(t),\cdots
\]
Then
$\delta =L(t),L(s),2L(s)-L(t),$
and we are in one of the following situations:
\begin{itemize}
\item $ \delta=L(t) $,  and $ u=v=d_I $;
\item $ \delta=L(s) $,  and $ \{u,v\}=\{d'_I,d_I\} $;
\item $ \delta=2L(s)-L(t) >0$,  and $ u=v=d'_I $.
\end{itemize}

Similarly, if $ L(s)>L(t)$, then we are in one of the following situations:
\begin{itemize}
\item $ \delta=2L(s)-L(t) $,  and $ u=v=d_I $;
\item $ \delta=L(s) $,  and $ \{u,v\}=\{d'_I,d_I\} $;
\item $ \delta=L(t)$,  and $ u=v=d'_I $;
\item  $ \delta=L(s)-L(t) $, and $ u=d_I,$  $ L(v)=L(w_I)-L(st) $;
\item  $ \delta=L(s)-2L(t) $, and $ u=d_I,$  $ L(v)=L(w_I)-L(tst) $.
\end{itemize}
This completes the proof.
\end{proof}

\begin{cor}\label{cor:degfp2}
Assume that $ u,v\in W_I\setminus\{w_I \} $. For  $ \delta=\deg f_{u,v,w_I}p_{sts,w_I} $, we must be in one of the following situations.
\begin{itemize}
\item [(1)] $ \delta\leq 0 $;
\item [(2)] $ su<u $ or $ vs<v $, and $ \delta< 2L(s) $;
\item [(3)] $ tu<u $ and $ vt<v $, and $ \delta=L(t) $.
\end{itemize}
\end{cor}

Similarly, one can prove the following lemma, see also \cite[Lem.4.7]{xie2019}.
\begin{lem}\label{lem:degfp2}
Assume that $ u,v\in W_I\setminus\{w_I \} $. For  $ \delta=\deg f_{u,v,w_I}p_{st,w_I} $, we must be in one of the following situations.
\begin{itemize}
\item [(1)] $ \delta\leq 0 $;
\item [(2)]  $ L(s)\neq L(t) $, $ \delta=|L(s)-L(t)| $,  and $ u=v=d_I$.
\end{itemize}
\end{lem}

\subsection{Possible degrees of $\gamma$}\label{subsec:gamma}

In this subsection, $ 3\leq m_{st}<\infty $.

\begin{lem}\label{lem:aaa}
Assume $L(s)\neq L(t)$. For any $w\leq d_{I}$, we have
$$\deg p_{w,d_{I}}=L'(w)-L'(d_{I}).$$
\end{lem}
\begin{proof}
See \cite[Lem.4.4]{xie2019}.
\end{proof}

\begin{lem}\label{lem:Fuv}
Assume $ L(t)>L(s) $.
Let $ u,v\in W_I\setminus\{w_I\} $, and write\[
F(u,v)=f_{u,v,d_{I}}-p_{d_{I},w_I}f_{u,v,w_I}.
\]
Then \begin{itemize}
\item if  $ vs<v $, then $ F(u,v)=-q^{-L(s)}F(u,vs) $;
\item if  $ su<u $, then $ F(u,v)=-q^{-L(s)}F(su,v) $;
\item if $ su>u $ and $ vs>v $,  then
 \begin{itemize}
\item if $ l(u)+l(v)< 2m-1 $, $ F(u,v)=0 $,
\item if $ l(u)+l(v)=2m-1 $, $ F(u,v)=1 $,
\item if $ l(u)+l(v)=2m $, $$  F(u,v)=  \begin{cases}
\xi_s &\text{ if } l(u),l(v) \text{ are even},\\
\xi_t &\text{ if } l(u),l(v) \text{ are odd},
\end{cases}$$
\item if $ l(u)+l(v)>2m $, then $ \deg F(u,v)=L'(u)+L'(v)-L'(d_I) $.
\end{itemize}
\end{itemize}
\end{lem}
\begin{proof}
See \cite[Lem.4.8]{xie2019}.
\end{proof}

\begin{cor}\label{cor:degnfp}
Assume that $ L(s)\neq L(t) $ and  $ u,v,sts\in W_I\setminus\{w_I,d_I\} $. For  $ \gamma:= \deg (f_{u,v,d_{I}}-f_{u,v,w_I}p_{d_{I},w_I})p_{sts,d_{I}}$, we must be in one of the following situations.

\begin{itemize}
\item [(1)]$ \gamma\leq 0$.
 \item [(2)]$ L(s)>L(t) $, $ su<u $, $ vs<v $, and   $\gamma \leq  L(t)$.
\end{itemize}
\end{cor}
\begin{proof}
If $ L(t)>L(s) $, then by Lemma \ref{lem:Fuv}, we have $ \gamma\leq L(t)-2L(s)-2(L(t)-L(s))=-L(t) $.

If $L(s)>L(t)$, then $ \gamma\leq 2L(s)-L(t)-2(L(s)-L(t))=L(t) $. Note that $sts\neq d_I $ implies that $ m_{st}\geq 6 $. By Lemma \ref{lem:Fuv},  if $ tu<u $ or $ vt<v $, then we have $ \gamma\leq 0 $.
Thus, if $ \gamma>0 $, then we have $ su<u $ and $ vs<v $. Note that by Lemma \ref{lem:Fuv} we have $ \gamma\leq 0 $ if $ u=e $ or $ v=e $.
\end{proof}

Similarly, we have the following corollary, see also \cite[Lem.4.10]{xie2019}.
\begin{cor}\label{cor:degnfp2}
Assume that $ L(s)\neq L(t) $ and  $ u,v,st\in W_I\setminus\{w_I,d_I\} $. For  $ \gamma:= \deg (f_{u,v,d_{I}}-f_{u,v,w_I}p_{d_{I},w_I})p_{st,d_{I}}$, we always have
$ \gamma\leq 0$.
\end{cor}

\begin{cor}\label{cor:bbb}
Let $ u,v\in W_I\setminus\{w_I\} $ and $w\in W$ with $l(w)\geq 2$. Then we have
\begin{itemize}
\item [(1)] $\deg f_{u,v,w}<L(w)$.
\item [(2)] $\deg f_{u,v,w_{I}}p_{w,w_{I}}<L(w)$.
\end{itemize}
Moreover, if $ L(s)\neq L(t) $, then we have
\begin{itemize}
\item [(3)] $\deg (f_{u,v,d_{I}}-f_{u,v,w_I}p_{d_{I},w_I})p_{w,d_{I}}<L(w)$.
\end{itemize}
\end{cor}
\begin{proof}
(1) First we have $\deg f_{u,v,w}\leq L(w)$. If the equality holds, since $f_{u,v,w}=f_{w^{-1},u,v^{-1}}$ and $l(w)\geq 2$, we must have $u=w_{I}$, a contradiction.
\\(2) If $f_{u,v,w_{I}}=0$, there is nothing to prove. If $f_{u,v,w_{I}}\neq 0$, by Lemma \ref{lem:plusdeg}, we have
$$\deg f_{u,v,w_{I}}p_{w,w_{I}}=(L(u)+L(v)-L(w_{I}))+(L(w)-L(w_{I}))<L(w).$$
(3) If $w\nleq d_{I}$, it is obvious. If $w\leq d_{I}$, by Lemma \ref{lem:aaa} and Lemma \ref{lem:Fuv}, we have
$$\deg (f_{u,v,d_{I}}-f_{u,v,w_I}p_{d_{I},w_I})p_{w,d_{I}}\leq L'(d_{I})+(L'(w)-L'(d_{I}))=L'(w)<L(w).$$
\end{proof}

\section{Expansions of some products}\label{sec:exp}

In sections \ref{sec:exp}, \ref{sec:take}, \ref{sec:add}, \ref{sec:est}, we assume $(W,S)$ is a  Coxeter group of rank 3 from section \ref{sec:intr}. We have $S=\{r,s,t\}$ and $m_{rt}=2$.

\begin{deff}
	For $x,w,y,x',w',y'\in W $, we call $ (x',w',y') $ the \textbf{transpose} of $( x,w,y )$ if \[
	x'=y^{-1}, w'=w^{-1}, y'=x^{-1}.
	\]
\end{deff}

\begin{deff}
	For $x,w,y,x',y'\in W $, we say $(x',w,y')$ is a \textbf{reduced extension} of $(x,w,y)$  if there exists $u,v\in W$ such that
\begin{itemize}
		\item [(1)] $x'=u\cdot x$,
		\item [(2)] $y'=y\cdot v$,
       \item [(3)] $l(uzv)=l(u)+l(z)+l(v)$ for any $z\in W$ such that $f_{x,y,z}\neq 0$.
\end{itemize}
\end{deff}
\subsection{The case of $\infty=m_{rs}>m_{st}\geq3$.}

\begin{lem}\label{lem:a0}
Let $w,x,y\in W$.
\begin{itemize}
		\item [(1)] There is no $w_{1},w_{2}\in W$ such that $w=w_{1}\cdot r=w_{2}\cdot s$.
		\item [(2)] If $w=w_{1}\cdot st$, then $r\notin \mathcal{R}(w)$.
        \item [(3)] If $w=w_{1}\cdot rs$, then $\mathcal{R}(w)=\{s\}$.
        \item [(4)] If $w\in W_{rs}$, $l(w)\geq 4$, $\mathcal{R}(x),\ \mathcal{L}(y)\subseteq\{t\}$, then $l(xwy)=l(x)+l(w)+l(y)$.
        \item [(5)] If $\mathcal{R}(x),\ \mathcal{L}(y)\subseteq\{s\}$, then $l(xrty)=l(x)+l(y)+2$.
        \item [(6)] If $w\in W_{st}$, $l(w)\geq 2$, $\mathcal{R}(x),\ \mathcal{L}(y)\subseteq\{r\}$, then $l(xwy)=l(x)+l(w)+l(y)$.
\end{itemize}
\end{lem}

\begin{proof}
See  \cite[3.1, 3.2]{gao}.
\end{proof}

\begin{lem}\label{lem:reduced0}
	Assume that $ I\subset S $, $ |I|=2 $, $ w\in W_I $, $ l(w)\geq 2 $ and $ \mc{R}(x)\cup\mc{L}(y)\subset S\setminus I $.
	
	If $l(xwy)<l(x)+l(w)+l(y)$, then  $ (x,w,y) $ or its transpose is in the following cases.
	
	\begin{itemize}
		\item [(1)] $ w=srs $,  $x=x'\cdot w_{st} s$, $y=sw_{st}\cdot y'$ for some $x',y'\in W$ with $\mathcal{R}(x'),\mathcal{L}(y')\subseteq\{r\}$. In this case we have
		$$
		T_{x}T_{srs}T_{y}=\xi_{t}T_{x'\cdot w_{st}\cdot r\cdot tw_{st}\cdot y'}+T_{x'\cdot w_{st}t\cdot r\cdot tw_{st}\cdot y'}.
		$$
		\item [(2)] $ w=rs $, $x=x'\cdot t$, $y=sw_{st}\cdot y'$ for some $x',y'\in W$ with $\mathcal{R}(x')\subseteq\{s\}$, $\mathcal{L}(y')\subseteq \{r\}$. We have
		$$
		T_{x}T_{rs}T_{y}=\xi_{t}T_{x'\cdot r\cdot w_{st}\cdot y'}+T_{x'\cdot r\cdot tw_{st}\cdot y'}.
		$$
	\end{itemize}
\end{lem}
\begin{proof}
Since $l(xwy)<l(x)+l(w)+l(y)$, by Lemma \ref{lem:a0}(4)(5)(6), we must have $I=\{r,s\}$ and $2\leq l(w)\leq 3$.

If $w=srs$, we claim that $\mathcal{R}(xs)=\{s,t\}$. Otherwise, $\mathcal{R}(xs)=\{s\}$, and then $\mathcal{R}(xsr)=\{r\}$.
We assume $y=y_{1}\cdot y_{2}$ for some $y_{1}\in W_{st}\setminus\{e\}$ and $y_{2}\in W$ with $\mathcal{L}(y_{2})\subseteq\{r\}$. Since $s\cdot y_{1}\in W_{st}$ and $l(sy_{1})\geq 2$, by Lemma \ref{lem:a0}(6), we have
$$l(x)+l(srs)+l(y)=l(xsr)+l(sy_{1})+l(y_{2})=l(xsrsy).$$
It is a contradiction. Similarly, we can prove $\mathcal{L}(sy)=\{s,t\}$. Now we assume $x=x'\cdot w_{st}s$, $y=sw_{st}\cdot y'$ for some $x',y'\in W$ with $\mathcal{R}(x'),\mathcal{L}(y')\subseteq\{r\}$.
Since $\mathcal{L}(r\cdot tw_{st}\cdot y')=\{r\}$, by Lemma \ref{lem:a0}(6), we have
$$
\begin{aligned}
T_{x}T_{srs}T_{y}&=T_{x'w_{st}}T_{r}T_{w_{st}y'}\\
&=\xi_{t}T_{x'}T_{w_{st}}T_{r\cdot tw_{st}\cdot y'}+T_{x'}T_{w_{st}t}T_{r\cdot tw_{st}\cdot y'}\\
&=\xi_{t}T_{x'\cdot w_{st}\cdot r\cdot tw_{st}\cdot y'}+T_{x'\cdot w_{st}t\cdot r\cdot tw_{st}\cdot y'}.
\end{aligned}
$$

If $w=rsr$, we assume $x=x'\cdot t$ for some $x'\in W$ with $\mathcal{R}(x')\subseteq\{s\}$. By Lemma \ref{lem:a0}(3), we know $\mathcal{L}(sr\cdot y)=\{s\}$, so we have
$$l(x)+l(rsr)+l(y)=l(x')+l(rt)+l(sry)=l(xrsry).$$

At last, we consider $w=rs$ and the case of $w=sr$ is similar. We assume $x=x'\cdot t$ for some $x'\in W$ with $\mathcal{R}(x')\subseteq\{s\}$.
If $\mathcal{L}(sy)=\{s\}$, then $l(x)+l(rs)+l(y)=l(x')+l(rt)+l(sy)=l(xrsy)$ by Lemma \ref{lem:a0}(5). So we must have $\mathcal{L}(sy)=\{s,t\}$. We assume $y=sw_{st}\cdot y'$ for some $y'\in W$ with $\mathcal{L}(y')\subseteq \{r\}$.
Since $\mathcal{R}(x'r)=\{r\}$, by Lemma \ref{lem:a0}(6), we have
$$
\begin{aligned}
T_{x}T_{rs}T_{y}&=T_{x'}T_{rt}T_{w_{st}y'}\\
&=\xi_{t}T_{x'r}T_{w_{st}y'}+T_{x'r}T_{tw_{st}y'}\\
&=\xi_{t}T_{x'\cdot r\cdot w_{st}\cdot y'}+T_{x'\cdot r\cdot tw_{st}\cdot y'}.
\end{aligned}
$$
This completes the proof.
\end{proof}

\subsection{The case of $ \infty> m_{rs}, m_{st}\geq 4$ but $(m_{rs},m_{st})\neq (4,4)$}

\begin{lem}\label{lem:a1}
Let $w,x,y\in W$.
\begin{itemize}
\item [(1)] If $w=w_{1}\cdot ts$, then $r\notin \mathcal{R}(w)$.
\item [(2)] If $w=w_{1}\cdot rs$, then $t\notin \mathcal{R}(w)$.
\item [(3)] If $w=w_{1}\cdot st,\ \mathcal{R}(w_{1}s)=\{s\}$, then $r\notin \mathcal{R}(w)$.
\item [(4)] If $w=w_{1}\cdot sr,\ \mathcal{R}(w_{1}s)=\{s\}$, then $t\notin \mathcal{R}(w)$.
\item [(5)] If $w=w_{1}\cdot tst$, then $r\notin \mathcal{R}(w)$.
\item [(6)] If $w=w_{1}\cdot rsr$, then $t\notin \mathcal{R}(w)$.
\item [(7)] There is no $w_{1},w_{2}\in W$ such that $w=w_{1}\cdot st=w_{2}\cdot sr$.
\item [(8)] If $\mathcal{L}(w)\subseteq\{r\}$, then $\mathcal{L}(r\cdot tw_{st}\cdot w)=\{r\}$.
\item [(9)] If $\mathcal{L}(w)\subseteq\{t\}$, then $\mathcal{L}(t\cdot rw_{rs}\cdot w)=\{t\}$.
\item [(10)] If $w\in W_{st},\ l(w)\geq 4$, $\mathcal{R}(x),\mathcal{L}(y)\subseteq\{r\}$,
then $l(xwy)=l(x)+l(w)+l(y)$.
\item [(11)] If $w\in W_{rs},\ l(w)\geq 4$, $\mathcal{R}(x),\mathcal{L}(y)\subseteq\{t\}$,
then $l(xwy)=l(x)+l(w)+l(y)$.
\item [(12)] If $\mathcal{R}(x),\mathcal{L}(y)\subseteq\{s\},\ \mathcal{R}(xt)=\{t\},\ \mathcal{R}(xr)=\{r\},$
then $l(xtry)=l(x)+l(y)+2$.
\item [(13)] If $\mathcal{R}(x),\mathcal{L}(y)\subseteq \{r\},\ \mathcal{R}(xs)=\{s\}$ or $\mathcal{L}(sy)=\{s\}$,
 then $l(xstsy)=l(x)+l(y)+3$.
\item [(14)] If $\mathcal{R}(x),\mathcal{L}(y)\subseteq \{t\},\ \mathcal{R}(xs)=\{s\}$ or $\mathcal{L}(sy)=\{s\}$,
 then $l(xsrsy)=l(x)+l(y)+3$.
\item [(15)] If $m_{rs}\geq 5$, $\mathcal{R}(x),\mathcal{L}(y)\subseteq\{r\}$, then $l(xtsty)=l(x)+l(y)+3$.
\item [(16)] If $m_{st}\geq 5$, $\mathcal{R}(x),\mathcal{L}(y)\subseteq\{t\}$, then $l(xrsry)=l(x)+l(y)+3$.
\end{itemize}
Note that the lemmas are still correct after exchanging $r$ and $t$.
\end{lem}

\begin{proof}
See \cite[4.1, 4.2]{gao}.
\end{proof}

\begin{lem}\label{lem:reduced}
	Assume that $ I\subset S $, $ |I|=2 $, $ w\in W_I $, $ l(w)\geq 2 $ and $ \mc{R}(x)\cup\mc{L}(y)\subset S\setminus I $.
	
	If $ l(xwy)<l(x)+l(w)+l(y) $, then  $ (x,w,y) $, or its transpose, or the one with $ r $, $ t $ exchanged,  is in one of the following cases:
	
	\begin{itemize}
		\item [(1)] $ w=srs $,  $ x=x'\cdot w_{st} s$, $ y=sw_{st}\cdot y' $ for some $x',y'\in W$ with $\mathcal{R}(x'),\mathcal{L}(y')\subseteq\{r\}$. We have
		$$
		T_{x}T_{srs}T_{y}=\xi_{t}T_{x'\cdot w_{st}\cdot r\cdot tw_{st}\cdot y'}+T_{x'\cdot w_{st}t\cdot r\cdot tw_{st}\cdot y'}.
		$$
		\item [(2)] $ w=rsr $, $ m_{st}=4 $, $x=x''\cdot w_{rs}r\cdot t$ and $y=t\cdot rw_{rs}\cdot y''$ for some $x'',y''\in W$ with $\mathcal{R}(x''),\mathcal{L}(y'')\subseteq \{t\}$. We have
		$$
		T_{x}T_{rsr}T_{y}=\xi_{s}T_{x''\cdot w_{rs}\cdot tst\cdot sw_{rs}y''}+T_{x''w_{rs}s\cdot tst\cdot sw_{rs}\cdot y''}.
		$$
		\item [(3)] $ w=rt $, $x=x'\cdot w_{rs}r$, $y=tw_{st}\cdot y'$ for some $x',y'\in W$ with $\mathcal{R}(x')\subseteq\{t\}$, $\mathcal{L}(y')\subseteq\{r\}$. We have
		$$
		T_{x}T_{rt}T_{y}=\xi_{s}T_{x'\cdot w_{rs}\cdot sw_{st}\cdot y'}+T_{x'\cdot w_{rs}s\cdot sw_{st}\cdot y'}.
		$$
        \item [(4)] $ w=rs $.
        \begin{enumerate}[label=\circled{\arabic*}]
			\item $x=x''\cdot w_{rs}r\cdot t$, $y=sw_{st}\cdot y'$ for some $x'',y'\in W$ with $\mathcal{R}(x'')\subseteq\{t\}$, $\mathcal{L}(y')\subseteq\{r\}$, $\mathcal{L}(stw_{st}y')=\{t\}$. We have
			\begin{align*}
			T_{x}T_{rs}T_{y}&=\xi_{t}\xi_{s}T_{x''\cdot w_{rs}s\cdot w_{st}\cdot y'}+\xi_{t}T_{x''\cdot w_{rs}s\cdot sw_{st}\cdot y'}\\
            &\ \ \ +\xi_{s}T_{x''\cdot w_{rs}\cdot stw_{st}\cdot y'}+T_{x''\cdot w_{rs}s\cdot stw_{st}\cdot y'}.
			\end{align*}
			\item $m_{st}=4$, $x=x''\cdot w_{rs}r\cdot t$, $y=sw_{st}\cdot sw_{rs}\cdot y''$ for some $x'',y''\in W$ with $\mathcal{R}(x'')\subseteq\{t\}$, $\mathcal{L}(y'')\subseteq\{t\}$. We have
			\begin{align*}
			T_{x}T_{rs}T_{y}&=\xi_{t}\xi_{s}T_{x''\cdot w_{rs}s\cdot w_{st}\cdot sw_{rs}\cdot y''}+\xi_{t}T_{x''\cdot w_{rs}s\cdot sw_{st}\cdot sw_{rs}\cdot y''}\\
            &\ \ \ +\xi_{s}\xi_{r}T_{x''\cdot w_{rs}r\cdot t\cdot w_{rs}\cdot y''}+\xi_{s}T_{x''\cdot w_{rs}r\cdot t\cdot rw_{rs}\cdot y''}\\
            &\ \ \ +\xi_{r}T_{x''\cdot w_{rs}sr\cdot t\cdot w_{rs}\cdot y''}+T_{x''\cdot w_{rs}sr\cdot t\cdot rw_{rs}\cdot y''}.
			\end{align*}
            \item $x=x'\cdot t$, $y=sw_{st}\cdot y'$ for some $x',y'\in W$ with $\mathcal{R}(x')\subseteq\{s\}$, $\mathcal{R}(x'r)=\{r\}$, $\mathcal{L}(y')\subseteq\{r\}$.
            At least one of $m_{st}\geq 5$, $\mathcal{R}(x'rs)=\{s\}$ and $\mathcal{L}(sy')=\{s\}$ holds. We have
			$$
			T_{x}T_{rs}T_{y}=\xi_{t}T_{x'\cdot r\cdot w_{st}\cdot y'}+T_{x'\cdot r\cdot tw_{st}\cdot y'}.
			$$
            \item $m_{st}=4$, $x=x''\cdot w_{rs}sr\cdot t$, $y=sw_{st}\cdot sw_{rs}\cdot y''$ for some $x'',y''\in W$ with $\mathcal{R}(x'')\subseteq\{t\}$, $\mathcal{L}(y'')\subseteq\{t\}$. We have
			$$
		    T_{x}T_{rs}T_{y}=\xi_{t}T_{x''\cdot w_{rs}s\cdot w_{st}\cdot sw_{rs}\cdot y''}+\xi_{r}T_{x''\cdot w_{rs}r\cdot t\cdot w_{rs}\cdot y''}+T_{x''\cdot w_{rs}r\cdot t\cdot rw_{rs}\cdot y''}.
		    $$
            \item $ x=x''\cdot (w_{rs} r)\cdot t$, $ y=stw_{st}\cdot y'' $ for some $x'',y''\in W$ with $\mathcal{R}(x'')\subseteq\{t\}$, $\mathcal{L}(y'')\subseteq\{r\}$. We have
			$$
		    T_{x}T_{rs}T_{y}=\xi_{s}T_{x''\cdot w_{rs}\cdot sw_{st}\cdot y''}+T_{x''\cdot w_{rs}s\cdot sw_{st}\cdot y''}.
		    $$
        \end{enumerate}
	\end{itemize}
\end{lem}

\begin{proof}
If $l(w)\geq 4$, then $I=\{s,r\}$ or $I=\{s,t\}$. By Lemma \ref{lem:a1}(10)(11), we have $l(xwy)=l(x)+l(w)+l(y)$.

If $l(w)=3$, we may assume $I=\{s,r\}$ because the case of $I=\{s,t\}$ is similar. Firstly we consider $w=srs$. Since $\mathcal{R}(x),\mathcal{L}(y)\subseteq\{t\}$, by Lemma \ref{lem:a1}(1),
we get $\mathcal{R}(xs)=\{s\}$ or $\{s,t\}$, $\mathcal{L}(sy)=\{s\}$ or $\{s,t\}$. If $\mathcal{R}(xs)=\{s\}$ or $\mathcal{L}(sy)=\{s\}$, then we have $l(xwy)=l(x)+l(w)+l(y)$ by Lemma \ref{lem:a1}(14).
If $\mathcal{R}(xs)=\{s,t\}$ and $\mathcal{L}(sy)=\{s,t\}$, we assume $x=x'\cdot w_{st}s$ and $y=sw_{st}\cdot y'$ for some $x',y'\in W$ with $\mathcal{R}(x'),\mathcal{L}(y')\subseteq\{r\}$.
Then by Lemma \ref{lem:a1}(8)(10), we have
$$
\begin{aligned}
T_{x}T_{srs}T_{y}&=T_{x'w_{st}}T_{r}T_{w_{st}y'}\\
&=\xi_{t}T_{x'\cdot w_{st}}T_{r\cdot tw_{st}\cdot y'}+T_{x'\cdot w_{st}t}T_{r\cdot tw_{st}\cdot y'}.
\end{aligned}
$$
By Lemma \ref{lem:a1}(7)(8), we have $\mathcal{L}(r\cdot tw_{st}\cdot y')=\{r\}$ and $\mathcal{L}(sr\cdot tw_{st}\cdot y')=\{s\}$. Then by Lemma \ref{lem:a1}(10)(13), we get
$$
T_{x}T_{srs}T_{y}=\xi_{t}T_{x'\cdot w_{st}\cdot r\cdot tw_{st}\cdot y'}+T_{x'\cdot w_{st}t\cdot r\cdot tw_{st}\cdot y'}.
$$
Secondly we consider $w=rsr$. If $m_{st}\geq 5$, we have $l(xwy)=l(x)+l(w)+l(y)$ by Lemma \ref{lem:a1}(16).  If $m_{st}=4$, then $m_{rs}\geq 5$. We assume $y=t\cdot y'$ for some $y'\in W$ with $\mathcal{L}(y')\subseteq\{s\}$.
By Lemma \ref{lem:a1}(2)(6), we know $\mathcal{R}(xrs)=\{s\}$ and $\mathcal{R}(xrsr)=\{r\}$. Since $l(x)+l(w)+l(y)=l(xrs)+l(rt)+l(y')<l(xrsrty')$, by Lemma \ref{lem:a1}(12), we must have $\mathcal{R}(xrst)=\{s,t\}$.
We assume $xrst=x'\cdot w_{st}$ for some $x'\in W$ with $\mathcal{R}(x')\subseteq\{r\}$, then $xt\cdot r=x'\cdot s$, so we have $\mathcal{R}(xt\cdot r)=\{s,r\}$. Similarly, we can prove $\mathcal{R}(r\cdot ty)=\{s,r\}$.
Now we assume $x=x''\cdot w_{rs}r\cdot t$ and $y=t\cdot rw_{rs}\cdot y''$ for some $x'',y''\in W$ with $\mathcal{R}(x''),\mathcal{L}(y'')\subseteq \{t\}$.
Since $\mathcal{L}(tst\cdot sw_{rs}\cdot y'') =\{t\}$, by Lemma \ref{lem:a1}(11), we have
$$
\begin{aligned}
T_{x}T_{rsr}T_{y}&=T_{x''w_{rs}}T_{tst}T_{w_{rs}y''}\\
&=\xi_{s}T_{x''w_{rs}}T_{tst}T_{sw_{rs}y''}+T_{x''w_{rs}s}T_{tst}T_{sw_{rs}y''}\\
&=\xi_{s}T_{x''\cdot w_{rs}\cdot tst\cdot sw_{rs}\cdot y''}+T_{x''w_{rs}s\cdot tst\cdot sw_{rs}\cdot y''}.
\end{aligned}
$$

Now we consider $w=rt$. If $\mathcal{R}(xr)=\{r\}$ and $\mathcal{R}(xt)=\{t\}$, we have $l(xwy)=l(x)+l(w)+l(y)$ by Lemma \ref{lem:a1}(12). If $\mathcal{R}(xr)=\{r,s\}$ and $\mathcal{L}(ty)=\{t\}$, we have $l(xwy)=l(x)+l(w)+l(y)$ by Lemma \ref{lem:a1}(11). If $\mathcal{R}(xt)=\{s,t\}$ and $\mathcal{L}(ry)=\{r\}$, we have $l(xwy)=l(x)+l(w)+l(y)$ by Lemma \ref{lem:a1}(10). Summarizing the arguments above, we must have $\mathcal{R}(xr)=\{r,s\}$, $\mathcal{L}(ty)=\{s,t\}$ or $\mathcal{R}(xt)=\{s,t\}$, $\mathcal{L}(ry)=\{r,s\}$. We only consider the former case and assume $x=x'\cdot w_{rs}r$, $y=tw_{st}\cdot y'$ for some $x',y'\in W$ with $\mathcal{R}(x')\subseteq\{t\}$, $\mathcal{L}(y')\subseteq\{r\}$. Then by Lemma \ref{lem:a1}(10)(11), we have
$$
\begin{aligned}
T_{x}T_{rt}T_{y}&=T_{x'w_{rs}}T_{w_{st}y'}\\
&=\xi_{s}T_{x'w_{rs}}T_{sw_{st}y'}+T_{x'w_{rs}s}T_{sw_{st}y'}\\
&=\xi_{s}T_{x'\cdot w_{rs}\cdot sw_{st}\cdot y'}+T_{x'\cdot w_{rs}s\cdot sw_{st}\cdot y'}.
\end{aligned}
$$

At last, we consider $w=rs$ and the case of $w=sr$ is similar. Since $l(xwy)<l(x)+l(w)+l(y)$, we have $l(x)\geq 1$, so we may assume $x=x'\cdot t$ for some $x'\in W$ with $\mathcal{R}(x')\subseteq\{s\}$. We get
$$T_{x}T_{rs}T_{y}=T_{x'\cdot rt}T_{sy}.$$

 If $\mathcal{L}(sy)=\{s,t\}$, we assume $y=sw_{st}\cdot y'$ for some $y'\in W$ with $\mathcal{L}(y')\subseteq\{r\}$. Then we have
$$
\begin{aligned}
T_{x}T_{rs}T_{y}&=T_{x'\cdot rt}T_{w_{st}\cdot y'}\\
&=\xi_{t}T_{x'\cdot r}T_{w_{st}\cdot y'}+T_{x'\cdot r}T_{tw_{st}\cdot y'}.
\end{aligned}
$$
We consider the following 4 cases.

\noindent\circled{1} $\mathcal{R}(x'r)=\{r,s\}$, $\mathcal{L}(stw_{st}y')=\{t\}$.

Then we assume $x'=x''\cdot w_{rs}r$ for some $x''\in W$ with $\mathcal{R}(x'')\subseteq\{t\}$. By Lemma \ref{lem:a1}(10)(11)(16), we have
$$
\begin{aligned}
T_{x}T_{rs}T_{y}&=\xi_{t}T_{x''\cdot w_{rs}}T_{w_{st}\cdot y'}+T_{x''\cdot w_{rs}}T_{tw_{st}\cdot y'}\\
&=\xi_{t}\xi_{s}T_{x''\cdot w_{rs}s\cdot w_{st}\cdot y'}+\xi_{t}T_{x''\cdot w_{rs}s\cdot sw_{st}\cdot y'}\\
&\ \ \ +\xi_{s}T_{x''\cdot w_{rs}\cdot stw_{st}\cdot y'}+T_{x''\cdot w_{rs}s\cdot stw_{st}\cdot y'}.
\end{aligned}
$$

\noindent\circled{2} $\mathcal{R}(x'r)=\{r,s\}$, $\mathcal{L}(stw_{st}y')=\{r,t\}$.

Then $m_{st}=4$ and $\mathcal{L}(sy')=\{r,s\}$. We assume $x'=x''\cdot w_{rs}r$, $y'=sw_{rs}\cdot y''$ for some $x'',y''\in W$ with $\mathcal{R}(x''),\mathcal{L}(y'')\subseteq\{t\}$.
By Lemma \ref{lem:a1}(10)(11), we have
$$
\begin{aligned}
T_{x}T_{rs}T_{y}&=\xi_{t}T_{x''\cdot w_{rs}}T_{w_{st}\cdot y'}+T_{x''\cdot w_{rs}}T_{st\cdot w_{rs}\cdot y''}\\
&=\xi_{t}\xi_{s}T_{x''\cdot w_{rs}s\cdot w_{st}\cdot y'}+\xi_{t}T_{x''\cdot w_{rs}s\cdot sw_{st}\cdot y'}\\
&\ \ \ +\xi_{s}T_{x''\cdot w_{rs}}T_{t\cdot w_{rs}\cdot y''}+T_{x''\cdot w_{rs}s}T_{t\cdot w_{rs}\cdot y''}\\
&=\xi_{t}\xi_{s}T_{x''\cdot w_{rs}s\cdot w_{st}\cdot y'}+\xi_{t}T_{x''\cdot w_{rs}s\cdot sw_{st}\cdot y'}\\
&\ \ \ +\xi_{s}\xi_{r}T_{x''\cdot w_{rs}r\cdot t\cdot w_{rs}\cdot y''}+\xi_{s}T_{x''\cdot w_{rs}r\cdot t\cdot rw_{rs}\cdot y''}\\
&\ \ \ +\xi_{r}T_{x''\cdot w_{rs}sr\cdot t\cdot w_{rs}\cdot y''}+T_{x''\cdot w_{rs}sr\cdot t\cdot rw_{rs}\cdot y''}.
\end{aligned}
$$

\noindent
\circled{3} $\mathcal{R}(x'r)=\{r\}$. At least one of $m_{st}\geq 5$, $\mathcal{R}(x'rs)=\{s\}$ and $\mathcal{L}(sy')=\{s\}$ holds.

Since $\mathcal{L}(y')\subseteq\{r\}$, by Lemma \ref{lem:a1}(10)(13), we have
$$
\begin{aligned}
T_{x}T_{rs}T_{y}&=\xi_{t}T_{x'\cdot r}T_{w_{st}\cdot y'}+T_{x'\cdot r}T_{tw_{st}\cdot y'}\\
&=\xi_{t}T_{x'\cdot r\cdot w_{st}\cdot y'}+T_{x'\cdot r\cdot tw_{st}\cdot y'}.
\end{aligned}
$$

\noindent
\circled{4} $\mathcal{R}(x'r)=\{r\}$, $m_{st}=4$, $\mathcal{R}(x'rs)=\mathcal{L}(sy')=\{r,s\}$.

We assume $x'=x''\cdot w_{rs}sr$, $y'=sw_{rs}\cdot y''$ for some $x'',y''\in W$ with $\mathcal{R}(x''),\mathcal{L}(y'')\subseteq\{t\}$.
Since $\mathcal{R}(x''\cdot w_{rs}r\cdot t)=\{t\}$, by Lemma \ref{lem:a1}(11), we have
$$
\begin{aligned}
T_{x}T_{rs}T_{y}&=\xi_{t}T_{x'\cdot r}T_{w_{st}\cdot y'}+T_{x'\cdot r}T_{tw_{st}\cdot y'}\\
&=\xi_{t}T_{x'\cdot r\cdot w_{st}\cdot y'}+T_{x''\cdot w_{rs}}T_{t\cdot w_{rs}\cdot y''}\\
&=\xi_{t}T_{x''\cdot w_{rs}s\cdot w_{st}\cdot sw_{rs}\cdot y''}+\xi_{r}T_{x''\cdot w_{rs}r\cdot t\cdot w_{rs}\cdot y''}+T_{x''\cdot w_{rs}r\cdot t\cdot rw_{rs}\cdot y''}.
\end{aligned}
$$
If $\mathcal{L}(sy)=\{s\}$, since $\mathcal{R}(x')\subseteq\{s\}$ and $\mathcal{R}(x't)=\{t\}$, by the discussion for $w=rt$, we must have
\\\circled{5} $\mathcal{R}(x'r)=\{r,s\}$ and $\mathcal{L}(tsy)=\{s,t\}$. Now we assume $x=x''\cdot (w_{rs} r)\cdot t$, $y=stw_{st}\cdot y''$ for some $x'',y''\in W$ with $\mathcal{R}(x'')\subseteq\{t\}$, $\mathcal{L}(y'')\subseteq\{r\}$.
Since $\mathcal{R}(x''\cdot w_{rs}s)=\{r\}$, $\mathcal{L}(sw_{st}\cdot y'')=\{t\}$, by Lemma \ref{lem:a1}(10)(11), we have
$$
\begin{aligned}
T_{x}T_{rs}T_{y}&=T_{x''w_{rs}}T_{w_{st}y''}\\
&=\xi_{s}T_{x''w_{rs}}T_{sw_{st}y''}+T_{x''w_{rs}s}T_{sw_{st}y''}\\
&=\xi_{s}T_{x''\cdot w_{rs}\cdot sw_{st}\cdot y''}+T_{x''\cdot w_{rs}s\cdot sw_{st}\cdot y''}.
\end{aligned}
$$
This completes the proof.\end{proof}

\subsection{The case of $ \infty>m_{rs}\geq 7, m_{st}=3$ }

\begin{lem}\label{lem:a2}
Let $w,x,y\in W$.
\begin{itemize}
\item [(1)] There is no $w_{1},w_{2}\in W$ such that $w=w_{1}\cdot st=w_{2}\cdot sr$.
\item [(2)] If $w=w_{1}\cdot srs$, then $t\notin \mathcal{R}(w)$.
\item [(3)] If $w=w_{1}\cdot srsr$, then $t\notin \mathcal{R}(w)$.
\item [(4)] If $w=w_{1}\cdot ts$, then $r\notin \mathcal{R}(w)$.
\item [(5)] If $w=w_{1}\cdot tsr$, then $s\notin \mathcal{R}(w)$.
\item [(6)] If $\mathcal{R}(x),\mathcal{L}(y)\subseteq\{t\}$, $w\in W_{rs}$, $l(w)\geq 6$ or $w=srsrs$, then $l(xwy)=l(x)+l(w)+l(y)$, $\mathcal{R}(xwy)=\mathcal{R}(wy)$, $\mathcal{L}(xwy)=\mathcal{L}(xw)$.
\item [(7)] If $\mathcal{R}(x),\mathcal{L}(y)\subseteq\{r\}$, $\mathcal{R}(xs)=\{s\}$ or $\mathcal{L}(sy)=\{s\}$, then $T_{xsts}T_{y}=T_{xstsy}$.
\item [(8)] If $\mathcal{R}(x),\mathcal{L}(y)\subseteq \{s\}$, $\mathcal{R}(xr)=\{r\}$, $\mathcal{R}(xt)=\{t\}$, $\mathcal{R}(xrs)=\{s\}$, then $T_{xtr}T_{y}=T_{xtry}$.
\item [(9)] If $\mathcal{R}(x)\subseteq \{s\}$, $\mathcal{L}(y)\subseteq \{r\}$, then $\deg T_{x\cdot rt}T_{sts\cdot y}\leq L(rsrs)$.
\end{itemize}
\end{lem}

\begin{proof}
See \cite[5.1, 5.2, 5.3, 5.4, 5.8]{gao}.
\end{proof}

\begin{lem}\label{lem:reduced2}
	Assume that $ I\subset S $, $ |I|=2 $, $ w\in W_I $, $ l(w)\geq 2 $ and $ \mc{R}(x)\cup\mc{L}(y)\subset S\setminus I $.
	
	If $ l(xwy)<l(x)+l(w)+l(y) $, then  $ (x,w,y) $ or its transpose is a reduced extension of that $ (x,w,y) $ in the following cases.
	\begin{itemize}
		\item [(1)] $ w=rsrsr $, $ x=w_{rs}r\cdot t $, $ y=t\cdot rw_{rs}$. In this case, we have
		\begin{align*}
		T_{x}T_{rsrsr}T_{y}=\xi_{t}T_{w_{rs}\cdot tsrst\cdot sw_{rs}}+T_{w_{rs}s\cdot tsrst\cdot sw_{rs}}.
		\end{align*}
		\item [(2)] $ w=rsrs$, $ x=w_{rs}r\cdot t $, $ y=ts$. In this case, we have
		\begin{align*}
		T_{x}T_{rsrs}T_{y}=\xi_{s}T_{w_{rs}\cdot tsrst}+T_{w_{rs}s\cdot tsrst}.
		\end{align*}
		\item [(3)]$ w=srs $.
		\begin{enumerate}[label=\circled{\arabic*}]
			\item $ x=st $, $ y=ts $. In this case, we have
			\begin{align*}
			T_{x}T_{srs}T_{y}=\xi_{t}T_{tstrst}+T_{tsrst}.
			\end{align*}
			\item $ m_{rs}=8 $,  $ x=w_{rs}s\cdot tsrst $, $ y=tsrst\cdot sw_{rs}$. In this case, we have
			\begin{align*}
			T_{x}T_{srs}T_{y}=\xi_{t}T_{w_{rs}s\cdot tsrtstrstrst\cdot sw_{rs}}+\xi_{r}T_{w_{rs}r\cdot t\cdot w_{rs}\cdot t\cdot rw_{rs}}+T_{w_{rs}r\cdot t\cdot rw_{rs}\cdot t\cdot rw_{rs}}.
			\end{align*}			
				\item  $ m_{rs}=7$, $ x=w_{rs}s\cdot tsrst $, $ y=tsrst\cdot sw_{rs}$. In this case, we have
				\begin{align*}
				T_{x}T_{srs}T_{y}&=\xi_{t}T_{w_{rs}s\cdot tsrtstrstrst\cdot sw_{rs}}+\xi_{r}^{2}T_{w_{rs}r\cdot t\cdot w_{rs}\cdot t\cdot rw_{rs}}+\xi_{r}T_{w_{rs}r\cdot t\cdot w_{rs}r\cdot t\cdot rw_{rs}}\\
                     &\ \ \ +\xi_{r}T_{w_{rs}r\cdot t\cdot rw_{rs}\cdot t\cdot rw_{rs}}+T_{w_{rs}r\cdot t\cdot rw_{rs}r\cdot t\cdot rw_{rs}}.
				\end{align*}
			\item 	$ m_{rs}=7$, $ x=w_{rs}s\cdot tsrst $, $ y=tsrst$. In this case, we have
				\begin{align*}
				T_{x}T_{srs}T_{y}=\xi_{t}T_{w_{rs}s\cdot tsrtstrstrst}+\xi_{r}T_{w_{rs}r\cdot t\cdot w_{rs}\cdot ts}+T_{w_{rs}r\cdot t\cdot rw_{rs}\cdot ts}.
				\end{align*}
			\end{enumerate}

		\item [(4)] $ w=rsr $.
		\begin{enumerate}[label=\circled{\arabic*}]
            \item $ x=w_{rs}sr\cdot t $, $ y=t\cdot rsw_{rs} $. In this case, we have
			\begin{align*}
			T_{x}T_{rsr}T_{y}=\xi_{r}T_{w_{rs}\cdot t\cdot rw_{rs}}+T_{w_{rs}r\cdot t\cdot rw_{rs}}.
			\end{align*}
            \item $ x=w_{rs}r\cdot t $, $ y=t\cdot rw_{rs}$. In this case, we have
			\begin{align*}
			T_{x}T_{rsr}T_{y}&=\xi_{s}^{2}\xi_{r}T_{w_{rs}r\cdot t\cdot w_{rs}}+\xi_{s}^{2}T_{w_{rs}r\cdot t\cdot rw_{rs}}+\xi_{s}\xi_{r}T_{w_{rs}r\cdot t\cdot sw_{rs}}\\
            &\ \ \ +\xi_{s}T_{w_{rs}r\cdot t\cdot rsw_{rs}}+\xi_{s}\xi_{r}T_{w_{rs}s\cdot t\cdot rw_{rs}}+\xi_{s}T_{w_{rs}sr\cdot t\cdot rw_{rs}}\\
            &\ \ \ +\xi_{r}T_{w_{rs}sr\cdot t\cdot sw_{rs}}+T_{w_{rs}sr\cdot t\cdot rsw_{rs}}.
			\end{align*}
            \item $ x=w_{rs}sr\cdot t $,  $ y=t\cdot rw_{rs} $. In this case, we have
			\begin{align*}
			T_{x}T_{rsr}T_{y}&=\xi_{s}\xi_{r}T_{w_{rs}\cdot t\cdot rw_{rs}}+\xi_{s}T_{w_{rs}r\cdot t\cdot rw_{rs}}\\
            &\ \ \ +\xi_{r}T_{w_{rs}r\cdot t\cdot sw_{rs}}+T_{w_{rs}r\cdot t\cdot rsw_{rs}}.
			\end{align*}
			\item $ x=t $, $ y=t\cdot rw_{rs} $. In this case, we have
			\begin{align*}
			T_{x}T_{rsr}T_{y}=\xi_{t}T_{rst\cdot w_{rs}}+T_{rst\cdot sw_{rs}}.
			\end{align*}

		\end{enumerate}
		\item [(5)] $ w=sts $, $ x=w_{rs}s $, $ y=sw_{rs} $. In this case, we have
			\begin{align*}
			T_{x}T_{sts}T_{y}=\xi_{r}T_{w_{rs}\cdot t\cdot rw_{rs}}+T_{w_{rs}r\cdot t\cdot rw_{rs}}.
			\end{align*}
		 \item [(6)] $ w=rt $.
        \begin{enumerate}[label=\circled{\arabic*}]
			\item $ x=w_{rs}r $, $ y=st $. In this case, we have
			\begin{align*}
			T_{x}T_{rt}T_{y}=\xi_{s}T_{w_{rs}\cdot ts}+T_{w_{rs}s\cdot ts}.
			\end{align*}
			\item $ x=w_{rs}r $, $ y=st\cdot sw_{rs} $. In this case, we have
			\begin{align*}
			T_{x}T_{rt}T_{y}=\xi_{s}\xi_{r}T_{w_{rs}r\cdot t\cdot w_{rs}}+\xi_{s}T_{w_{rs}r\cdot t\cdot rw_{rs}}+\xi_{r}T_{w_{rs}sr\cdot t\cdot w_{rs}}+T_{w_{rs}sr\cdot t\cdot rw_{rs}}.
			\end{align*}
            \item $ x=w_{rs}sr$, $ y=st\cdot sw_{rs} $. In this case, we have
			\begin{align*}
			T_{x}T_{rt}T_{y}=\xi_{r}T_{w_{rs}\cdot t\cdot rw_{rs}}+T_{w_{rs}r\cdot t\cdot rw_{rs}}.
			\end{align*}
		\end{enumerate}
		
        \item [(7)] $ w=st $.
\begin{enumerate}[label=\circled{\arabic*}]
			\item $ x=w_{rs}rs $, $ y=rst $. In this case, we have
			\begin{align*}
			T_{x}T_{st}T_{y}=\xi_{s}T_{w_{rs}\cdot ts}+T_{w_{rs}s\cdot ts}.
			\end{align*}
			\item $ x=w_{rs}rs $, $ y=rst\cdot sw_{rs} $. In this case, we have
			\begin{align*}
			T_{x}T_{st}T_{y}=\xi_{s}\xi_{r}T_{w_{rs}r\cdot t\cdot w_{rs}}+\xi_{s}T_{w_{rs}sr\cdot t\cdot rw_{rs}}+\xi_{r}T_{w_{rs}sr\cdot t\cdot w_{rs}}+T_{w_{rs}sr\cdot t\cdot rw_{rs}}.
			\end{align*}
            \item $ x=w_{rs}srs$, $ y=rst\cdot sw_{rs} $. In this case, we have
			\begin{align*}
			T_{x}T_{st}T_{y}=\xi_{r}T_{w_{rs}\cdot t\cdot rw_{rs}}+T_{w_{rs}r\cdot t\cdot rw_{rs}}.
			\end{align*}
             \item $ x=w_{rs}s$, $ y=r$. In this case, we have
			\begin{align*}
			T_{x}T_{st}T_{y}=\xi_{r}T_{w_{rs}\cdot t}+T_{w_{rs}r\cdot t}.
			\end{align*}
             \item $ x=w_{rs}s$, $ y=rst$. In this case, we have
			\begin{align*}
			T_{x}T_{st}T_{y}=\xi_{r}\xi_{s}T_{w_{rs}\cdot ts}+\xi_{r}T_{w_{rs}s\cdot ts}+\xi_{s}T_{w_{rs}r\cdot ts}+T_{w_{rs}rs\cdot ts}.
			\end{align*}
             \item $m_{rs}=7$, $ x=w_{rs}r\cdot t\cdot w_{rs}s$, $ y=rst\cdot srw_{rs}$. In this case, we have
			\begin{align*}
			T_{x}T_{st}T_{y}&=\xi_{r}\xi_{s}T_{w_{rs}r\cdot t\cdot w_{rs}\cdot t\cdot rw_{rs}}+\xi_{r}T_{w_{rs}r\cdot t\cdot w_{rs}s\cdot t\cdot rw_{rs}}+\xi_{s}T_{w_{rs}r\cdot t\cdot w_{rs}r\cdot t\cdot rw_{rs}}\\
&\ \ \ +\xi_{s}T_{w_{rs}s\cdot tsrst\cdot w_{rs}}+T_{w_{rs}s\cdot tsrst\cdot sw_{rs}}.
			\end{align*}
            \item $ x=w_{rs}s$, $ y=rst\cdot sw_{rs}$. In this case, we have
			\begin{align*}
			T_{x}T_{st}T_{y}&=\xi_{r}^{2}\xi_{s}T_{w_{rs}r\cdot t\cdot w_{rs}}+\xi_{r}\xi_{s}T_{w_{rs}r\cdot t\cdot rw_{rs}}\\
&\ \ \ +\xi_{r}^{2}T_{w_{rs}sr\cdot t\cdot w_{rs}}+\xi_{r}T_{w_{rs}sr\cdot t\cdot rw_{rs}}\\
&\ \ \ +\xi_{s}T_{w_{rs}r\cdot t\cdot w_{rs}}+\xi_{r}T_{w_{rs}rs\cdot t\cdot rw_{rs}}+T_{w_{rs}rsr\cdot t\cdot rw_{rs}}.
			\end{align*}
            \item $m_{rs}=7$, $ x=w_{rs}r\cdot t\cdot w_{rs}s$, $ y=rst\cdot sw_{rs}$. In this case, we have
			\begin{align*}
			T_{x}T_{st}T_{y}&=\xi_{r}^{2}\xi_{s}T_{w_{rs}r\cdot t\cdot w_{rs}r\cdot t\cdot w_{rs}}+\xi_{r}\xi_{s}T_{w_{rs}r\cdot t\cdot w_{rs}r\cdot t\cdot rw_{rs}}\\
&\ \ \ +\xi_{r}^{2}T_{w_{rs}r\cdot t\cdot w_{rs}sr\cdot t\cdot w_{rs}}+\xi_{r}T_{w_{rs}r\cdot t\cdot w_{rs}sr\cdot t\cdot rw_{rs}}\\
&\ \ \ +\xi_{s}T_{w_{rs}r\cdot t\cdot w_{rs}r\cdot t\cdot w_{rs}}+\xi_{r}\xi_{s}T_{w_{rs}s\cdot tsrst \cdot w_{rs}}+\xi_{r}T_{w_{rs}s\cdot tsrst \cdot sw_{rs}}\\
&\ \ \ +\xi_{s}T_{w_{rs}s\cdot tsrst \cdot rw_{rs}}+T_{w_{rs}s\cdot tsrst \cdot srw_{rs}}.
			\end{align*}
		\end{enumerate}
\item [(8)] $ w=rs $.
		\begin{enumerate}[label=\circled{\arabic*}]
      \item  $ x=w_{rs}r\cdot t $, $ y=t $. In this case, we have
			\begin{align*}
			T_{x}T_{rs}T_{y}=\xi_{s}T_{w_{rs}\cdot ts}+T_{w_{rs}s\cdot ts}.
			\end{align*}
      \item  $ x=t $, $ y=ts $. In this case, we have
			\begin{align*}
			T_{x}T_{rs}T_{y}=\xi_{t}T_{rsts}+T_{rst}.
			\end{align*}
      \item $ x=w_{rs}rsr\cdot t $, $ y=tsrst$. In this case, we have
			\begin{align*}
			T_{x}T_{rs}T_{y}=\xi_{t}T_{w_{rs}r\cdot tsrst}+\xi_{s}T_{w_{rs}\cdot ts}+T_{w_{rs}s\cdot ts}.
			\end{align*}
      \item $ x=w_{rs}rsr\cdot t $, $ y=tsrst\cdot sw_{rs} $. In this case, we have
			\begin{align*}
			T_{x}T_{rs}T_{y}&=\xi_{t}T_{w_{rs}r\cdot tsrst\cdot sw_{rs}}+\xi_{s}\xi_{r}T_{w_{rs}r\cdot t\cdot w_{rs}}+\xi_{s}T_{w_{rs}sr\cdot t\cdot rw_{rs}}\\
            &\ \ \ +\xi_{r}T_{w_{rs}sr\cdot t\cdot w_{rs}}+T_{w_{rs}sr\cdot t\cdot rw_{rs}}.
			\end{align*}
      \item $ x=w_{rs}srsr\cdot t$, $ y=tsrst\cdot sw_{rs} $. In this case, we have
			\begin{align*}
			T_{x}T_{rs}T_{y}=\xi_{t}T_{w_{rs}sr\cdot tsrst\cdot sw_{rs}}+\xi_{r}T_{w_{rs}\cdot t\cdot rw_{rs}}+T_{w_{rs}r\cdot t\cdot rw_{rs}}.
			\end{align*}
      \item $ x=w_{rs}sr\cdot t $, $ y=tsr $. In this case, we have
			\begin{align*}
			T_{x}T_{rs}T_{y}=\xi_{t}T_{w_{rs}\cdot tsr}+\xi_{r}T_{w_{rs}\cdot t}+T_{w_{rs}r\cdot t}.
			\end{align*}
      \item $ x=w_{rs}sr\cdot t $, $ y=tsrst $. In this case, we have
			\begin{align*}
			T_{x}T_{rs}T_{y}&=\xi_{t}T_{w_{rs}\cdot tsrst}+\xi_{r}\xi_{s}T_{w_{rs}\cdot ts}+\xi_{r}T_{w_{rs}s\cdot ts}\\
            &\ \ \ +\xi_{s}T_{w_{rs}r\cdot ts}+T_{w_{rs}rs\cdot ts}.
			\end{align*}
      \item $m_{rs}=7$, $ x=w_{rs}r\cdot t\cdot w_{rs}sr\cdot t $, $ y=tsrst\cdot srw_{rs} $. In this case, we have
			\begin{align*}
			T_{x}T_{rs}T_{y}&=\xi_{t}T_{w_{rs}r\cdot t\cdot w_{rs}\cdot tsrst\cdot srw_{rs}}+\xi_{r}\xi_{s}T_{w_{rs}r\cdot t\cdot w_{rs}\cdot t\cdot rw_{rs}}+\xi_{r}T_{w_{rs}r\cdot t\cdot w_{rs}s\cdot t\cdot rw_{rs}}\\
&\ \ \ +\xi_{s}T_{w_{rs}r\cdot t\cdot w_{rs}r\cdot t\cdot rw_{rs}}+\xi_{s}T_{w_{rs}s\cdot tsrst\cdot w_{rs}}+T_{w_{rs}s\cdot tsrst\cdot sw_{rs}}.
			\end{align*}
      \item $ x=w_{rs}sr\cdot t $, $ y=tsrst\cdot sw_{rs} $. In this case, we have
			\begin{align*}
			T_{x}T_{rs}T_{y}&=\xi_{t}T_{w_{rs}\cdot tsrst\cdot sw_{rs}}+\xi_{r}^{2}\xi_{s}T_{w_{rs}r\cdot t\cdot w_{rs}}+\xi_{r}\xi_{s}T_{w_{rs}r\cdot t\cdot rw_{rs}}\\
&\ \ \ +\xi_{r}^{2}T_{w_{rs}sr\cdot t\cdot w_{rs}}+\xi_{r}T_{w_{rs}sr\cdot t\cdot rw_{rs}}+\xi_{s}T_{w_{rs}r\cdot t\cdot w_{rs}}\\
&\ \ \ +\xi_{r}T_{w_{rs}rs\cdot t\cdot rw_{rs}}+T_{w_{rs}rsr\cdot t\cdot rw_{rs}}.
			\end{align*}
      \item $m_{rs}=7$, $ x=w_{rs}r\cdot t\cdot w_{rs}sr\cdot t $, $ y=tsrst\cdot sw_{rs} $. In this case, we have
			\begin{align*}
			T_{x}T_{rs}T_{y}&=\xi_{t}T_{w_{rs}r\cdot t\cdot w_{rs}\cdot tsrst\cdot sw_{rs}}+\xi_{r}^{2}\xi_{s}T_{w_{rs}r\cdot t\cdot w_{rs}r\cdot t\cdot w_{rs}}+\xi_{r}\xi_{s}T_{w_{rs}r\cdot t\cdot w_{rs}r\cdot t\cdot rw_{rs}}\\
&\ \ \ +\xi_{r}^{2}T_{w_{rs}r\cdot t\cdot w_{rs}sr\cdot t\cdot w_{rs}}+\xi_{r}T_{w_{rs}r\cdot t\cdot w_{rs}sr\cdot t\cdot rw_{rs}}+\xi_{s}T_{w_{rs}r\cdot t\cdot w_{rs}r\cdot t\cdot w_{rs}}\\
&\ \ \ +\xi_{r}\xi_{s}T_{w_{rs}s\cdot tsrst \cdot w_{rs}}+\xi_{r}T_{w_{rs}s\cdot tsrst \cdot sw_{rs}}\\
&\ \ \ +\xi_{s}T_{w_{rs}s\cdot tsrst \cdot rw_{rs}}+T_{w_{rs}s\cdot tsrst \cdot srw_{rs}}.
			\end{align*}
      \item $ x=w_{rs}r\cdot t$, $ y=ts $. In this case, we have
      \begin{align*}
			T_{x}T_{rs}T_{y}=\xi_{t}\xi_{s}T_{w_{rs}\cdot ts}+\xi_{t}T_{w_{rs}s\cdot ts}+\xi_{s}T_{w_{rs}\cdot t}+T_{w_{rs}s\cdot t}.
			\end{align*}
      \item $ x=w_{rs}r\cdot t$, $ y=tsr $. In this case, we have
      \begin{align*}
			T_{x}T_{rs}T_{y}&=\xi_{t}\xi_{s}T_{w_{rs}\cdot tsr}+\xi_{t}T_{w_{rs}s\cdot tsr}+\xi_{s}\xi_{r}T_{w_{rs}\cdot t}\\
&\ \ \ +\xi_{s}T_{w_{rs}r\cdot t}+\xi_{r}T_{w_{rs}s\cdot t}+T_{w_{rs}sr\cdot t}.
			\end{align*}
       \item $ x=w_{rs}r\cdot t$, $ y=tsrst$. In this case, we have
      \begin{align*}
			T_{x}T_{rs}T_{y}&=\xi_{t}\xi_{s}T_{w_{rs}\cdot tsrst}+\xi_{t}T_{w_{rs}s\cdot tsrst}+\xi_{s}^{2}\xi_{r}T_{w_{rs}\cdot ts}\\
&\ \ \ +\xi_{s}\xi_{r}T_{w_{rs}s\cdot ts}+\xi_{s}^{2}T_{w_{rs}r\cdot ts}+\xi_{s}T_{w_{rs}rs\cdot ts}\\
&\ \ \ +\xi_{r}T_{w_{rs}\cdot ts}+\xi_{s}T_{w_{rs}sr\cdot ts}+T_{w_{rs}srs\cdot ts}.
	 \end{align*}
      \item $m_{rs}=8$, $ x=w_{rs}r\cdot t\cdot w_{rs}r\cdot t$, $ y=tsrst\cdot srw_{rs}$. In this case, we have
      \begin{align*}
			T_{x}T_{rs}T_{y}&=\xi_{t}\xi_{s}T_{w_{rs}r\cdot t\cdot w_{rs}\cdot tsrst\cdot srw_{rs}}+\xi_{t}T_{w_{rs}r\cdot t\cdot w_{rs}s\cdot tsrst\cdot srw_{rs}}+\xi_{s}^{2}\xi_{r}T_{w_{rs}r\cdot t\cdot w_{rs}\cdot t\cdot rw_{rs}}\\
&\ \ \ +\xi_{s}\xi_{r}T_{w_{rs}r\cdot t\cdot w_{rs}s\cdot t\cdot rw_{rs}}+\xi_{s}^{2}T_{w_{rs}r\cdot t\cdot w_{rs}r\cdot t\cdot rw_{rs}}+\xi_{s}T_{w_{rs}r\cdot t\cdot w_{rs}rs\cdot t\cdot rw_{rs}}\\
&\ \ \ +\xi_{r}T_{w_{rs}r\cdot t\cdot w_{rs}\cdot t\cdot rw_{rs}}+\xi_{s}T_{w_{rs}r\cdot t\cdot w_{rs}sr\cdot t\cdot rw_{rs}}\\
&\ \ \ +\xi_{t}T_{w_{rs}\cdot tsrst\cdot sw_{rs}}+T_{w_{rs}s\cdot tsrst\cdot sw_{rs}}.
	 \end{align*}
\item $m_{rs}=7$,  $ x=st\cdot w_{rs}r\cdot t$, $ y=tsrst\cdot srw_{rs}$. In this case, we have
      \begin{align*}
			T_{x}T_{rs}T_{y}&=\xi_{t}\xi_{s}T_{st\cdot w_{rs}\cdot tsrst\cdot srw_{rs}}+\xi_{t}T_{st\cdot w_{rs}s\cdot tsrst\cdot srw_{rs}}+\xi_{s}^{2}\xi_{r}T_{st\cdot w_{rs}\cdot t\cdot rw_{rs}}\\
&\ \ \ +\xi_{s}\xi_{r}T_{st\cdot w_{rs}s\cdot t\cdot rw_{rs}}+\xi_{s}^{2}T_{st\cdot w_{rs}r\cdot t\cdot rw_{rs}}+\xi_{s}T_{st\cdot w_{rs}rs\cdot t\cdot rw_{rs}}\\
&\ \ \ +\xi_{r}T_{st\cdot w_{rs}\cdot t\cdot rw_{rs}}+\xi_{s}T_{st\cdot w_{rs}sr\cdot t\cdot rw_{rs}}+\xi_{s}T_{tsrst\cdot w_{rs}}+T_{tsrst\cdot sw_{rs}}.
	 \end{align*}
\item $m_{rs}=7$, $ x=w_{rs}r\cdot t\cdot w_{rs}r\cdot t$, $ y=tsrst\cdot srw_{rs}$. In this case, we have
      \begin{align*}
			T_{x}T_{rs}T_{y}&=\xi_{t}\xi_{s}T_{w_{rs}r\cdot t\cdot w_{rs}\cdot tsrst\cdot srw_{rs}}+\xi_{t}T_{w_{rs}r\cdot t\cdot w_{rs}s\cdot tsrst\cdot srw_{rs}}+\xi_{s}^{2}\xi_{r}T_{w_{rs}r\cdot t\cdot w_{rs}\cdot t\cdot rw_{rs}}\\
&\ \ \ +\xi_{s}\xi_{r}T_{w_{rs}r\cdot t\cdot w_{rs}s\cdot t\cdot rw_{rs}}+\xi_{s}^{2}T_{w_{rs}r\cdot t\cdot w_{rs}r\cdot t\cdot rw_{rs}}+\xi_{s}\xi_{t}T_{w_{rs}\cdot tsrst\cdot sw_{rs}}\\
&\ \ \ +\xi_{s}T_{w_{rs}s\cdot tsrst\cdot sw_{rs}}+\xi_{r}T_{w_{rs}r\cdot t\cdot w_{rs}\cdot t\cdot rw_{rs}}+\xi_{s}T_{w_{rs}r\cdot t\cdot w_{rs}sr\cdot t\cdot rw_{rs}}\\
&\ \ \ +\xi_{s}T_{w_{rs}r\cdot tsrst\cdot sw_{rs}}+T_{w_{rs}rs\cdot tsrst\cdot sw_{rs}}.
	 \end{align*}
\item $x=x'\cdot w_{rs}r\cdot t$, $y=tsrst\cdot sw_{rs}\cdot y'$ for some $x',y'\in W$ with $\mathcal{R}(x'),\mathcal{L}(y')\subseteq\{t\}$. In this case, we have $\deg T_{x}T_{rs}T_{y}\leq L(rsrs)$.
     \end{enumerate}
	\end{itemize}
\end{lem}

\begin{proof}
If $l(w)\geq 6$ or $w=srsrs$, then we have $l(xwy)=l(x)+l(w)+l(y)$ by Lemma \ref{lem:a2}(6). We consider the following cases.

(1) If $w=rsrsr$, we assume $x=x'\cdot t$ and $y=t\cdot y'$ for some $x',y'\in W$ with $\mathcal{R}(x'),\mathcal{L}(y')\subseteq\{s\}$, then we have $T_{x}T_{rsrsr}T_{y}=T_{x'}T_{rt}T_{srsrty'}$.
Since $\mathcal{L}(rsrsrty')=\{r\}$, $\mathcal{L}(srsrsrty')=\{s\}$, by Lemma \ref{lem:a2}(8), we must have $\mathcal{L}(tsrsrty')=\{s,t\}$. Thus we get $\mathcal{L}(ry')=\{r,s\}$. Similarly, we can prove $\mathcal{R}(x'r)=\{r,s\}$.
Now we assume $x'=x''\cdot w_{rs}r$, $y'=rw_{rs}\cdot y''$ for some $x'',y''\in W$ with $\mathcal{R}(x''),\mathcal{L}(y'')\subseteq\{t\}$.
By Lemma \ref{lem:a2}(6), we have $\mathcal{L}(tsrst\cdot sw_{rs}\cdot y'')=\mathcal{L}(tsrst\cdot sw_{rs})=\{t\}$, and then
$$
\begin{aligned}
T_{x}T_{rsrsr}T_{y}&=T_{x''\cdot w_{rs}r\cdot t}T_{rsrsr}T_{t\cdot rw_{rs}\cdot y''}\\
&=T_{x''\cdot w_{rs}s\cdot tst}T_{r}T_{tst\cdot sw_{rs}\cdot y''}\\
&=\xi_{t}T_{x''\cdot w_{rs}\cdot tsrst\cdot sw_{rs}\cdot y''}+T_{x''\cdot w_{rs}s\cdot tsrst\cdot sw_{rs}\cdot y''}.
\end{aligned}
$$

(2) Now we consider $w=rsrs$. Since $ l(xwy)<l(x)+l(w)+l(y) $, we have $l(y)\geq 2$, so we assume $y=ts\cdot y'$ for some $y'\in W$ with $\mathcal{L}(y')\subseteq\{r\}$. Thus, we get $T_{x}T_{rsrs}T_{y}=T_{xrsr}T_{sts}T_{y'}$.
Since $\mathcal{L}(sy')=\{s\}$, by Lemma \ref{lem:a2}(7), we must have $\mathcal{R}(xrsr)=\{r,t\}$, so we have $x=x'\cdot w_{rs}r\cdot t$ for some $x'\in W$ with $\mathcal{R}(x')\subseteq\{t\}$.
Since $\mathcal{L}(tsrst\cdot y')=\{t\}$,
by Lemma \ref{lem:a2}(6), we have
$$
\begin{aligned}
T_{x}T_{rsrs}T_{y}&=T_{x'\cdot w_{rs}r\cdot t}T_{rsrs}T_{ts\cdot y'}\\
&=T_{x'\cdot w_{rs}}T_{stsrst\cdot y'}\\
&=\xi_{s}T_{x'\cdot w_{rs}\cdot tsrst\cdot y'}+T_{x'\cdot w_{rs}s\cdot tsrst\cdot y'}.
\end{aligned}
$$

(3) If $w=srs$, since $l(x)\geq 2$, $l(y)\geq 2$, we assume $x=x'\cdot st$, $y=ts\cdot y'$ for some $x',y'\in W$ with $\mathcal{R}(x'),\mathcal{L}(y')\subseteq\{r\}$. Then by Lemma \ref{lem:a2}(7), we have
$$
\begin{aligned}
T_{x}T_{srs}T_{y}&=T_{x'\cdot st}T_{srs}T_{ts\cdot y'}\\
&=T_{x'\cdot tst}T_{r}T_{tst\cdot y'}\\
&=\xi_{t}T_{x'\cdot tst\cdot rst\cdot y'}+T_{x'\cdot ts}T_{r}T_{st\cdot y'}.
\end{aligned}
$$
If $T_{x'\cdot ts}T_{r}T_{st\cdot y'}=T_{x'\cdot tsrst\cdot y'}$, then \\\circled{1} $(x,srs,y)$ is a reduced extension of $(st,sts,ts)$.

Now we consider when $deg\ (T_{x'\cdot ts}T_{r}T_{st\cdot y'})> 0$. We must have $l(x)\geq 4$ and $l(y)\geq 4$, so we assume $x=x''\cdot srst$, $y=tsrs\cdot y''$ for some $x'',y''\in W$. Then
$$
\begin{aligned}
T_{x}T_{srs}T_{y}&=\xi_{t}T_{x''\cdot srtst\cdot rstrs\cdot y''}+T_{x''\cdot srts}T_{r}T_{strs\cdot y''}\\
&=\xi_{t}T_{x''\cdot srtst\cdot rstrs\cdot y''}+T_{x''\cdot st}T_{rsrsr}T_{ts\cdot y''}.
\end{aligned}
$$
By the proof of the case $w=rsrsr$, we must have $\mathcal{R}(x''\cdot st)=\mathcal{L}(ts\cdot y'')=\{s,t\}$, so we assume $x=x'''\cdot tsrst$, $y=tsrst\cdot y'''$ for some $x''',y'''\in W$ with $\mathcal{R}(x'''),\mathcal{L}(y''')\subseteq\{r\}$. Then
$$
T_{x}T_{srs}T_{y}=\xi_{t}T_{x'''\cdot tsrtst\cdot rstrst\cdot y'''}+T_{x'''\cdot st}T_{srsrsrs}T_{ts\cdot y'''}.
$$
When $deg\ (T_{x'''\cdot st}T_{srsrsrs}T_{ts\cdot y'''})>0$, by Lemma \ref{lem:a2}(6), we must be in the following cases.
\\\circled{2} $m_{rs}=8$, $\mathcal{R}(x'''\cdot st)=\mathcal{L}(ts\cdot y''')=\{r,t\}$. In this case, $(x,srs,y)$ is a reduced extension of $(w_{rs}s\cdot tsrst,srs,tsrst\cdot sw_{rs})$.
\\\circled{3} $m_{rs}=7$, $\mathcal{R}(x'''\cdot st)=\mathcal{L}(ts\cdot y''')=\{r,t\}$. In this case, $(x,srs,y)$ is a reduced extension of $(w_{rs}s\cdot tsrst,srs,tsrst\cdot sw_{rs})$.
\\\circled{4} $m_{rs}=7$, $\mathcal{R}(x'''\cdot st)=\{r,t\}$, $\mathcal{L}(ts\cdot y''')=\{t\}$, or $\mathcal{R}(x'''\cdot st)=\{t\}$, $\mathcal{L}(ts\cdot y''')=\{r,t\}$. In this case, $(x,srs,y)$ or its transpose is a reduced extension of $(w_{rs}s\cdot tsrst,srs,tsrst)$.

(4) If $w=rsr$, since $l(x)\geq 1$, $l(y)\geq 1$, we assume $x=x'\cdot t$, $y=t\cdot y'$ for some $x',y'\in W$ with $\mathcal{R}(x'),\mathcal{L}(y')\subseteq\{s\}$. Then
$$T_{x}T_{rsr}T_{y}=T_{x'r}T_{sts}T_{ry'}.$$
First we assume $\mathcal{R}(x'r)=\mathcal{L}(ry')=\{r\}$. By Lemma \ref{lem:a2}(7), we have $\mathcal{R}(x'rs)=\mathcal{L}(sry')=\{r,s\}$ since $l(xwy)<l(x)+l(w)+l(y)$. We assume $x'=x''\cdot w_{rs}sr$,
$y'=rsw_{rs}\cdot y''$ for some $x'',y''\in W$ with $\mathcal{R}(x''),\mathcal{L}(y'')\subseteq\{t\}$. Then by Lemma \ref{lem:a2}(6),
\\\circled{1} $(x,rsr,y)$ is a reduced extension of $(w_{rs}sr\cdot t,rsr,t\cdot rsw_{rs})$, and
$$
\begin{aligned}
T_{x}T_{rsr}T_{y}&=T_{x''\cdot w_{rs}}T_{t}T_{w_{rs}\cdot y''}\\
&=\xi_{r}T_{x''\cdot w_{rs}\cdot t\cdot rw_{rs}\cdot y''}+T_{x''\cdot w_{rs}r\cdot t\cdot rw_{rs}\cdot y''}.
\end{aligned}
$$
If $\mathcal{R}(x'r)=\{r,s\}$ or $\mathcal{L}(ry')=\{r,s\}$, we only consider the latter case. We assume $y=t\cdot rw_{rs}\cdot y''$ for some $y''\in W$ with $\mathcal{L}(y'')\subseteq\{t\}$. Then we have
$$
\begin{aligned}
T_{x}T_{rsr}T_{y}&=T_{x'r}T_{sts}T_{w_{rs}\cdot y''}\\
&=\xi_{s}T_{x'r}T_{st}T_{w_{rs}\cdot y''}+T_{x'r}T_{st}T_{sw_{rs}\cdot y''}.
\end{aligned}
$$
By Lemma \ref{lem:a2}(6), we have the following 3 cases.
\\\circled{2} $\mathcal{R}(x'r)=\{r,s\}$. In this case, $(x,rsr,y)$ is a reduced extension of $(w_{rs}r\cdot t,rsr,t\cdot rw_{rs})$.
\\\circled{3} $\mathcal{R}(x'r)=\{r\}$, $\mathcal{R}(x'rs)=\{r,s\}$. In this case, $(x,rsr,y)$ is a reduced extension of $(w_{rs}sr\cdot t,rsr,t\cdot rw_{rs})$.
\\\circled{4} $\mathcal{R}(x'r)=\{r\}$, $\mathcal{R}(x'rs)=\{s\}$. In this case, $(x,rsr,y)$ is a reduced extension of $(t,rsr,t\cdot rw_{rs})$.
\\Then $T_{x}T_{rsr}T_{y}$ can be easily computed in all these cases.

(5) If $w=sts$, by Lemma \ref{lem:a2}(7), we must have $\mathcal{R}(xs)=\mathcal{L}(sy)=\{r,s\}$. We assume $x=x'\cdot w_{rs}s$, $y=sw_{rs}\cdot y'$ for some $x',y'\in W$ with $\mathcal{R}(x'),\mathcal{L}(y')\subseteq\{t\}$.
By Lemma \ref{lem:a2}(6), we have
$$
\begin{aligned}
T_{x}T_{sts}T_{y}&=T_{x'\cdot w_{rs}}T_{t}T_{w_{rs}\cdot y'}\\
&=\xi_{r}T_{x'\cdot w_{rs}\cdot t\cdot rw_{rs}\cdot y'}+T_{x'\cdot w_{rs}r\cdot t\cdot rw_{rs}\cdot y'}.
\end{aligned}
$$

(6) Now we consider the case of $w=rt$. By Lemma \ref{lem:a2}(8), we must have $\mathcal{R}(xr)=\{r,s\}$ or $\mathcal{R}(xt)=\{s,t\}$ or $\mathcal{R}(xrs)=\{r,s\}$.
\\(i) $\mathcal{R}(xr)=\{r,s\}$. We assume $x=x'\cdot w_{rs}r$ for some $x'\in W$ with $\mathcal{R}(x')\subseteq\{t\}$. If $\mathcal{L}(ty)=\{t\}$, then $T_{x}T_{rt}T_{y}=T_{xrty}$ by Lemma \ref{lem:a2}(6). If $\mathcal{L}(ty)=\{s,t\}$, we assume $y=st\cdot y'$ for some $y'\in W$ with $\mathcal{L}(y')\subseteq\{r\}$. Thus we have
$$
T_{x}T_{rt}T_{y}=\xi_{s}T_{x'\cdot w_{rs}}T_{ts\cdot y'}+T_{x'\cdot w_{rs}s}T_{ts\cdot y'}.
$$
If $\mathcal{L}(ts\cdot y')=\{t\}$, then by Lemma \ref{lem:a2}(6), \\\circled{1} $(x,rt,y)$ is a reduced extension of $(w_{rs}r,rt,st)$, and
$$
T_{x}T_{rt}T_{y}=\xi_{s}T_{x'\cdot w_{rs}\cdot ts\cdot y'}+T_{x'\cdot w_{rs}s\cdot ts\cdot y'}.
$$
If $\mathcal{L}(ts\cdot y')=\{r,t\}$, we assume $y'=sw_{rs}\cdot y''$ for some $y''\in W$ with $\mathcal{L}(y'')\subseteq\{t\}$. Then by Lemma \ref{lem:a2}(6), \\\circled{2} $(x,rt,y)$ is a reduced extension of $(w_{rs}r,rt,st\cdot sw_{rs})$, and
$$
\begin{aligned}
T_{x}T_{rt}T_{y}&=\xi_{s}T_{x'\cdot w_{rs}}T_{t\cdot w_{rs}\cdot y''}+T_{x'\cdot w_{rs}s}T_{t\cdot w_{rs}\cdot y''}\\
&=\xi_{s}\xi_{r}T_{x'\cdot w_{rs}r\cdot t\cdot w_{rs}\cdot y''}+\xi_{s}T_{x'\cdot w_{rs}r\cdot t\cdot rw_{rs}\cdot y''}\\
&\ \ \ +\xi_{r}T_{x'\cdot w_{rs}sr\cdot t\cdot w_{rs}\cdot y''}+T_{x'\cdot w_{rs}sr\cdot t\cdot rw_{rs}\cdot y''}.
\end{aligned}
$$
(ii) $\mathcal{R}(xrs)=\{r,s\}$. We assume $x=x'\cdot w_{rs}sr$ for some $x'\in W$ with $\mathcal{R}(x')\subseteq\{t\}$, thus $T_{x}T_{rt}T_{y}=T_{x'\cdot w_{rs}s}T_{ty}$. Since $l(xwy)<l(x)+l(w)+l(y)$, we must have
$\mathcal{L}(ty)\subseteq\{s,t\}$, we assume $y=st\cdot y'$ for some $y'\in W$ with $\mathcal{L}(y')\subseteq\{r\}$ , then $T_{x}T_{rt}T_{y}=T_{x'\cdot w_{rs}}T_{tsy'}$. By Lemma \ref{lem:a2}(6), we have
$\mathcal{L}(tsy')\subseteq\{r,t\}$ since $l(xwy)<l(x)+l(w)+l(y)$. Now we assume $y'=sw_{rs}\cdot y''$ for some $y''\in W$ with $\mathcal{L}(y'')\subseteq\{t\}$. Then by Lemma \ref{lem:a2}(6), \\\circled{3} $(x,rt,y)$ is a reduced extension of $(w_{rs}sr,rt,st\cdot sw_{rs})$, and
$$
\begin{aligned}
T_{x}T_{rt}T_{y}&=T_{x'\cdot w_{rs}}T_{t\cdot w_{rs}\cdot y''}\\
&=\xi_{r}T_{x'\cdot w_{rs}\cdot t\cdot rw_{rs}\cdot y''}+T_{x'\cdot w_{rs}r\cdot t\cdot rw_{rs}\cdot y''}.
\end{aligned}
$$
(iii) $\mathcal{R}(xt)=\{s,t\}$. We assume $x=x'\cdot ts$ for some $x'\in W$ with $\mathcal{R}(x')\subseteq\{r\}$, thus $T_{x}T_{rt}T_{y}=T_{x'}T_{sts}T_{ry}$. If $\mathcal{L}(ry)=\{r,s\}$, we may consider the transpose of $(x,rt,y)$, then we are in case (i). If $\mathcal{L}(ry)=\{r\}$, by Lemma \ref{lem:a2}(7), we must have $\mathcal{R}(x's)=\mathcal{L}(sry)=\{r,s\}$. We consider the transpose of $(x,rt,y)$, then we are in case (ii).

(7) If $w=st$, since $l(y)\geq 1$, we assume $y=ry'$ for some $y'\in W$ with $\mathcal{L}(y')\subseteq\{s\}$, thus $T_{x}T_{st}T_{y}=T_{xs}T_{rt}T_{y'}$. If $\mathcal{R}(xs)=\{s\}$, by the case of $w=rt$, we must be in the following cases if $l(xwy)<l(x)+l(w)+l(y)$.
\\\circled{1} $(x,st,y)$ is a reduced extension of $(w_{rs}rs,st,rst)$.
\\\circled{2} $(x,st,y)$ is a reduced extension of $(w_{rs}rs,st,rst\cdot sw_{rs})$.
\\\circled{3} $(x,st,y)$ is a reduced extension of $(w_{rs}srs,st,rst\cdot sw_{rs})$.
\\Then $T_{x}T_{st}T_{y}$ can be easily computed in all these cases. If $\mathcal{R}(xs)=\{r,s\}$, we assume $x=x'\cdot w_{rs}s$ for some $x'\in W$ with $\mathcal{R}(x')\subseteq\{t\}$. Then
$$
\begin{aligned}
T_{x}T_{st}T_{y}&=T_{x'\cdot w_{rs}}T_{rt}T_{y'}\\
&=\xi_{r}T_{x'\cdot w_{rs}}T_{ty'}+T_{x'\cdot w_{rs}r}T_{ty'}.
\end{aligned}
$$
If $\mathcal{L}(ty')=\{t\}$, then by Lemma \ref{lem:a2}(6), we have $$T_{x}T_{st}T_{y}=\xi_{r}T_{x'\cdot w_{rs}\cdot ty'}+T_{x'\cdot w_{rs}r\cdot ty'},$$
so \circled{4} $(x,st,y)$ is a reduced extension of $(w_{rs}s,st,r)$.

If $\mathcal{L}(ty')=\{s,t\}$, we assume $y'=st\cdot y''$ for some $y''\in W$ with $\mathcal{L}(y'')\subseteq\{r\}$. Then
$$
\begin{aligned}
T_{x}T_{st}T_{y}&=\xi_{r}T_{x'\cdot w_{rs}}T_{sts\cdot y''}+T_{x'\cdot w_{rs}r}T_{sts\cdot y''}\\
&=\xi_{r}\xi_{s}T_{x'\cdot w_{rs}}T_{ts\cdot y''}+\xi_{r}T_{x'\cdot w_{rs}s}T_{ts\cdot y''}\\
&\ \ \ +\xi_{s}T_{x'\cdot w_{rs}r}T_{ts\cdot y''}+T_{x'\cdot w_{rs}rs}T_{ts\cdot y''}.
\end{aligned}
$$
If $\mathcal{L}(tsy'')=\{t\}$, by Lemma \ref{lem:a2}(6) and the case of $w=rsrsr$, we we must be in the following cases.
\\\circled{5} $(x,st,y)$ is a reduced extension of $(w_{rs}s,st,rst)$.
\\\circled{6} $m_{rs}=7$, $(x,st,y)$ is a reduced extension of $(w_{rs}r\cdot t\cdot w_{rs}s,st,rst\cdot srw_{rs})$.
\\If $\mathcal{L}(tsy'')=\{r,t\}$, we assume $y''=sw_{rs}\cdot y'''$ for some $y'''\in W$ with $\mathcal{L}(y''')\subseteq\{t\}$. By Lemma \ref{lem:a2}(6), we have
$$
\begin{aligned}
T_{x}T_{st}T_{y}&=\xi_{r}\xi_{s}T_{x'\cdot w_{rs}}T_{t\cdot w_{rs}\cdot y'''}+\xi_{r}T_{x'\cdot w_{rs}s}T_{t\cdot w_{rs}\cdot y'''}\\
&\ \ \ +\xi_{s}T_{x'\cdot w_{rs}r}T_{t\cdot w_{rs}\cdot y'''}+T_{x'\cdot w_{rs}rs}T_{t\cdot w_{rs}\cdot y'''}\\
&=\xi_{r}^{2}\xi_{s}T_{x'\cdot w_{rs}r\cdot t\cdot w_{rs}\cdot y'''}+\xi_{r}\xi_{s}T_{x'\cdot w_{rs}r\cdot t\cdot rw_{rs}\cdot y'''}\\
&\ \ \ +\xi_{r}^{2}T_{x'\cdot w_{rs}sr\cdot t\cdot w_{rs}\cdot y'''}+\xi_{r}T_{x'\cdot w_{rs}sr\cdot t\cdot rw_{rs}\cdot y'''}\\
&\ \ \ +\xi_{s}T_{x'\cdot w_{rs}r\cdot t\cdot w_{rs}\cdot y'''}+\xi_{r}T_{x'\cdot w_{rs}rs}T_{t\cdot rw_{rs}\cdot y'''}+T_{x'\cdot w_{rs}rsr}T_{t\cdot rw_{rs}\cdot y'''}.
\end{aligned}
$$
By Lemma \ref{lem:a2}(6), we must be in the following cases.
\\\circled{7} $(x,st,y)$ is a reduced extension of $(w_{rs}s,st,rst\cdot sw_{rs})$.
\\\circled{8} $m_{rs}=7$, $(x,st,y)$ is a reduced extension of $(w_{rs}r\cdot t\cdot w_{rs}s,st,rst\cdot sw_{rs})$.

(8) At last, we consider the case of $w=rs$. Since $l(x)\geq 1$, we assume $x=x't$ for some $x'\in W$ with $\mathcal{R}(x')\subseteq\{s\}$, thus $T_{x}T_{rs}T_{y}=T_{x'}T_{rt}T_{sy}$. If $\mathcal{L}(sy)=\{s\}$, by the case of $w=rt$, if $l(xwy)<l(x)+l(w)+l(y)$, we must have
\\\circled{1} $(x,rs,y)$ is a reduced extension of $(w_{rs}r\cdot t,rs,t)$.
\\If $\mathcal{L}(sy)=\{s,t\}$, we assume $y=ts\cdot y'$, for some $y'\in W$ with $\mathcal{L}(y')\subseteq\{r\}$, then we have
\begin{equation}\label{eq:rsrs}
\begin{aligned}
T_{x}T_{rs}T_{y}&=T_{x'rt}T_{tsty'}\\
&=\xi_{t}T_{x'r}T_{stsy'}+T_{x'r}T_{sty'}.
\end{aligned}
\end{equation}
First we consider $\mathcal{L}(x'r)=\{r\}$. By Lemma \ref{lem:a2}(7), we have $T_{x'r}T_{stsy'}=T_{x'\cdot rsts\cdot y'}$. By the case of $w=st$, we must be in the following cases.
\\\circled{2} $(x,rs,y)$ is a reduced extension of $(t,rs,ts)$.
\\\circled{3} $(x,rs,y)$ is a reduced extension of $(w_{rs}rsr\cdot t,rs,tsrst)$.
\\\circled{4} $(x,rs,y)$ is a reduced extension of $(w_{rs}rsr\cdot t,rs,tsrst\cdot sw_{rs})$.
\\\circled{5} $(x,rs,y)$ is a reduced extension of $(w_{rs}srsr\cdot t,rs,tsrst\cdot sw_{rs})$.
\\\circled{6} $(x,rs,y)$ is a reduced extension of $(w_{rs}sr\cdot t,rs,tsr)$.
\\\circled{7} $(x,rs,y)$ is a reduced extension of $(w_{rs}sr\cdot t,rs,tsrst)$.
\\\circled{8} $m_{rs}=7$, $(x,rs,y)$ is a reduced extension of $(w_{rs}r\cdot t\cdot w_{rs}sr\cdot t,rs,tsrst\cdot srw_{rs})$.
\\\circled{9} $(x,rs,y)$ is a reduced extension of $(w_{rs}sr\cdot t,rs,tsrst\cdot sw_{rs})$.
\\\circled{10} $m_{rs}=7$, $(x,rs,y)$ is a reduced extension of $(w_{rs}r\cdot t\cdot w_{rs}sr\cdot t,rs,tsrst\cdot sw_{rs})$.
Then we consider $\mathcal{L}(x'r)=\{r,s\}$. We assume $x'=x''\cdot w_{rs}r$ for some $x''\in W$ with $\mathcal{R}(x'')\subseteq\{t\}$, thus by Lemma \ref{lem:a2}(6), we have
$$
\begin{aligned}
T_{x}T_{rs}T_{y}&=\xi_{t}T_{x''\cdot w_{rs}}T_{stsy'}+T_{x''\cdot w_{rs}}T_{sty'}\\
&=\xi_{t}\xi_{s}T_{x''\cdot w_{rs}}T_{tsy'}+\xi_{t}T_{x''\cdot w_{rs}s}T_{tsy'}+\xi_{s}T_{x''\cdot w_{rs}}T_{ty'}+T_{x''\cdot w_{rs}s}T_{ty'}\\
&=\xi_{t}\xi_{s}T_{x''\cdot w_{rs}\cdot tsy'}+\xi_{t}T_{x''\cdot w_{rs}s\cdot tsy'}+\xi_{s}T_{x''\cdot w_{rs}}T_{ty'}+T_{x''\cdot w_{rs}s}T_{ty'}.
\end{aligned}
$$
If $y'=e$, then
\\\circled{11} $(x,rs,y)$ is a reduced extension of $(w_{rs}r\cdot t,rs,ts)$.
\\If $y'\neq e$, we assume $y'=ry''$ for some $y''\in W$ with $\mathcal{L}(y'')\subseteq\{s\}$, thus
$$
\begin{aligned}
T_{x}T_{rs}T_{y}&=\xi_{t}\xi_{s}T_{x''\cdot w_{rs}\cdot tsry''}+\xi_{t}T_{x''\cdot w_{rs}s\cdot tsry''}+\xi_{s}T_{x''\cdot w_{rs}}T_{try''}+T_{x''\cdot w_{rs}s}T_{try''}\\
&=\xi_{t}\xi_{s}T_{x''\cdot w_{rs}\cdot tsry''}+\xi_{t}T_{x''\cdot w_{rs}s\cdot tsry''}+\xi_{s}\xi_{r}T_{x''\cdot w_{rs}}T_{ty''}\\
&\ \ \ +\xi_{s}T_{x''\cdot w_{rs}r}T_{ty''}+\xi_{r}T_{x''\cdot w_{rs}s}T_{ty''}+T_{x''\cdot w_{rs}sr}T_{ty''}.
\end{aligned}
$$
If $\mathcal{L}(ty'')=\{t\}$, then
\\\circled{12} $(x,rs,y)$ is a reduced extension of $(w_{rs}r\cdot t,rs,tsr)$.
\\If $\mathcal{L}(ty'')=\{s,t\}$, we assume $y''=sty'''$ for some $y'''\in W$ with $\mathcal{L}(y''')\subseteq\{r\}$, thus
$$
\begin{aligned}
T_{x}T_{rs}T_{y}&=\xi_{t}\xi_{s}T_{x''\cdot w_{rs}\cdot tsrsty'''}+\xi_{t}T_{x''\cdot w_{rs}s\cdot tsrsty'''}+\xi_{s}\xi_{r}T_{x''\cdot w_{rs}}T_{stsy'''}\\
&\ \ \ +\xi_{s}T_{x''\cdot w_{rs}r}T_{stsy'''}+\xi_{r}T_{x''\cdot w_{rs}}T_{tsy'''}+T_{x''\cdot w_{rs}sr}T_{stsy'''}\\
&=\xi_{t}\xi_{s}T_{x''\cdot w_{rs}\cdot tsrsty'''}+\xi_{t}T_{x''\cdot w_{rs}s\cdot tsrsty'''}+\xi_{s}^{2}\xi_{r}T_{x''\cdot w_{rs}}T_{tsy'''}\\
&\ \ \ +\xi_{s}\xi_{r}T_{x''\cdot w_{rs}s}T_{tsy'''}+\xi_{s}^{2}T_{x''\cdot w_{rs}r}T_{tsy'''}+\xi_{s}T_{x''\cdot w_{rs}rs}T_{tsy'''}\\
&\ \ \ +\xi_{r}T_{x''\cdot w_{rs}}T_{tsy'''}+\xi_{s}T_{x''\cdot w_{rs}sr}T_{tsy'''}+T_{x''\cdot w_{rs}srs}T_{tsy'''}.
\end{aligned}
$$
If $\mathcal{L}(tsy''')=\{t\}$, then we must be in the following cases.
\\\circled{13} $(x,rs,y)$ is a reduced extension of $(w_{rs}r\cdot t,rs,tsrst)$.
\\\circled{14} $m_{rs}=8$, $(x,rs,y)$ is a reduced extension of $(w_{rs}r\cdot t\cdot w_{rs}r\cdot t,rs,tsrst\cdot srw_{rs})$.
\\\circled{15} $m_{rs}=7$, $(x,rs,y)$ is a reduced extension of $(st\cdot w_{rs}r\cdot t,rs,tsrst\cdot srw_{rs})$.
\\\circled{16} $m_{rs}=7$, $(x,rs,y)$ is a reduced extension of $(w_{rs}r\cdot t\cdot w_{rs}r\cdot t,rs,tsrst\cdot srw_{rs})$.
\\If $\mathcal{L}(tsy''')=\{r,t\}$, we assume $y'''=sw_{rs}\cdot y''''$ for some $y''''\in W$ with $\mathcal{L}(y'''')\subseteq\{t\}$, thus
\\\circled{17} $x=x'\cdot w_{rs}r\cdot t$, $y=tsrst\cdot sw_{rs}\cdot y'$ for some $x'',y''''\in W$ with $\mathcal{R}(x''),\mathcal{L}(y'''')\subseteq\{t\}$.
We have $\deg T_{x}T_{rs}T_{y}\leq L(rsrs)$ by \eqref{eq:rsrs} and Lemma \ref{lem:a2}(9).

This completes the proof.
\end{proof}

The following corollary follows  from Lemma \ref{lem:reduced2} and \eqref{eq:bound}.
\begin{cor}\label{cor:est}
Let  $N\in \mathbb{N}$, $ w\in W_{rs}$ and $x,y\in W$ with $ \mc{R}(x),\mc{L}(y)\subset\{t\} $.
\begin{itemize}
\item[(1)] Assume $ l(w)\geq 2 $. If $ m_{rs}\geq 8 $, then \[
\deg p_{w,w_{rs}}\nt_x\nt_w\nt_y\leq -L(rs).
\]
If $ m_{rs}=7 $, then \[
\deg p_{w,w_{rs}}\nt_x\nt_w\nt_y\leq -L(r)=-L(s).
\]
\item [(2)] Assume $ l(w)\leq 1$. Then we have
\begin{equation}\label{eq:complement}
\deg p_{w,w_{rs}}\nt_x\nt_w\nt_y\leq \max\{L(r),L(s) \}-L(w_{rs})+N.
\end{equation}
\end{itemize}
\end{cor}

\section{Conditions for the equality}\label{sec:take}

In sections \ref{sec:take}, \ref{sec:add}, \ref{sec:est},  we prove Conjecture \ref{conj:dim2} for  Coxeter groups of rank 3 that are listed in section \ref{sec:intr}.
In these sections, we  assume that $ W_{>N}=\Omega_{>N} $, and $ W_{>N} $ is $ \prec_{LR} $ closed.

By assumption,  $ \wn=\Omega_{\leq N} $. We will frequently use the argument: if $ d\in D $ appears in $ z \in \wn $, then we have $ d\in D_{\leq N} $. Otherwise, $ d\in D_{>N} $ implies that $ z\in \Omega_{>N}=W_{>N} $, a contradiction with $ z\in \wn $.
 \begin{prop}\label{prop:bound}
For any $ x,y\in \wn $, we have $ \deg \nt_x\nt_y\leq N $, and  the equality holds only if $ x,y\in \Omega_{\geq N} $.
\end{prop}

The proof will occupy the rest of this section.
The first half of this proposition is known by \eqref{eq:bound}. The key point is  to prove $ x,y\in\Omega_{\geq N} $ when the equality holds.
For this, we use induction on $l(y) $. It is easy to check the proposition for $ l(y)=0,1 $. Now assume  $ l(y)\geq 2 $ and that\begin{equation}\label{eq:ind-hyp}
\text{ for } y'\text{ with }l(y')<l(y), \deg \nt_{x'}\nt_{y'}= N\text{ implies }x',y'\in \Omega_{\geq N}.
\end{equation} Let $ t_1\in\mc{L}(y) $, $ t_2\in\mc{L}(t_1y) $, $ I=\{t_1,t_2\} $.
Write $ x=x_1\cdot u $ and $ y=v\cdot y_1 $ with $u,v\in W_I$, $x_{1},y_{1}\in W$ and $\mathcal{R}(x_{1}),\mathcal{L}(y_{1})\subseteq S\setminus I$. We have\[
\nt_x\nt_y=\sum_{w\in(W_I)_{\leq N}} \nf_{u,v,w} \nt_{x_1}\nt_w\nt_{y_1}.
\]
Thus for our purpose, it suffices to prove that, for every $w\in(W_I)_{\leq N}$,
\begin{equation}\label{eq:claim}
\deg( \nf_{u,v,w} \nt_{x_1}\nt_w\nt_{y_1} )\leq N  ,\text{ and the equality holds only if }x,y\in \Omega_{\geq N} .
\end{equation}

If $ l(x_1wy_1)=l(x_1)+l(w)+l(y_1) $, then the equality in claim \eqref{eq:claim} holds only if $  \deg \nf_{u,v,w}  =N$, which implies that $ u,v\in \Omega_{\geq N} $ and hence $ x,y\in \Omega_{\geq N} $ (see case (i) of the proof of \cite[Prop.6.3]{xie2019}).
For the case of $ w $ with $ l(w)\leq 1 $, we refer the reader to cases (iii)(iv) of the proof of \cite[Prop. 6.3]{xie2019}.

In the rest of the proof, we assume that $ l(w)\geq 2 $ and $ l(x_1wy_1)<l(x_1)+l(w)+l(y_1)  $.

Note that $  \nf_{u,v,w} $ has  following 3 cases.
\begin{itemize}
\item [(1)] $ (W_I)_{>N}=\emptyset $, and $ \nf_{u,v,w}= f_{u,v,w}$;
\item [(2)] $ (W_I)_{>N}=\{w_I\}$, and $\nf_{u,v,w}= f_{u,v,w}-f_{u,v,w_I}p_{w,w_I} ;$
\item [(3)] $L(t_{1})\neq L(t_{2})$, $ (W_I)_{>N}=\{d_I, w_I\}$, and \[
	\nf_{u,v,w} =f_{u,v,w}-f_{u,v,w_I}p_{w,w_I}-(f_{u,v,d_I}-f_{u,v,w_I}p_{d_I,w_I})p_{w,d_I}.
	\]
\end{itemize}

\textbf{A general procedure}.
Let $  T_{x_1}T_wT_{y_1}  =\sum_{z}\alpha_zT_z\text{ with } \alpha_z\in\mc{A}$. Claim \eqref{eq:claim} will follow from the following three steps.

\newcommand{\cs}[1]{

\vspace{5pt}\noindent\textbf{#1}}

\newcommand{\css}[1]{

\vspace{5pt}{\color{blue}#1}.}

\css{Step I}  The goal of this step is to prove that $\deg (f_{u,v,w} \nt_{x_1}\nt_{w}\nt_{y_1})\leq N$ and the equality holds only if $ x,y\in \Omega_{\geq N} $.

Let $ \mb{\xi} $ be a monomial that appears in $ f_{u,v,w} $ (with positive coefficient).
Then for this step it suffice to prove  \[
\deg \xi+\deg(\nt_{x_1}\nt_{w}\nt_{y_1})\leq N,
\]
and that the equality holds only if $ x,y\in \Omega_{\geq N} $.
If $  \deg\xi\leq 0 $, then we can use induction hypothesis \eqref{eq:ind-hyp}, and  hence we focus on the case of $ \deg\xi >0$.
By subsection \ref{subsec:pterm}, we have a  condition $\mf{R}_1$ on $ u,v $. Then condition $ \mf{R}_1 $, together with $ x=x_1u, y=vy_1\in \wn=\Omega_{\leq N} $, gives a restriction $ \mf{R}_2 $ on $ L(r), L(s), L(t)$. Using  $ \mf{R}_2 $, we have an inequality $ \mf{R}_3 $.
Then we prove  $ \deg \xi+\deg \alpha_z +\deg \nt_z \leq N$   for any $ z $ and that the equality holds only if $ x,y\in \Omega_{\geq N} $. This will complete  Step I.

\css{Step II} Assume that $ W_I $ is finite and $ w_I\in (W_I)_{>N} $, i.e. $ N< L(w_I) $. The goal of this step is to prove that
$$\deg f_{u,v,w_I}p_{w,w_I} \nt_{x_1}\nt_{w}\nt_{y_1}\leq N,$$
 and the equality holds only if $ x,y\in \Omega_{\geq N} $.

 Let $\delta= \deg f_{u,v,w_I}p_{w,w_I} $. If $ \delta\leq 0 $, then we apply induction hypothesis \eqref{eq:ind-hyp} to $  \nt_{x_1w}\nt_{y_1}$. If $ \delta>0 $, then by subsection \ref{subsec:delta}, we have a restriction    $ \mf{J}_1 $ on $ u,v,\delta $, which, together with $ x_1u,v y_1\in \Omega_{\leq N} $, gives a  restriction $ \mf{J}_2 $ on $ L(r) $, $ L(s) $, $ L(t) $. This restriction $ \mf{J}_2 $ gives an inequality $ \mf{J}_3 $, and then implies  $ \delta+\deg \alpha_z+\deg \nt_z\leq N  $ for any $ z $ and  the condition of taking equality. This  will complete Step II.

\css{Step III} Assume that $ W_I $ is finite, $ d_I,w_I\in (W_I)_{>N} $ and $w\leq d_{I}$.
The goal of this step is to prove that\[
\deg( (f_{u,v,d_I}-f_{u,v,w_I}p_{d_I,w_I})p_{w,d_I}\nt_{x_1}\nt_{w}\nt_{y_1})\leq N,
\]
and the equality holds only if $ x,y\in \Omega_{\geq N} $.

 Let $\gamma= (f_{u,v,d_I}-f_{u,v,w_I}p_{d_I,w_I})p_{w,d_I} $. If $ \gamma\leq 0 $, then  we apply induction hypothesis \eqref{eq:ind-hyp} to $  \nt_{x_1w}\nt_{y_1}$. If $ \gamma>0 $, then  by results from subsection \ref{subsec:gamma}, we have a restriction    $ \mf{K}_1 $ on $ u,v,\delta $ and $ \gamma $, which, together with $ x_1u,v y_1\in \Omega_{\leq N} $, gives a  restriction $ \mf{K}_2 $ on $ L(r) $, $ L(s) $, $ L(t) $. This restriction $ \mf{K}_2 $ gives an inequality $ \mf{K}_3 $, and then implies  $ \gamma+\deg \alpha_z+\deg \nt_z\leq N  $ for any $ z $ and the condition of taking equality. This will complete  Step III.



\subsection{The case of $\infty=m_{rs}>m_{st}\geq3$} According to Lemma \ref{lem:reduced0}, the proof  is  divided into two cases as follows.

Since $ W_{rs}$ is an infinite group,   we only need to consider Step I when $ w\in W_{rs} $.
\cs{Case (1)}
$ w=srs $,  $x_{1}=x_{2}\cdot w_{st} s$, $y_{1}=sw_{st}\cdot y_{2}$ for some $x_{2},y_{2}\in W$ and
		$$
		T_{x_{1}}T_{srs}T_{y_{1}}=\xi_{t}T_{x_{2}\cdot w_{st}\cdot r\cdot tw_{st}\cdot y_{2}}+T_{x_{2}\cdot w_{st}t\cdot r\cdot tw_{st}\cdot y_{2}}.
		$$
 Apply the general procedure. According to Lemma \ref{lem:infinite2}, we have two cases for Step I. (Substitute $ \xi $, $ \mf{R}_1 $, $ \mf{R}_2  $, $ \mf{R}_3  $ of the following table into the general procedure outlined above.)
   \[ \begin{tabular}{cccc}
\hline
$ \xi $&$\mf{R}_1$&$ \mf{R}_2 $ &$ \mf{R}_3 $   \\
\hline
$ \xi_s $ &$ su<u $ and $ vs<v $ & $ L(w_{st})\leq N $ &  $ L(s)+L(t) < N$\\
\hline
$ \xi_r $ &$ (u,v)=( sr\cdot u' , u'^{-1}\cdot rs) $&$ L(rt)\leq N $&$ L(r)+L(t)\leq N  $\\
\hline
    \end{tabular}\]

\cs{Case (2)}
$ w=rs $, $x_{1}=x_{2}\cdot t$, $y_{1}=sw_{st}\cdot y_{2}$ for some $x_{2},y_{2}\in W$ and
		$$
		T_{x_{1}}T_{rs}T_{y_{1}}=\xi_{t}T_{x_{2}\cdot r\cdot w_{st}\cdot y_{2}}+T_{x_{2}\cdot r\cdot tw_{st}\cdot y_{2}}.
		$$
According to Lemma \ref{lem:infinite1}, we have two cases for Step I.
  \begin{center}
    \begin{tabular}{cccc}
\hline
$ \xi $&$\mf{R}_1$&$ \mf{R}_2 $ &$ \mf{R}_3 $     \\
\hline
$ \xi_s $ &$ (u,v)=(rs\cdot u', u'^{-1}\cdot s) $&$ L(w_{st})\leq N $&  $ L(s)+L(t) < N$\\
\hline
$ \xi_r $ &$ (u,v)=(r\cdot u', u'^{-1}\cdot rs)$&$ L(rt) \leq N$&$ L(r)+L(t)\leq N  $\\
\hline
    \end{tabular}
  \end{center}

Note that we also need to verify the transpose cases, but the proofs are  similar.
Hereafter, we always omit the transpose cases.

\subsection{ The case of $4\leq m_{rs}, m_{st}< \infty$ with $(m_{rs},m_{st})\neq (4,4)$}

Since $ 2L(r)+2L(s)\leq L(w_{rs}) $ and $ 2L(s)+2L(t)\leq L(w_{st}) $ and at most one of them holds, we have
\begin{equation}\label{eq:frequent44}
 L(srst)< \max\{L(w_{rs}),L(w_{st})\}.
\end{equation}
 In the following this property  is often used without mention.

According to Lemma \ref{lem:reduced}, this case  is  divided into the following ones. (In addition to the transpose ones, the cases with $ r,t $ exchanged are omitted.)

\cs{Case (1)}
$ w=srs $,  $ x_1=x_2\cdot w_{st} s$, $ y_1=sw_{st}\cdot y_2 $ for some $x_2,y_2\in W$ and
		$$
		T_{x_1}T_{srs}T_{y_1}=\xi_{t}T_{x_2\cdot w_{st}\cdot r\cdot tw_{st}\cdot y_2}+T_{x_2\cdot w_{st}t\cdot r\cdot tw_{st}\cdot y_2}.
		$$
We have $ I=\{s,r\} $. We follow the general procedure.

\css{Step I} If $u=w_{I}$ or $v=w_{I}$, then $ w_{rs}, w_{st}\in D_{\leq N}$, so we have
$$\deg f_{u,v,w} \nt_{x_1}\nt_{w}\nt_{y_1}\leq L(srs)+L(t)<N.$$
If $u\neq w_{I}$ and $v\neq w_{I}$, we have the following cases according to Lemma \ref{lem:fuvw}.
 \[
 \begin{tabular}{cccc}
\hline
$ \xi $&$\mf{R}_1$&$ \mf{R}_2 $  & $ \mf{R}_3 $ \\\hline
$ \xi_s $&$ su<u $ or $ vs<v $&$ w_{st}\in D_{\leq N} $&$ L(st) < N$\\\hline
$ \xi_r $&$(u,v)=( sr\cdot u' , u'^{-1}\cdot rs)$ or $(sw_I, w_Is)$&$ rt\in D_{\leq N} $&$ L(rt)\leq N $\\\hline
    \end{tabular}
  \]

\css{Step II} $ w_I\in (W_I)_{>N} $. By Corollary \ref{cor:degfp2}, for $ \delta>0 $ we have the following cases.
\[\begin{tabular}{ccc}
\hline
$ \mf{J}_1 $&$ \mf{J}_2 $&$ \mf{J}_3 $\\ \hline
$ su<u $ or $ vs<v $, and $ \delta< 2L(s) $& $ w_{st}\in D_{\leq N} $&$  2L(s)+L(t)<N  $\\ \hline
$ ru<u $, $ vr<v $,  and $ \delta=L(r) $ & $ rt\in D_{\leq N} $ & $ L(r)+L(t)\leq N $\\ \hline
\end{tabular}\]
(Substitute $ \mf{J}_1 $, $ \mf{J}_2 $, $ \mf{J}_3 $ into Step II of the general procedure.)

\css{Step III} $ \{d_I,w_I\} \subset  (W_I)_{>N}$. For $ \gamma> 0$, by Corollary \ref{cor:degnfp}, we have
\[
\mf{K}_1:  L(s)>L(r) ,su<u  \text{ and }vs<v ,\text{ and }\gamma \leq  L(r)
\]
\[
\mf{K}_2: w_{st}\in D_{\leq N}\quad\quad \mf{K}_3: L(s)+L(t)<N.
\]
(Substitute $ \mf{K}_1 $, $ \mf{K}_2 $, $ \mf{K}_3 $ into Step III of the general procedure.)

\cs{Case (2)} $ w=rsr $, $ m_{st}=4 $, $x_1=x_2\cdot w_{rs}r\cdot t$ and $y_1=t\cdot rw_{rs}\cdot y_2$ for some $x_2,y_2\in W$,  and
		$$
		T_{x_1}T_{rsr}T_{y_1}=\xi_{s}T_{x_2\cdot w_{rs}\cdot tst\cdot sw_{rs}y_2}+T_{x_2w_{rs}s\cdot tst\cdot sw_{rs}\cdot y_2}.
		$$
We have $ I=\{s,r\} $. We follow the general procedure.

\css{Step I} If $u=w_{I}$ or $v=w_{I}$, then $ w_{rs}\in D_{\leq N}$, so we have
$$\deg f_{u,v,w} \nt_{x_1}\nt_{w}\nt_{y_1}\leq L(rsr)+L(s)<N.$$
If $u\neq w_{I}$ and $v\neq w_{I}$, we have the following cases according to Lemma \ref{lem:fuvw}.
\begin{itemize}
\item [(1)]$\xi= \xi_r $.
  \[\mf{R}_1:ru<u\text{ or } vr<v\quad \mf{R}_2: w_I\in D_{\leq N} \]\[
\mf{R}_3: L(s)+L(r)< N.\]

\item [(2)]$ \xi_s $ appears in $f_{u,v,w}   $, and   $(u,v)=( rs\cdot u' , u'^{-1}\cdot sr)$ or $(rw_I, w_Ir)$.

The first case implies that $ L(w_{rs})\leq N $ since $ y\in \wn $, hence $ 2L(s)<N $.

In the second case, i.e. $ (u,v)= (rw_I, w_Ir)$ we need to prove   \begin{equation}\label{eq:minus}
2L(s)+\deg \nt_{x_2\cdot w_{rs}\cdot tst\cdot sw_{rs}y_2}\leq N
\end{equation}and that the equality holds only if $ x,y\in\Omega_{\geq N} $. If $ L(w_{rs})\leq N $, it is obvious.  Hence we assume $ L(w_{rs})> N $, so $ x_2\cdot w_{rs}\cdot tst\cdot sw_{rs}y_2\notin\wn $ and $ (m_{st}=4) $ \[
\nt_{x_2\cdot w_{rs}\cdot tst\cdot sw_{rs}y_2}=-\sum_{z<w_{rs}}p_{z,w_{rs}}\nt_{x_2\cdot (w_{rs}s)\cdot tst\cdot z}\nt_{y_2}
\]

For every $z<w_{rs}$, we will prove that
\begin{equation}\label{eq:case11}
2L(s)+\deg p_{z,w_{rs}}+\deg \nt_{x_2\cdot (w_{rs}s)\cdot tst\cdot z}\nt_{y_2}\leq N
\end{equation}
and the equality holds only if $ x,y\in\Omega_{\geq N} $.

Assume first $ l(z)\geq 4$, then $ x_2\cdot (w_{rs}s)\cdot tst\cdot z\cdot y_2 $ is reduced by Lemma \ref{lem:a1}(11). We have
$$2L(s)+\deg p_{z,w_{rs}}\leq 2L(s)-L(r)<L'(d_{I})\leq N, \mbox{ if } L(s)>L(r),$$
$$2L(s)+\deg p_{z,w_{rs}}\leq L(s)\leq N, \mbox{ if } L(r)\geq L(s),$$
and the equality holds only if $ x,y\in\Omega_{\geq N} $ since $ s $ appears in $x$ and $y$. (Note that $ m_{rs}\geq 5 $ due to $ m_{st}=4 $.)

Assume now $ l(z)\leq 3 $, by induction hypothesis, we have
$$\deg\nt_{x_2\cdot (w_{rs}s)\cdot tst\cdot z}\nt_{y_2}\leq N,$$
and the equality holds only if $ y_2\in \Omega_{\geq N} $. If $m_{rs}\neq 6$ or $w\neq srs$, then $2L(s)+\deg p_{z,w_{rs}}\leq 0$, so \eqref{eq:case11} holds. Assume $m_{rs}=6$ and $w=srs$. If $ x_2\cdot (w_{rs}s)\cdot tst\cdot z\cdot y_2 $ is reduced, then $ L(s)-2L(r)< N $.
If $ x_2\cdot (w_{rs}s)\cdot tst\cdot z\cdot y_2 $ is not reduced, by Lemma \ref{lem:reduced}(1), we have $ \deg \nt_{x_2\cdot (w_{rs}s)\cdot tst\cdot z}\nt_{y_2}\leq L(s) $ and $ y_2=sw_{st}\cdot y' $ for some $ y' $ and hence $ w_{st} $ appears in $ y\in \wn $. We get
\begin{align*}
&2L(s)+\deg p_{z,w_{rs}}+\deg \nt_{x_2\cdot (w_{rs}s)\cdot tst\cdot z}\nt_{y_2}\\
 \leq &2L(s)-(L(s)+2L(r))+L(s)
<L(w_{st})\leq N.
\end{align*}

\end{itemize}

\css{Step II} $ w_I\in (W_I)_{>N}$.
According to Corollary \ref{cor:degfp2}, for $ \delta>0 $, we have the following cases.
\begin{itemize}
\item[(1)]
$ \mf{J}_1: ru<u \text{ or } vr<v \text{ and }\delta< 2L(s) , $\[
\mf{J}_2: w_{rs}\in D_{\leq N},\quad  \mf{J}_3: 2L(r)+L(s)<N .
\]
\item [(2)]   $ su<u $, $ vs<v $,  and $ \delta=L(s) $. Then we need to prove $ L(s)+\deg (\nt_{x_1}\nt_{rsr}$ $\nt_{y_1})\leq N $ and that the equality holds only if $ x,y\in\Omega_{\geq N} $. For this, see \eqref{eq:minus}.
\end{itemize}

\css{Step III} $\{d_I,w_I\}\subset  (W_I)_{>N} $.
According to Corollary \ref{cor:degnfp}, we have
$ \mf{K_1} $: $ L(r)>L(s) $,  $ ru<u $ and  $ vr<v $, and   $ \gamma\leq  L(s)$,
\[ \mf{K}_2: w_{rs}\in D_{\leq N},\quad \mf{K}_3: \gamma+L(s)\leq 2L(s)<N.
\]

\cs{Case (3)} $ w=rt $, $x_1=x_2\cdot w_{rs}r$, $y_2=tw_{st}\cdot y_2$ for some $x_2,y_2\in W$,  and
		$$
		T_{x_1}T_{rt}T_{y_1}=\xi_{s}T_{x_2\cdot w_{rs}\cdot sw_{st}\cdot y_2}+T_{x_2\cdot w_{rs}s\cdot sw_{st}\cdot y_2}.
		$$
In this case, $ I=\{r,t\} $. Since $ w=rt\in (W_I)_{\leq N} $, and $ m_{rt}=2 $, we only need to consider step I. We have the following cases.\[
\begin{tabular}{cccc}
\hline
$ \xi $&$\mf{R}_1$&$ \mf{R}_2 $ &$ \mf{R}_3 $   \\
\hline
$ \xi_r\xi_t $ &$ u=v=rt $ &$  w_{st}, w_{rs}\in D_{\leq N}  $&  $L(r)+L(t)+L(s)< N$\\
\hline
$ \xi_r $ &$ (u,v)=(r,rt), (rt,r) $&$w_{rs}\in D_{\leq N}$&$ L(r)+L(s)< N  $\\
\hline
$ \xi_t$ &$ (u,v)=(t,rt), (rt,t) $&$w_{st}\in D_{\leq N}$&$ L(s)+L(t)< N  $\\
\hline
\end{tabular}
\]
\cs{Case (4)} $ w=rs $.

We have $I=\{r,s\}$. By Corollary \ref{cor:degnfp2}, $ \gamma\leq 0 $ since $ l(w)=2 $. Thus we only consider Step I and  II.

Here we only give the proof of case \circled{1} of Lemma \ref{lem:reduced}(4) and the other cases are similar.
In this case,  $x_1=x_2\cdot w_{rs}r\cdot t$, $y_1=sw_{st}\cdot y_2$ for some $x_2,y_2\in W$, and
			$$
            \begin{aligned}
            T_{x_1}T_{rs}T_{y_1}&=\xi_{t}\xi_{s}T_{x_2\cdot w_{rs}s\cdot w_{st}\cdot y_2}+\xi_{t}T_{x_2\cdot w_{rs}s\cdot sw_{st}\cdot y_2}\\
            &\ \ \ +\xi_{s}T_{x_2\cdot w_{rs}\cdot stw_{st}\cdot y_2}+T_{x_2\cdot w_{rs}s\cdot stw_{st}\cdot y_2}.
            \end{aligned}
            $$

\css{Step I} By Lemma \ref{lem:fuvw2}, we have following 3 cases.
\begin{itemize}
\item [(1)] $ \xi=\xi_r\xi_s $.  $ \mf{R}_1 $: $ u=v=w_{rs} $.
$ \mf{R}_2: w_{rs}, w_{st}\in D_{\leq N}. $ $ \mf{R}_3: L(r)+2L(s)+L(t)<N. $
\item [(2)] $ \xi_r $ appears in $f_{u,v,w}$ and $ ru<u $.


Then $ x\in \wn $ implies that $ L(w_{rs})\leq N $ and $ L(rt)\leq N $.
 If $ L(t)\leq L(s) $, we have $L(r) +L(s)+L(t)\leq 2L(s)+L(r)<N  $.
It is similar for $ L(t)\leq L(r) $. For the case of $ L(w_{st})\leq N $, it is also easy.

In the following we assume that $ L(t)>L(s) $, $ L(t)>L(r) $ and $  L(w_{st})> N$. Thus by Lemma \ref{lem:fuvw2}(3), $(u,v)=(w_{rs},w_{rs}s)$.
\begin{itemize}
\item [(a)] $ m_{rs}\geq 5 $. Since $ L(t)>L(s) $, then the fact that $ sw_{st} $ appears in $ y$ implies that $ 2L(t)-L(s)\leq N $, i.e. $ L(t)-\frac12L(s)\leq \frac12 N $. Since $ m_{rs}\geq 5 $, we have $ 3L(s)+2L(r)\leq N $, i.e. $ \frac32 L(s)+L(r)\leq \frac12 N $. Then we obtain $ L(s)+L(r)+L(t)\leq N $, and the equality holds only if $ x,y\in \Omega_{\geq N} $.

\item [(b)]$ m_{rs}=4 $.
We will
 prove  \[
L(r)+\deg \xi_{t}\xi_{s}\nt_{x_2\cdot w_{rs}s\cdot w_{st}\cdot y_2}\leq N,
\]
and that the equality holds only if $ x,y\in \Omega_{\geq N} $.
Then we consider \[
\nt_{x_2\cdot w_{rs}s\cdot w_{st}\cdot y_2}=-\sum _{z< w_{st}}p_{z,w_{st}}\nt_{x_2\cdot w_{rs}s\cdot z}\nt_{ y_2}.
\]
It suffices to prove that
\begin{equation}\label{eq:case1}
\quad L(rst)+\deg p_{z,w_{st}}+\deg \nt_{x_2\cdot w_{rs}s\cdot z}\nt_{ y_2}\leq N,
\end{equation}
and the equality holds only if $ x,y\in \Omega_{\geq N} $.
Note that $ m_{st}\geq 6 $, since $ L(s)\neq L(t) $ and $ (m_{rs},m_{st})\neq (4,4) $. We have  two cases.
\begin{itemize}
\item [(i)]If $ l(z)\leq 2 $ or $ z=sts $, we have $ L(rst)+\deg p_{z,w_{st}}< 0$. Then \eqref{eq:case1} follows since $\deg \nt_{x_2\cdot w_{rs}s\cdot z}\nt_{ y_2}\leq N$.

\item [(ii)]If $ l(z)\geq 4 $ or $ z=tst $, then $ x_2\cdot w_{rs}s\cdot z\cdot  y_2 $ is reduced\footnote{Assume  $ x_2\cdot w_{rs}s\cdot z\cdot  y_2  $ is not reduced. By applying Lemma \ref{lem:reduced} (with $ r,t $ exchanged) to $ ( x_2\cdot w_{rs}s, z, y_2) $, we have $ z=tst $ and $  x_2\cdot w_{rs}s=x_2'\cdot w_{st}t\cdot r$ for some $ x'_2 $. Since $ m_{rs}=4 $, we have
$ x_2\cdot r=x_2'\cdot w_{st}ts $, a contradiction with Lemma  \ref{lem:a1}(5).}, and  we have $ \deg p_{z,w_{st}}\leq -L(s) $.
Then \eqref{eq:case1} follows and the equality holds only if $ L(rt)=N$, which implies $ x,y\in \Omega_{\geq N} $.
\end{itemize}
\end{itemize}
\item [(3)] $ \xi=\xi_s $. $ \mf{R}_1 : sv<v$. $ \mf{R}_2 : w_{st}\in D_{\leq N}$. $ \mf{R}_3: 2L(s)+L(t)< N  $.
\end{itemize}

\css{Step II} $ w_I\in (W_I)_{>N} $. By Lemma \ref{lem:degfp2}, for $ \delta>0 $, we have $$   L(r)\neq L(s), \delta=|L(r)-L(s)|, u=v=d_I . $$
If $ L(r)>L(s) $,
then $ w_{rs} $ appears in $ x\in x_1u \in \wn$, a contradiction with $ w_I\in (W_I)_{>N}  $.
If $ L(r)<L(s) $, then $ vs<v $ and $ w_{st}\in D_{\leq N} $, which implies that  $ \delta+L(st)<2L(s)+L(t)<N $.

\subsection{The case of $ m_{rs}\geq 7, m_{st}=3$}  One important feature of this case is that $ L(s)=L(t) $.

According to Lemma \ref{lem:reduced2}, the proof  is  divided into the following cases.

\cs{Case (1)} $ w=rsrsr $. We have $ I=\{r,s\} $.

If $L(w_{rs})\leq N$, then we have $\nf_{u,v,w}=f_{u,v,w}$, thus
$$\deg\nf_{u,v,w}\nt_{x_1}\nt_w\nt_{y_1}\leq L(rsrsr)+L(s)<L(w_{rs})\leq N.$$

If $L(w_{rs})>N$, then $u\neq w_{rs}$ and $v\neq w_{rs}$ since $u,v\in (W_{I})_{\leq N}$. We have $x_{1}=x_{2}\cdot w_{rs}r\cdot t$, $y_{1}=t\cdot rw_{rs}\cdot y_{2}$ for some $y_{1},y_{2}\in W$ with $\mathcal{R}(x_{2}),\mathcal{L}(y_{2})\subseteq\{t\}$, and
$$
\begin{aligned}
\nt_{x_1}\nt_{w}\nt_{y_1}&=\xi_{s}\nt_{x_{2}\cdot w_{rs}\cdot tsrst\cdot sw_{rs}\cdot y_{2}}+\nt_{x_{2}\cdot w_{rs}s\cdot tsrst\cdot sw_{rs}\cdot y_{2}}\\
&=-\xi_{s}\sum_{z<w_{rs}}p_{z,w_{rs}}\nt_{x_{2}\cdot z}\nt_{tsrst\cdot sw_{rs}\cdot y_{2}}+\nt_{x_{2}\cdot w_{rs}s\cdot tsrst\cdot sw_{rs}\cdot y_{2}}.
\end{aligned}
$$
In $ W_I $ we have
$\deg \nf_{u,v,w}\leq N,\mbox{ and the equality holds only if }u,v\in \Omega_{\geq N}.$
Then by Corollary \ref{cor:est}(1), we obtain
$$\deg \nf_{u,v,w}\left (-\xi_{s}\sum_{z<w_{rs},l(z)\geq 2}p_{z,w_{rs}}\nt_{x_{2}\cdot z}\nt_{tsrst\cdot sw_{rs}\cdot y_{2}} +\nt_{x_{2}\cdot w_{rs}s\cdot tsrst\cdot sw_{rs}\cdot y_{2}}\right )\leq N,$$
and the equality holds only if $x,y\in \Omega_{\geq N}$.

On the other hand, by Corollary \ref{cor:bbb}, we have
$\deg \nf_{u,v,w}<L(w).$
Then by Corollary \ref{cor:est}(2), we obtain
$$
\begin{aligned}
&\ \ \ \ \deg \nf_{u,v,w}\xi_{s}\sum_{z<w_{rs},l(z)\leq 1}p_{z,w_{rs}}\nt_{x_{2}\cdot z}\nt_{tsrst\cdot sw_{rs}\cdot y_{2}}\\
&<L(w)+L(s)+\max\{L(r),L(s)\}-L(w_{rs})+N\\
&\leq N.
\end{aligned}
$$

\cs{Case (2)} $ w=rsrs $. Take the same method as the case (1).

	\cs{Case (3)} $ w=srs $. We have $ I=\{r,s\} $.

\css{Step I}
By Lemma \ref{lem:reduced2}(3), we have $\deg T_{x_{1}}T_{srs}T_{y_{1}}\leq L(rs)$. Hence\[
\deg f_{u,v,w}\nt_{x_1}\nt_w\nt_{y_1}\leq L(srsrs).
\]
Thus if $ L(w_{rs})\leq N $, then we are done. In the following, we assume $ L(w_{rs})>N $. Thus $u,v\neq w_{rs}$. For case \circled{1} of Lemma \ref{lem:reduced2}(3), we have the following cases.\[
\begin{tabular}{cccc}
\hline
$ \xi $&$\mf{R}_1$&$ \mf{R}_2 $ &$ \mf{R}_3 $   \\
\hline
$ \xi_s $ &$su<u$ or $vs<v$&$sts\in D_{\leq N}$&$ L(s)+L(t)< N  $\\
\hline
$ \xi_r$ &$u=v^{-1}=sr\cdot u'$ or $sw_{rs}$&$rt\in D_{\leq N}$&$ L(r)+L(t)\leq N  $\\
\hline
\end{tabular}
\]
For cases \circled{2} and \circled{3}, we only need to consider the case of $ u=v^{-1}=sw_{rs} $, then apply Corollary \ref{cor:est}. For case \circled{4}, we have the following cases.\[
\begin{tabular}{cccc}
\hline
$ \xi $&$\mf{R}_1$&$ \mf{R}_2 $ &$ \mf{R}_3 $   \\
\hline
$ \xi_s $ &$vs<v$&$sts\in D_{\leq N}$&$ L(s)+L(t)< N  $\\
\hline
$ \xi_r$ &$u=v^{-1}=sw_{rs}$&$rt\in D_{\leq N}$&$ L(r)+L(t)\leq N  $\\
\hline
\end{tabular}
\]

\css{Step II} $ w_I\in (W_I)_{>N}$. Assume $\delta>0 $. For cases \circled{1} of Lemma \ref{lem:reduced2}(3), we have the following cases.
\[\begin{tabular}{ccc}
\hline
$ \mf{J}_1 $&$ \mf{J}_2 $&$ \mf{J}_3 $\\ \hline
$ su<u $ or $ vs<v $, and $ \delta< 2L(s) $& $ sts\in D_{\leq N} $&$  2L(s)+L(t)\leq N  $\\ \hline
$ ru<u $, $ vr<v $, and $ \delta=L(r) $ & $ rt\in D_{\leq N} $ & $ L(r)+L(t)\leq N $\\ \hline
\end{tabular}\]
For cases \circled{2} and \circled{3}, we only need to consider the case of $ ru<u $, $ vr<v $, and $\delta=L(r)$, then apply Corollary \ref{cor:est}. For case \circled{4}, we have the following cases.
\[\begin{tabular}{ccc}
\hline
$ \mf{J}_1 $&$ \mf{J}_2 $&$ \mf{J}_3 $\\ \hline
$ vs<v $ and $ \delta< 2L(s) $& $ sts\in D_{\leq N} $&$  2L(s)+L(t)\leq N  $\\ \hline
$ ru<u $, $ vr<v $, and $ \delta=L(r) $ & $ rt\in D_{\leq N} $ & $ L(r)+L(t)\leq N $\\ \hline
\end{tabular}\]

\css{Step III} $ d_I,w_I\in (W_I)_{>N} $. For $\gamma>0$, by Corollary \ref{cor:degnfp}, we have
$L(s)>L(r)$, $su<u$, $vs<v$, $\gamma\leq L(r)$. Since $m_{rs}$ is even and $L(w_{rs})>N$, we must be in case \circled{1} of Lemma \ref{lem:reduced2}(3). Then it is obvious.

\cs{Case (4)} $ w=rsr $. We have $ I=\{r,s\} $.

\css{Step I}  If $ w_{rs} \in D_{\leq N}$,  then $ \deg f_{u,v,w}\nt_{x_1}\nt_{w}\nt_{y_1} <N$. Assume in the rest of this step that $ L(w_{rs})>N $. Thus $u,v\neq w_{rs}$.

For cases \circled{1}\circled{4} of Lemma \ref{lem:reduced2}(4) we use Corollary \ref{cor:est}.

Consider case \circled{2} of Lemma \ref{lem:reduced2}(4). If $ ru<u$ or $ vr<v $, then $ w_{rs} \in D_{\leq N}$. By Lemma \ref{lem:fuvw}, it remains to consider the  case
$ u=rw_{rs}=v^{-1} $, $ \xi=\xi_t $. Then we have \begin{itemize}
\item $ \deg \xi_t(\xi_s^2\xi_r\nt_{x_2\cdot w_{rs}r\cdot t\cdot w_{rs}\cdot y_2}+\xi_s\xi_r\nt_{x_2\cdot w_{rs}r\cdot t\cdot sw_{rs}\cdot y_2}+\xi_s\xi_r\nt_{x_2\cdot w_{rs}s\cdot t\cdot rw_{rs}\cdot y_2})$ $<N  $ using Corollary \ref{cor:est},
\item $ 3L(s)\leq N $ with the equality holding only if $ x,y\in \Omega_{\geq N} $,
\item $ L(rt)\leq N $ with the equality holding only if $ x,y\in \Omega_{\geq N} $.
\end{itemize}

Consider case \circled{3} of Lemma \ref{lem:reduced2}(4). If $ vr<v $, then $ w_{rs} \in D_{\leq N}$. By Lemma \ref{lem:fuvw}, it remains to consider the cases
(i) $ u=sw_{rs}, v=w_{rs}r ,\xi=\xi_r$ and (ii)
$ u=rw_{rs}=v^{-1} $, $ \xi=\xi_t $. Then they are proved  as the above case \circled{2}.

\medskip
\css{Step II}  $w_I\in  (W_I)_{>N} $. Assume $\delta>0 $.
By Corollary \ref{cor:degfp2}, we have the following two cases:
\begin{itemize}
\item [(i)] $ ru<u $ or $ vr<v $, and $ \delta< 2L(r) $;
\item [(ii)] $ su<u $ and $ vs<v $, and $ \delta=L(s) $.
\end{itemize}

For cases \circled{1}\circled{4} of Lemma \ref{lem:reduced2}(4),  using  Corollary \ref{cor:est} to conclude that $ \deg f_{u,v,w}$ $\nt_{x_1}\nt_{w}\nt_{y_1} \leq N $ and the equality holds only if $ x,y\in \Omega_{\geq N} $.

For \circled{2} of Lemma \ref{lem:reduced2}(4), we can exclude (i). For (ii) we  using  Corollary \ref{cor:est} again.

For \circled{3} of Lemma \ref{lem:reduced2}(4), we must have $vs<v  $. By Lemma \ref{lem:degfp}, (i) can be refined as (i') : $ ru<u,vs<v ,\delta\leq L(r)$.  Then we apply Corollary \ref{cor:est}.

\css{Step III} $ \{d_I,w_I\} \subset(W_I)_{>N}$. For $\gamma>0$, by Corollary \ref{cor:degnfp}, we have
$L(r)>L(s)$, $ru<u$, $vr<r$, $\gamma\leq L(s)$. Since $L(w_{rs})>N$, we must be in case \circled{1} of Lemma \ref{lem:reduced2}(4). Thus $\deg T_{x_{1}}T_{rs}T_{y_{1}}\leq L(r)$. We obtain
$$
\gamma+\deg\nt_{x_1}\nt_{w}\nt_{y_1}\leq L(rs)\leq N,
$$
The equality holds only if $ L(rt)= N$, which implies  $x,y\in\Omega_{\geq N} $.

\cs{Case (5)} $ w=sts $.
In this case, $  I=\{s,t\} $, $(x_{1},sts,y_{1})$ is a reduced extension of $ (w_{rs}s,sts,sw_{rs}) $, and $\deg T_{x_{1}}T_{sts}T_{y_{1}}=L(r)$. Since $ L(s)=L(t) $, by Corollary \ref{cor:degfp2} and Lemma \ref{lem:degfp}, we only need to consider Step I.

If $ su<u $ or $ vs<v $, then $ w_{rs} $ appears in $ x _1 u$ or $ vy_1 $, and hence
$$
\deg f_{u,v,w}\nt_{x_1}\nt_w\nt_{y_1}\leq  L(sts)+L(r)<L(w_{rs})\leq N.
$$
According to Lemma \ref{lem:fuvw}, it remains to prove that case $ u=v^{-1}=ts$, $ \xi=\xi_t $. In this case, we have \[
\deg \xi\nt_{x_1}\nt_w\nt_{y_1}\leq L(rt)\leq N,
\]and the equality holds only if $ x,y\in\Omega_{\geq N} $.

\cs{Case(6)} $ w=rt $.

In this case, $ I=\{r,t\} $, $ m_{rt}=2 $. Since $ rt\in (W)_{\leq N} $, we only need to do Step I.

Assume first that $ L(w_{rs}) \leq N$. We have \[
\deg f_{u,v,w}\nt_{x_1}\nt_w\nt_{y_1}\leq L(rt)+L(rs)=L(srsr)<N.
\]Then we are done. Assume in the rest that $ L(w_{rs})>N $.

Consider  \circled{1} of Lemma \ref{lem:reduced2}(6). If $ ru<u $, then $ w_{rs} $ appears in $ x_1u $. It remains to consider the case $ u=t,v=rt $, $ f_{u,v,w}=\xi_t $. Then we have $ \deg f_{u,v,w}\nt_{x_1}\nt_{rt}\nt_{y_1}\leq 2L(s)<N $.

Consider   \circled{2} of Lemma \ref{lem:reduced2}(6). If $ \deg f_{u,v,w} >0$ we have $ru<u  $ or $ vt<v $, and hence $ w_{rs} $  appears in $ x_1u $ or $v y_1 $.

Consider  \circled{3} of Lemma \ref{lem:reduced2}(6). If $ vt<v $, then $ w_{rs} $ appears in $ vy_1$.  It remains to consider the case $ u=rt, v=r , f_{u,v,w}=\xi_r$. Then we use Corollary \ref{cor:est} to obtain    $ \deg f_{u,v,w}\nt_{x_1}\nt_w\nt_{y_1}<N $.

\cs{Case(7)} $ w=st $.
In this case, $ I=\{s,t\} $. Since $ L(s)=L(t) $, we only need to consider Step I and Step II.

\css{Step I}
Assume first that $ L(w_{rs}) \leq N$. We have \[
\deg f_{u,v,w}\nt_{x_1}\nt_w\nt_{y_1}\leq L(st)+L(rsr)=L(srsrs)<N.
\]Then we are done. Assume in the rest of this step that $ L(w_{rs})>N $.

Consider \circled{1} of  Lemma \ref{lem:reduced2} (7). According to Lemma \ref{lem:fuvw2}, if $ \deg\ f_{u,v,w}> 0 $, then $ u=w_{st} $ or $ vt<v$.  Thus $ sts\in (W)_{\leq N} $. Therefore
$$
\deg f_{u,v,w}\nt_{x_1}\nt_w\nt_{y_1}\leq L(st)+L(s)=L(sts)\leq N.
$$
The equality holds only if $ L(sts)= N$ and $u=v=sts$, which implies $ x,y\in\Omega_{\geq N} $.

Consider \circled{2}\circled{3} of Lemma \ref{lem:reduced2}(7). If $ vt<v $, then $ w_{rs} $ appears in $ y $. Since we assume $ L(w_{rs})>N $, by Lemma \ref{lem:fuvw2}, it remains to consider the case of $ u=sts $, $ v=ts $, $ f_{u,v,w}=\xi_t $. Then we conclude $ \deg  f_{u,v,w}\nt_{x_1}\nt_w\nt_{y_1}<N $ using Corollary \ref{cor:est} and $ 2L(s)<N $.

Consider \circled{4}\circled{5}\circled{6} of Lemma \ref{lem:reduced2}(7). If $ su<u $, then $ w_{rs} $ appears in $ x $. Since we assume $ L(w_{rs})>N $, by Lemma \ref{lem:fuvw2}, it remains to consider the case of $ u=ts $, $ v=sts $, $ f_{u,v,w}=\xi_s $. Then we conclude $ \deg  f_{u,v,w}\nt_{x_1}\nt_w\nt_{y_1}<N $ using Corollary \ref{cor:est} and $ 2L(s)<N $. (In the case of \circled{5}, we also need $ L(rt)\leq N $ and the equality holds only if $ y\in \Omega_{\geq N} $.)

Consider  \circled{7}\circled{8}  of Lemma \ref{lem:reduced2}(7).
If $ \deg f_{u,v,w}>0 $, then by Lemma \ref{lem:fuvw2}, we have $ u=w_{st} $ or $ vt<v $, and then $ w_{rs} $ always appears in $ x_1u $ or $ vy_1 $ in these two cases.

\css{Step II} Since $ L(s)=L(t) $, by Lemma \ref{lem:degfp2}, we always have $ \delta\leq 0 $.

\cs{Case (8)} $ w=rs $.

In this case, we have $ I=\{r,s\} $ and by Corollary \ref{cor:degnfp2}, we only need to do Step I and Step II. By Lemma \ref{lem:reduced2}, we have $\deg T_{x_{1}}T_{rs}T_{y_{1}}\leq L(rsrs)$.

\css{Step I}
Assume first that $ L(w_{rs}) \leq N$. We have \[
\deg f_{u,v,w}\nt_{x_1}\nt_w\nt_{y_1}\leq L(rs)+L(rsrs)=L(rsrsrs)<N.
\]Then we are done. Assume in the rest of this step that $ L(w_{rs})>N $, thus $u,v\neq w_{rs}$. According to Lemma \ref{lem:fuvw2}, if $\deg f_{u,v,rs}>0$, then $ru<u$ and $vs<v$,
so we are in case \circled{2}\circled{3}\circled{6}\circled{7}\circled{8} of Lemma \ref{lem:reduced2}(8), and we have $L(rt)\leq N$ and $L(sts)\leq N$. If $m_{rs}$ is odd, then we obtain
$$\deg f_{u,v,rs}\nt_{x_1}\nt_{w}\nt_{y_1}<L(s)+2L(s)=L(sts)\leq N.$$
If $m_{rs}$ is even, then we are in case \circled{2}\circled{3}\circled{6}\circled{7} of Lemma \ref{lem:reduced2}(8). Here we only give the proof of case \circled{7} and the other cases are similar. In this case, we assume $x_{1}=x_{2}\cdot w_{rs}sr\cdot t$, $y_{1}=tsrst\cdot y_{2}$ for some $x_{2},y_{2}\in W$ with $\mathcal{R}(x_{2})\subseteq\{t\}$, $\mathcal{L}(y_{2})\subseteq\{r\}$, then
$$
\begin{aligned}
&\ \ \ \ \nt_{x_{2}\cdot w_{rs}sr\cdot t}\nt_{rs}\nt_{tsrst\cdot y_{2}}\\
&=\xi_{s}\nt_{x_{2}\cdot w_{rs}\cdot tsrst\cdot y_{2}}+\xi_{r}\xi_{s}\nt_{x_{2}\cdot w_{rs}\cdot ts\cdot y_{2}}+\xi_{r}\nt_{x_{2}\cdot w_{rs}s\cdot ts\cdot y_{2}}\\
&\ \ \ +\xi_{s}\nt_{x_{2}\cdot w_{rs}r\cdot ts\cdot y_{2}}+\nt_{x_{2}\cdot w_{rs}rs\cdot ts\cdot y_{2}}\\
&=\xi_{s}\nt_{x_{2}\cdot w_{rs}\cdot tsrst\cdot y_{2}}+\xi_{r}\xi_{s}(-\sum_{z<w_{rs},z\neq w_{rs}s,w_{rs}r}p_{z,w_{rs}}\nt_{x_{2}\cdot z}\nt_{ts\cdot y_{2}})\\
&\ \ \ +q^{-2L(s)}\xi_{r}\nt_{x_{2}\cdot w_{rs}s\cdot ts\cdot y_{2}}+q^{-2L(r)}\xi_{s}\nt_{x_{2}\cdot w_{rs}r\cdot ts\cdot y_{2}}+\nt_{x_{2}\cdot w_{rs}rs\cdot ts\cdot y_{2}}.
 \end{aligned}
$$
It is easy to see
$$\deg f_{u,v,rs}(\xi_{s}\nt_{x_{2}\cdot w_{rs}\cdot tsrst\cdot y_{2}}+q^{-2L(r)}\xi_{s}\nt_{x_{2}\cdot w_{rs}r\cdot ts\cdot y_{2}}+\nt_{x_{2}\cdot w_{rs}rs\cdot ts\cdot y_{2}})<N.$$
By Corollary \ref{cor:est}, we have
$$\deg f_{u,v,rs}\xi_{r}\xi_{s}(-\sum_{z<w_{rs}, z\neq w_{rs}s,w_{rs}r}p_{z,w_{rs}}\nt_{x_{2}\cdot z}\nt_{ts\cdot y_{2}})<N.$$
If $L(s)\geq L(r)$, then $\deg q^{-2L(s)}f_{u,v,rs}\xi_{r}\nt_{x_{2}\cdot w_{rs}s\cdot ts\cdot y_{2}}\leq 0<N$. If $L(r)>L(s)$, then $d_{I}$ appears in $x_{1}$, thus
$$\deg q^{-2L(s)}f_{u,v,rs}\xi_{r}\nt_{x_{2}\cdot w_{rs}s\cdot ts\cdot y_{2}}\leq 2L(r)-2L(s)<L'(d_{I})\leq N.$$

\css{Step II} $ w_I\in (W_I)_{>N} $. Let $ \delta>0 $. We have $m_{rs}=2m$ is even, $L(r)\neq L(s)$ and $u=v=d_{I}$, and $ \delta=|L(r)-L(s)| $ by Lemma \ref{lem:degfp2}.

We only need to consider cases \circled{1}--\circled{7}, \circled{9}, \circled{11}--\circled{14} of Lemma \ref{lem:reduced2}(8). Since $ m_{rs} $ is even, we exclude \circled{8}\circled{10}\circled{15}\circled{16}. We exclude \circled{17} because $ w_I $ appears in $ x_1u $ or $ vy_1 $.

If we are in cases \circled{2}\circled{3}\circled{5}\circled{6}\circled{7} and $L(r)>L(s)$, then we have $\deg T_{x_{1}}T_{rs}$ $T_{y_{1}}\leq L(r)$ and $\delta=L(r)-L(s)$.  Hence
\[\delta+ \deg\nt_{x_1}\nt_{w}\nt_{y_1}\leq 2L(r)-L(s)< L'(d_I)\leq N.\]

If we are in case \circled{4} and $L(r)>L(s)$, then $x_{1}=x_{2}\cdot w_{rs}rsr\cdot t$, $y_{1}=tsrst\cdot sw_{rs}\cdot y_{2}$ for some $x_{2},y_{2}\in W$
with $\mathcal{R}(x_{2}),\mathcal{L}(y_{2})\subseteq\{t\}$, and
$$
\begin{aligned}
&\ \ \ \ \nt_{x_{2}\cdot w_{rs}rsr\cdot t}\nt_{rs}\nt_{tsrst\cdot sw_{rs}\cdot y_{2}}\\
&=\xi_{t}\nt_{x_{2}\cdot w_{rs}r\cdot tsrst\cdot sw_{rs}\cdot y_{2}}+\xi_{s}\xi_{r}\nt_{x_{2}\cdot w_{rs}r\cdot t\cdot w_{rs}\cdot y_{2}}\\
            &\ \ \ +\xi_{s}\nt_{x_{2}\cdot w_{rs}sr\cdot t\cdot rw_{rs}\cdot y_{2}}+\xi_{r}\nt_{x_{2}\cdot w_{rs}sr\cdot t\cdot w_{rs}\cdot y_{2}}+\nt_{x_{2}\cdot w_{rs}sr\cdot t\cdot rw_{rs}\cdot y_{2}}\\
&=\xi_{t}\nt_{x_{2}\cdot w_{rs}r\cdot tsrst\cdot sw_{rs}\cdot y_{2}}+\xi_{s}\xi_{r}(-\sum_{z<w_{rs}}p_{z,w_{rs}}\nt_{x_{2}\cdot w_{rs}r\cdot t}\nt_{z\cdot y_{2}})\\
            &\ \ \ +\xi_{s}\nt_{x_{2}\cdot w_{rs}sr\cdot t\cdot rw_{rs}\cdot y_{2}}+\xi_{r}\nt_{x_{2}\cdot w_{rs}sr\cdot t\cdot w_{rs}\cdot y_{2}}+\nt_{x_{2}\cdot w_{rs}sr\cdot t\cdot rw_{rs}\cdot y_{2}}.\\
 \end{aligned}
$$
By  Corollary \ref{cor:est}, we have
$$\deg \left (\sum_{z<w_{rs}}p_{z,w_{rs}}\nt_{x_{2}\cdot w_{rs}r\cdot t}\nt_{z\cdot y_{2}}\right )\leq \max\{-L(rs),L(r)-L(w_{rs})+N\}.$$
Then we conclude that $\delta+\deg \nt_{x_1}\nt_{w}\nt_{y_1}<N$.

If we are in case \circled{9} and $L(r)>L(s)$,  then $x_{1}=x_{2}\cdot w_{rs}sr\cdot t$, $y_{1}=tsrst\cdot sw_{rs}\cdot y_{2}$ for some $x_{2},y_{2}\in W$
with $\mathcal{R}(x_{2}),\mathcal{L}(y_{2})\subseteq\{t\}$, and
$$
\begin{aligned}
&\ \ \ \ \nt_{x_{2}\cdot w_{rs}sr\cdot t}\nt_{rs}\nt_{tsrst\cdot sw_{rs}\cdot y_{2}}\\
&=\xi_{t}\nt_{x_{2}\cdot w_{rs}\cdot tsrst\cdot sw_{rs}\cdot y_{2}}+\xi_{r}^{2}\xi_{s}\nt_{x_{2}\cdot w_{rs}r\cdot t\cdot w_{rs}\cdot y_{2}}+\xi_{r}\xi_{s}\nt_{x_{2}\cdot w_{rs}r\cdot t\cdot rw_{rs}\cdot y_{2}}\\
&\ \ \ +\xi_{r}^{2}\nt_{x_{2}\cdot w_{rs}sr\cdot t\cdot w_{rs}\cdot y_{2}}+\xi_{r}\nt_{x_{2}\cdot w_{rs}sr\cdot t\cdot rw_{rs}\cdot y_{2}}+\xi_{s}\nt_{x_{2}\cdot w_{rs}r\cdot t\cdot w_{rs}\cdot y_{2}}\\
&\ \ \ +\xi_{r}\nt_{x_{2}\cdot w_{rs}rs\cdot t\cdot rw_{rs}\cdot y_{2}}+\nt_{x_{2}\cdot w_{rs}rsr\cdot t\cdot rw_{rs}\cdot y_{2}}\\
&=\xi_{t}\nt_{x_{2}\cdot w_{rs}\cdot tsrst\cdot sw_{rs}\cdot y_{2}}+\xi_{r}^{2}\xi_{s}(-\sum_{z<w_{rs},z\neq rw_{rs}}p_{z,w_{rs}}\nt_{x_{2}\cdot w_{rs}r\cdot t}\nt_{z\cdot y_{2}})\\
&\ \ \ +q^{-2L(r)}\xi_{r}\xi_{s}\nt_{x_{2}\cdot w_{rs}r\cdot t\cdot rw_{rs}\cdot y_{2}}+\xi_{r}^{2}(-\sum_{z<w_{rs},z\neq rw_{rs}}p_{z,w_{rs}}\nt_{x_{2}\cdot w_{rs}sr\cdot t}\nt_{z\cdot y_{2}})\\
&\ \ \ +q^{-2L(r)}\xi_{r}\nt_{x_{2}\cdot w_{rs}sr\cdot t\cdot rw_{rs}\cdot y_{2}}+\xi_{s}\nt_{x_{2}\cdot w_{rs}r\cdot t\cdot w_{rs}\cdot y_{2}}\\
&\ \ \ +\xi_{r}\nt_{x_{2}\cdot w_{rs}rs\cdot t\cdot rw_{rs}\cdot y_{2}}+\nt_{x_{2}\cdot w_{rs}rsr\cdot t\cdot rw_{rs}\cdot y_{2}}.
 \end{aligned}
$$
Then, as the case of \circled{4}, we conclude that $\delta+\deg \nt_{x_1}\nt_{w}\nt_{y_1}<N$.

If we are in cases \circled{2}--\circled{7}, \circled{9}, and $L(s)>L(r)$, then we only need to consider \circled{2}\circled{3}\circled{6}\circled{7}; other cases are excluded using $w_{rs}\in D_{>N}$. We have $\deg T_{x_{1}}T_{rs}T_{y_{1}}\leq L(rs)$,  $\delta=L(s)-L(r)$, and $ sts $ appears in $ vy_1 $. Thus
$$\delta +\deg  \nt_{x_1}\nt_{w}\nt_{y_1}\leq 2L(s)<L(sts)\leq N.$$

If we are in case \circled{1} or \circled{11}-\circled{14}, then we must have $L(s)>L(r)$; otherwise, $ w_{rs} \in D_{>N}$ appears in $ x_1v \in \Omega_{\leq N}$, a contradiction. We have $\deg T_{x_{1}}T_{rs}T_{y_{1}}\leq L(srs)$ and $\delta=L(s)-L(r)$. Then
$$\delta+\deg  \nt_{x_1}\nt_{w}\nt_{y_1}\leq 3L(s)\leq N.$$
The equality holds only if we are in case \circled{13} or \circled{14}, $L(sts)=N$ and $ sts $ appears in $x$ and $ y$, which implies $ x,y\in \Omega_{\geq N} $.

This completes the proof of Proposition \ref{prop:bound}.

\section{A  property of length adding}\label{sec:add}

\begin{lem}\label{lem:length}
For any $ d\in D_N $, $ x\in B_d $, $ y\in U_d $, we have\[
l(xdy)=l(x)+l(d)+l(y).
\]
\end{lem}
\begin{proof}
If $d=e $, then $N=0$ and $x=y=e$, the result is obvious. If $d\in S$, we have $d,xd,dy\in \Omega_{N}$ and $x\in \Omega_{<N}$. Thus, $x\in W_{S\setminus\{d\}}$ and $\mathcal{L}(dy)=\{d\}$. Then we get $l(xdy)=l(x)+l(d)+l(y)$.

Now suppose $ l(d)\geq 2 $. Then $d=w_{J}$ for some $J\subseteq S$ with $|J|=2$, or $d=r_{2}w_{r_{1}r_{2}}$ for some $r_{1},r_{2}\in S$ with $m_{r_{1}r_{2}}\geq 4$ and $L(r_{1})>L(r_{2})$.
We have  following 3 cases.

\cs{Case (1)} $\infty=m_{rs}>m_{st}\geq3$. We have $d=rt$ or $w_{st}$ or $sw_{st}$ or $tw_{st}$. Then by Lemma \ref{lem:a0}(5)(6),
we have $l(xdy)=l(x)+l(d)+l(y)$.

\cs{Case (2)} $4\leq m_{rs}, m_{st}< \infty$ with $(m_{rs},m_{st})\neq (4,4)$. Without loss of generality, we assume $m_{rs}\geq m_{st}$. Thus $m_{rs}\geq 5$. By  Lemma \ref{lem:a1}(10)(11)(15), we only need to verify the case of $d=rt$ and the case of $d=sts$ with $m_{st}=4$ and $L(s)>L(t)$. If $d=rt$ and $l(xdy)<l(x)+l(d)+l(y)$, then by Lemma \ref{lem:reduced}(3), one of
$ \{\mathcal{R}(xr),\mathcal{L}(ty)\}$ and $ \{\mathcal{R}(xt),\mathcal{L}(ry)\}$ is $\{ \{r,s\}, \{s,t\}\}$. Thus $L(w_{rs})\leq N$ and $L(w_{st})\leq N$, so $L(rt)<N$, which contradicts with $rt\in D_{N}$. If $d=sts$, $m_{st}=4$, $L(s)>L(t)$ and $l(xdy)<l(x)+l(d)+l(y)$, by Lemma \ref{lem:reduced}(1),  we have $ \mc{R}(xs)=\mc{L}(sy)=\{r,s\} $ and hence $ L(w_{rs})\leq N$, which contradicts with $sts\in D_{N}$.

\cs{Case (3)} $\infty>m_{rs}\geq 7, m_{st}=3$. Note that $L(s)=L(t)  $. By Lemma \ref{lem:a2}(6), we only need to verify the case of $d=rt$ and the case of $d=sts$. If $d=rt$ and $l(xdy)<l(x)+l(d)+l(y)$, then by Lemma  \ref{lem:reduced2}(6), $w_{rs}$ appears in $ xd $ or $ dy $, so $ L(w_{rs})\leq N $, which contradicts with $rt\in D_{N}$. If $d=sts$ and $l(xdy)<l(x)+l(d)+l(y)$, by Lemma  \ref{lem:reduced2}(5), $ w_{rs} $ appears in $ xd$ and $dy $, and hence $ L(w_{rs})\leq N $, which contradicts with $sts\in D_{N}$.
\end{proof}

\begin{rem}
From the above proof, one can see that  if $ d\in D_N $ with $ l(d)\geq 2$, and $ x\in U_d^{-1} $, $ y\in U_d $, then we have \[
l(xdy)=l(x)+l(d)+l(y).
\]
For $ l(d) =1$, it is easy to find a counterexample.
\end{rem}

\section{Estimation of degrees}\label{sec:est}

The aim of this section is to prove a stronger version of Conjecture \ref{conj:dim2}(3) for the Coxeter groups that are listed in section \ref{sec:intr}.
\begin{prop}\label{prop:strictineq}
For $ d\in D_N $, $ x\in U_d^{-1} $, $ y\in U_d $, $ w\leq d $, we have\[
\deg (\nt_{x}\nt_w\nt_y)\leq -\deg p_{w,d}.
\]
If furthermore $ w<d $ and one of $ x\in B_d , y\in B_d^{-1} , l(w)\geq 2$ holds, then
$$  \deg (\nt_{x}\nt_w\nt_y)< -\deg p_{w,d}.  $$

\end{prop}

The proof will occupy the rest of this section.

If $ w=e $, then $-\deg p_{w,d}=N$ and it follows from Proposition \ref{prop:bound}.

Assume that $ w=r'$ for some $ r'\in S $. By \eqref{eq:bound}, $\deg (\nt_{xr'}\nt_{r'y})\leq N $ and $\deg (\nt_{x}\nt_{r'y})\leq N $. Since $T_{xr'}T_{r'y}=\xi_{r'}T_{xr'}T_{y}+T_{x}T_{y}$,
 we have $\deg (\xi_{r'}\nt_{xr'}\nt_y) \leq N $. This implies $$ \deg (\nt_{xr'}\nt_y) \leq N -L(r')\leq -\deg p_{r',d}. $$
 Assume furthermore $ w< d $ and $ x\in B_d $  (the proof is similar for $y\in B_d^{-1}$). If $ r'd<d $, then we have $ x,xr'\in\Omega_{<N} $ since $r'<d$, and hence by Proposition \ref{prop:bound}, $\deg (\nt_{xr'}\nt_{r'y})< N $ and $\deg (\nt_{x}\nt_{r'y})< N $. Then we have $$\deg (\nt_{xr'}\nt_y) < N -L(r')\leq-\deg p_{r',d}.$$
 If $ r'd>d $, then $ d=d_I$ with $ I=\{r',s'\}$ and $ L(s')>L(r') $ for some $s'\in S$. Then
 $\deg (\nt_{xr'}\nt_y) \leq N -L(r')<N+L(r')= -\deg p_{r',d}$ by Lemma \ref{lem:aaa}. This proves the proposition for $ l(w)= 1 $.

If $ l(xwy)=l(x)+l(w)+l(y) $, then  $ \deg(\nt_{x}\nt_w\nt_y)\leq 0$, and the proposition follows since $ \deg p_{w,d}<0 $ when $ w<d $.

In the remainder we assume  $l(d)\geq l(w)\geq 2 $ and $ l(xwy)<l(x)+l(w)+l(y) $, and prove the strict inequality in the proposition.

 If $ m_{rs}=\infty $ and $ 3\leq m_{st} <\infty$, then by Lemma \ref{lem:reduced0} nothing needs to prove.

\subsection{The case of $4\leq m_{rs}, m_{st}< \infty$ with $(m_{rs},m_{st})\neq (4,4)$} According to Lemma \ref{lem:reduced}, the proof is divided into the following cases.

We omit the transpose cases and the ones with  $r$, $t$ exchanged.

\cs{Case (1)} $ w=srs $,  $ x=x_2\cdot w_{st} s$, $ y=sw_{st}\cdot y_2 $ for some $x_2,y_2\in W$, and
		$$
		T_{x}T_{srs}T_{y}=\xi_{t}T_{x_2\cdot w_{st}\cdot r\cdot tw_{st}\cdot y_2}+T_{x_2\cdot w_{st}t\cdot r\cdot tw_{st}\cdot y_2}.
		$$

Assume $ d=w_{rs} $. Since $ xd\in \Omega_{N} $, we have $ L(w_{st})\leq N=L(w_{rs})$. We obtain
$$\deg (\nt_{x}\nt_w\nt_y)\leq L(t)< L(w_{rs})-L(srs)=-\deg p_{w,d}.$$
The second inequality follows from \eqref{eq:frequent44}.

Assume $ d=d_{I} $ with $ I=\{r,s\} $. If $ L(r)>L(s) $,  then $ xd\in \Omega_{N} $ implies  that $ L(rt)\leq  N=L'(d_{I})$. Thus
$$\deg (\nt_{x}\nt_w\nt_y)\leq L(t)< L(rt)-L'(srs)\leq L'(d_{I})-L'(srs)=-\deg p_{w,d}.$$
If $ L(s)>L(r) $, then $ xd\in \Omega_{N} $ implies that $ L(w_{st})\leq  N=L'(d_{I})$. Thus
$$\deg (\nt_{x}\nt_w\nt_y)\leq L(t)< L(w_{st})-L'(srs)\leq L'(d_{I})-L'(srs)=-\deg p_{w,d}.$$

\cs{Case (2)} $ w=rsr $, $ m_{st}=4 $, $x=x_2\cdot w_{rs}r\cdot t$ and $y=t\cdot rw_{rs}\cdot y_2$ for some $x_2,y_2\in W$, and
		$$
		T_{x}T_{rsr}T_{y}=\xi_{s}T_{x_2\cdot w_{rs}\cdot tst\cdot sw_{rs}y_2}+T_{x_2\cdot w_{rs}s\cdot tst\cdot sw_{rs}\cdot y_2}.
		$$

Assume $ d=w_{rs} $. Then we have $L(w_{rs})=N$ with $m_{rs}\geq 5$. We obtain
$$\deg (\nt_{x}\nt_w\nt_y)\leq L(s)< L(w_{rs})-L(rsr)=-\deg p_{w,d}.$$

Assume $ d=d_{I} $ with $ I=\{r,s\} $ and $m_{rs}\geq 6$. Since $L(w_{rs})>L'(d_{I})=N$, we must have $ L(s)>L(r) $. Then
$$\deg (\nt_{x}\nt_w\nt_y)\leq L(s)< L'(d_{I})-L'(rsr)=-\deg p_{w,d}.$$

\cs{Case (3)} $ w=rt $, $x=x_2\cdot w_{rs}r$, $y=tw_{st}\cdot y_2$ for some $x_2,y_2\in W$, and
		$$
		T_{x}T_{rt}T_{y}=\xi_{s}T_{x_2\cdot w_{rs}\cdot sw_{st}\cdot y_2}+T_{x_2\cdot w_{rs}s\cdot sw_{st}\cdot y_2}.
		$$
We  must have $ d=w=rt $. Then $ xd,dy\in\Omega_{N} $ imply that $L(w_{rs}),L(w_{st})\leq N $. Thus $ L(rt)<N$, a contradiction with $ d\in D_{N} $. This case cannot happen.

\cs{Case (4)} $ w=rs $. There are  5 sub-cases as follows.

\noindent\circled{1} $x=x_2\cdot w_{rs}r\cdot t$, $y=sw_{st}\cdot y_2$ for some $x_2,y_2\in W$, and
			$$
            \begin{aligned}
            T_{x}T_{rs}T_{y}&=\xi_{t}\xi_{s}T_{x_2\cdot w_{rs}s\cdot w_{st}\cdot y_2}+\xi_{t}T_{x_2\cdot w_{rs}s\cdot sw_{st}\cdot y_2}\\
            &\ \ \ +\xi_{s}T_{x_2\cdot w_{rs}\cdot stw_{st}\cdot y_2}+T_{x_2\cdot w_{rs}s\cdot stw_{st}\cdot y_2}.
            \end{aligned}
            $$

Assume $ d=w_{rs} $. Then $dy\in\Omega_{N} $ implies that $L(w_{st})\leq N=L(w_{rs}) $. Thus
$$\deg (\nt_{x}\nt_w\nt_y)\leq L(st)< L(w_{rs})-L(rs)=-\deg p_{w,d}.$$

Assume $ d=d_{I} $ with $ I=\{r,s\} $. Since $L(w_{rs})>L'(d_{I})=N$, we must have $ L(s)>L(r) $. Then $dy\in\Omega_{N} $ implies that  $ L(w_{st})\leq N=L'(d_I)$. Hence
$$\deg (\nt_{x}\nt_w\nt_y)\leq L(st)<L(w_{st})-L'(rs)\leq L'(d_{I})-L'(rs)=-\deg p_{w,d}.$$
\noindent\circled{2} $m_{st}=4$, $x=x_2\cdot w_{rs}r\cdot t$, $y=sw_{st}\cdot sw_{rs}\cdot y_2$ for some $x_2,y_2\in W$, and
			$$
            \begin{aligned}
            T_{x}T_{rs}T_{y}&=\xi_{t}\xi_{s}T_{x_2\cdot w_{rs}s\cdot w_{st}\cdot sw_{rs}\cdot y_2}+\xi_{t}T_{x_2\cdot w_{rs}s\cdot sw_{st}\cdot sw_{rs}\cdot y_2}\\
            &\ \ \ +\xi_{s}\xi_{r}T_{x_2\cdot w_{rs}r\cdot t\cdot w_{rs}\cdot y_2}+\xi_{s}T_{x_2\cdot w_{rs}r\cdot t\cdot rw_{rs}\cdot y_2}\\
            &\ \ \ +\xi_{r}T_{x_2\cdot w_{rs}sr\cdot t\cdot w_{rs}\cdot y_2}+T_{x_2\cdot w_{rs}sr\cdot t\cdot rw_{rs}\cdot y_2}.
            \end{aligned}
            $$
We have  $m_{rs}\geq 5$.
Assume $ d=w_{rs} $. Since $ dy\in \Omega_{N} $, we have $ L(w_{st})\leq N=L(w_{rs})$. Then we obtain
$$\deg (\nt_{x}\nt_w\nt_y)\leq \max \{L(st),L(sr)\}< L(w_{rs})-L(rs)=-\deg p_{w,d}.$$

Assume $ d=d_{I} $ with $ I=\{r,s\} $. Then $ xd,dy\in \Omega_{N} $ imply that $ L(w_{rs})\leq N $, thus $L'(d_{I})<N$, contradict with $ d\in D_{N} $. So this case cannot happen.

\noindent\circled{3} $x=x_2\cdot t$, $y=sw_{st}\cdot y_2$ for some $x_2,y_2\in W$, and
			$$
            T_{x}T_{rs}T_{y}=\xi_{t}T_{x_2\cdot r\cdot w_{st}\cdot y_2}+T_{x_2\cdot r\cdot tw_{st}\cdot y_2}.
            $$

Assume $ d=w_{rs} $. Since $ dy\in \Omega_{N} $, we have $ L(w_{st})\leq N=L(w_{rs})$. Then we obtain
$$\deg (\nt_{x}\nt_w\nt_y)\leq L(t)< L(w_{rs})-L(rs)=-\deg p_{w,d}.$$
The second inequality follows from \eqref{eq:frequent44}.

Assume $ d=d_{I} $ with $ I=\{r,s\} $. If $ L(r)>L(s) $, then $ xd\in \Omega_{N} $ implies that $ L(rt)\leq N=L'(d_{I}) $, so we obtain
$$\deg (\nt_{x}\nt_w\nt_y)\leq L(t)<L(rt)-L'(rs)\leq L'(d_{I})-L'(rs)=-\deg p_{w,d}.$$
If $ L(s)>L(r) $, then $ dy\in \Omega_{N} $ implies that $ L(w_{st})\leq N=L'(d_{I}) $. Thus
$$\deg (\nt_{x}\nt_w\nt_y)\leq L(t)<L(w_{st})-L'(rs)\leq L'(d_{I})-L'(rs)=-\deg p_{w,d}.$$
\noindent\circled{4} $m_{st}=4$, $x=x_2\cdot w_{rs}sr\cdot t$, $y=sw_{st}\cdot sw_{rs}\cdot y_2$ for some $x_2,y_2\in W$, and
			$$
T_{x}T_{rs}T_{y}=\xi_{t}T_{x_2\cdot w_{rs}s\cdot w_{st}\cdot sw_{rs}\cdot y_2}+\xi_{r}T_{x_2\cdot w_{rs}r\cdot t\cdot w_{rs}\cdot y_2}+T_{x_2\cdot w_{rs}r\cdot t\cdot rw_{rs}\cdot y_2}.
		    $$

We have $m_{rs}\geq 5$.
Assume $ d=w_{rs} $.  Then $ dy\in \Omega_{N} $ implies that $ L(w_{st})\leq N=L(w_{rs}) $, so we obtain
$$\deg (\nt_{x}\nt_w\nt_y)\leq \max \{L(r),L(t)\}< L(w_{rs})-L(rs)=-\deg p_{w,d}.$$

Assume $ d=d_{I} $ with $ I=\{r,s\} $. Then $ m_{rs}\geq 6 $. Since $L(w_{rs})>L'(d_{I})=N$, we must have $ L(r)>L(s) $. Then $ xd\in \Omega_{N} $ implies that $ L(rt)\leq N$, so we obtain
$$\deg (\nt_{x}\nt_w\nt_y)\leq \max \{L(r),L(t)\}< L'(d_{I})-L'(rs)=-\deg p_{w,d}.$$
\noindent\circled{5} $ x=x_2\cdot (w_{rs} r)\cdot t$, $ y=stw_{st}\cdot y_2 $ for some $x_2,y_2\in W$, and
			$$
		    T_{x}T_{rs}T_{y}=\xi_{s}T_{x_2\cdot w_{rs}\cdot sw_{st}\cdot y_2}+T_{x_2\cdot w_{rs}s\cdot sw_{st}\cdot y_2}.
		    $$

Assume $ d=w_{rs} $.  Then we have
$$\deg (\nt_{x}\nt_w\nt_y)\leq L(s)< L(w_{rs})-L(rs)=-\deg p_{w,d}.$$

Assume $ d=d_{I} $ with $ I=\{r,s\} $. Since $L(w_{rs})>L'(d_{I})=N$, we must have $ L(s)>L(r) $, so we obtain
$$\deg (\nt_{x}\nt_w\nt_y)<L(s)\leq L'(d_{I})-L'(rs)=-\deg p_{w,d}.$$

\subsection{The case of $ m_{rs}\geq 7, m_{st}=3$} According to Lemma \ref{lem:reduced2}, the proof  is  divided into the following cases. The transpose cases are omitted.

If $ d=w_{rs}\in D_N $, then $ \deg(\nt_x\nt_w\nt_y)<-\deg p_{w,d} $ follows directly from Corollary \ref{cor:est}(1). Thus in the following we assume $ d\neq w_{rs} $.

\cs{Case (1)} $ w=rsrsr $.

%
We have $ d=d_{I} $ with $ I=\{r,s\} $.
Since $ L(w_{rs})> L'(d_I)=N $ and $ w_{rs} $ appears in $ xr $, we must have $ L(s)>L(r) $. Then $xd,dy\in \Omega_{N}$ imply  $L(sts)\leq N=L'(d_{I})$. Thus
$$\deg (\nt_{x}\nt_w\nt_y)\leq L(s)< L'(d_{I})-L'(rsrsr)=-\deg p_{w,d}.$$

\cs{Case (2)} $ w=rsrs $.

We have $ d=d_{I} $ with $ I=\{r,s\} $. As the last case we must have $L(s)>L(r)$. Then $xd,dy\in \Omega_{N}$ imply $L(sts)\leq N=L'(d_{I})$. Thus
$$\deg (\nt_{x}\nt_w\nt_y)\leq L(s)< L'(d_{I})-L'(rsrs)=-\deg p_{w,d}.$$

\cs{Case (3)} $ w=srs $.


We have $ d=d_{I} $ with $ I=\{r,s\} $. Then we are in case \circled{1} or \circled{2} of Lemma \ref{lem:reduced2}(3). If $L(r)>L(s)$, then
$$\deg (\nt_{x}\nt_w\nt_y)\leq \max \{L(r),L(s)\}< L'(d_{I})-L'(srs)=-\deg p_{w,d}.$$
If $L(s)>L(r)$, then we must be in case \circled{1} of Lemma \ref{lem:reduced2}(3) since $L(w_{rs})>L'(d_{I})=N$. Then $xd,dy\in \Omega_{N}$ imply $L(sts)\leq N=L'(d_{I})$. Thus
$$\deg (\nt_{x}\nt_w\nt_y)\leq L(s)< L'(d_{I})-L'(srs)=-\deg p_{w,d}.$$

\cs{Case (4)} $ w=rsr $.


We have $ d=d_{I} $ with $ I=\{r,s\} $. If $L(r)>L(s)$, then we must be in case \circled{1} of Lemma \ref{lem:reduced2}(4) since $L(w_{rs})>L'(d_{I})=N$.
We have $x=x'\cdot w_{rs}sr\cdot t$, $y=t\cdot rsw_{rs}\cdot y'$ for some $x',y'\in W$ with $\mathcal{R}(x'),\mathcal{L}(y')\subseteq\{t\}$, and
$$
            \begin{aligned}
            \nt_{x}\nt_w\nt_y&=\xi_{r}\nt_{x'\cdot w_{rs}\cdot t\cdot rw_{rs}\cdot y'}+\nt_{x'\cdot w_{rs}r\cdot t\cdot rw_{rs}\cdot y'}\\
            &=\xi_{r}(-\sum_{z<w_{rs}}p_{z,w_{rs}}\nt_{x'\cdot z}\nt_{t\cdot rw_{rs}\cdot y'})+\nt_{x'\cdot w_{rs}r\cdot t\cdot rw_{rs}\cdot y'}.
            \end{aligned}
            $$
Let $m_{rs}=2m$.
By Corollary \ref{cor:est} and the assumption $ L(r)>L(s) $, we have
$$
\begin{aligned}
\deg (\nt_{x}\nt_w\nt_y)&\leq \max \{2L(r)-(2m-1)L(s),0\}\\
&<(m-2)L(r)-(m-2)L(s)\\
&=L'(d_{I})-L'(rsr)\\
&=-\deg p_{w,d}.
 \end{aligned}
$$
If $L(s)>L(r)$, then $xd,dy\in \Omega_{N}$ imply $L(sts)\leq N=L'(d_{I})$, and hence
$$\deg (\nt_{x}\nt_w\nt_y)\leq L(srs)< L'(d_{I})-L'(rsr)=-\deg p_{w,d}.$$

\cs{Case (5)} $ w=sts $.

We must have $ d=w=sts $. Then $ xd,dy\in\Omega_{N} $ imply that $L(w_{rs})\leq N $. Thus $ L(sts)<N$,  a contradiction with $ d\in D_{N} $. This case cannot happen.

\cs{Case (6)} $ w=rt $.

We must have $ d=w=rt $. Then $ xd,dy\in\Omega_{N} $ imply that $L(w_{rs})\leq N $. Thus $ L(rt)<N$, a contradiction with $ d\in D_{N} $. This case cannot happen.

\cs{Case (7)} $ w=st $.

We must have $ d=sts $. Since $L(w_{rs})>L(sts)=N$, we must be in case \circled{1} of Lemma \ref{lem:reduced2}(7). Thus
$$\deg (\nt_{x}\nt_w\nt_y)<L(s)=L(sts)-L(st)=-\deg p_{w,d}.$$
The first inequality is due to $ \deg( \nt_{x'\cdot w_{rs}\cdot ts\cdot y'})<0 $.

\cs{Case (8)} $ w=rs $.


We have  $ d=d_{I} $ with $ I=\{r,s\} $. If $L(r)>L(s)$, then we must be in cases \circled{2}\circled{3}\circled{4}\circled{5}\circled{6}\circled{7}\circled{9} of Lemma \ref{lem:reduced2}(8) since $L(w_{rs})>L'(d_{I})=N$ and $ m_{rs}\neq 7 $.
Here we only give details for case \circled{9} and the other cases are similar. In this case,  $x=x'\cdot w_{rs}sr\cdot t$, $y=tsrst\cdot sw_{rs}\cdot y'$ for some $x',y'\in W$ with $\mathcal{R}(x'),\mathcal{L}(y')\subseteq\{t\}$, and
$$
\begin{aligned}
			\nt_{x}\nt_{w}\nt_{y}&=\xi_{t}\nt_{x'\cdot w_{rs}\cdot tsrst\cdot sw_{rs}\cdot y'}+\xi_{r}^{2}\xi_{s}\nt_{x'\cdot w_{rs}r\cdot t\cdot w_{rs}\cdot y'}+\xi_{r}\xi_{s}\nt_{x'\cdot w_{rs}r\cdot t\cdot rw_{rs}\cdot y'}\\
&\ \ \ +\xi_{r}^{2}\nt_{x'\cdot w_{rs}sr\cdot t\cdot w_{rs}\cdot y'}+\xi_{r}\nt_{x'\cdot w_{rs}sr\cdot t\cdot rw_{rs}\cdot y'}+\xi_{s}\nt_{x'\cdot w_{rs}r\cdot t\cdot w_{rs}\cdot y'}\\
&\ \ \ +\xi_{r}\nt_{x'\cdot w_{rs}rs\cdot t\cdot rw_{rs}\cdot y'}+\nt_{x'\cdot w_{rs}rsr\cdot t\cdot rw_{rs}\cdot y'}\\
&=\xi_{t}\nt_{x'\cdot w_{rs}\cdot tsrst\cdot sw_{rs}\cdot y'}+\xi_{r}^{2}\xi_{s}(-\sum_{z<w_{rs},z\neq rw_{rs}}p_{z,w_{rs}}\nt_{x'\cdot w_{rs}r\cdot t}\nt_{z\cdot y'})\\
&\ \ \ +q^{-2L(r)}\xi_{r}\xi_{s}\nt_{x'\cdot w_{rs}r\cdot t\cdot rw_{rs}\cdot y'}+\xi_{r}^{2}(-\sum_{z<w_{rs}}p_{z,w_{rs}}\nt_{x'\cdot w_{rs}sr\cdot t}\nt_{z\cdot y'})\\
&\ \ \ +\xi_{r}\nt_{x'\cdot w_{rs}sr\cdot t\cdot rw_{rs}\cdot y'}+\xi_{s}\nt_{x'\cdot w_{rs}r\cdot t\cdot w_{rs}\cdot y'}\\
&\ \ \ +\xi_{r}\nt_{x'\cdot w_{rs}rs\cdot t\cdot rw_{rs}\cdot y'}+\nt_{x'\cdot w_{rs}rsr\cdot t\cdot rw_{rs}\cdot y'}.
 \end{aligned}
$$
Let $m_{rs}=2m$ for some $m\geq 4$. Using Corollary \ref{cor:est} and the assumption $ L(r)>L(s)=L(t) $, we have
$$
\begin{aligned}
\deg (\nt_{x}\nt_w\nt_y)&\leq \max \{3L(r)-(2m-2)L(s),L(r)\}\\
&<(m-1)L(r)-(m-2)L(s)\\
&=L'(d_{I})-L'(rs)\\
&=-\deg p_{w,d}.
 \end{aligned}
$$
If $L(s)>L(r)$, then we must be in cases \circled{1}\circled{2}\circled{3}\circled{6}\circled{7}\circled{11}\circled{12}\circled{13}\circled{14} of Lemma \ref{lem:reduced2}(8) since $L(w_{rs})>L'(d_{I})=N$ and $ m_{rs} \neq 7$. Here we only give details for case \circled{14} and the other cases are similar. In this case, $m_{rs}=8$. We assume $x=x'\cdot w_{rs}r\cdot t\cdot w_{rs}r\cdot t$, $y=tsrst\cdot srw_{rs}\cdot y'$ for some $x',y'\in W$ with $\mathcal{R}(x'),\mathcal{L}(y')\subseteq\{t\}$, then
\begin{align*}
&\ \ \ \ \nt_{x}\nt_{w}\nt_{y}\\
&=\xi_{t}\xi_{s}\nt_{x'\cdot w_{rs}r\cdot t\cdot w_{rs}\cdot tsrst\cdot srw_{rs}\cdot y'}+\xi_{t}\nt_{x'\cdot w_{rs}r\cdot t\cdot w_{rs}s\cdot tsrst\cdot srw_{rs}\cdot y'}\\
&\ \ \ +\xi_{s}^{2}\xi_{r}\nt_{x'\cdot w_{rs}r\cdot t\cdot w_{rs}\cdot t\cdot rw_{rs}\cdot y'}+\xi_{s}\xi_{r}\nt_{x'\cdot w_{rs}r\cdot t\cdot w_{rs}s\cdot t\cdot rw_{rs}\cdot y'}\\
&\ \ \ +\xi_{s}^{2}\nt_{x'\cdot w_{rs}r\cdot t\cdot w_{rs}r\cdot t\cdot rw_{rs}\cdot y'}+\xi_{s}\nt_{x'\cdot w_{rs}r\cdot t\cdot w_{rs}rs\cdot t\cdot rw_{rs}\cdot y'}\\
&\ \ \ +\xi_{r}\nt_{x'\cdot w_{rs}r\cdot t\cdot w_{rs}\cdot t\cdot rw_{rs}\cdot y'}+\xi_{s}\nt_{x'\cdot w_{rs}r\cdot t\cdot w_{rs}sr\cdot t\cdot rw_{rs}\cdot y'}\\
&\ \ \ +\xi_{t}\nt_{x'\cdot w_{rs}\cdot tsrst\cdot sw_{rs}\cdot y'}+\nt_{x'\cdot w_{rs}s\cdot tsrst\cdot sw_{rs}\cdot y'}\\
&=\xi_{t}\xi_{s}(-\sum_{z<w_{rs}}p_{z,w_{rs}}\nt_{x'\cdot w_{rs}r\cdot t\cdot z}\nt_{tsrst\cdot srw_{rs}\cdot y'})+\xi_{t}\nt_{x'\cdot w_{rs}r\cdot t\cdot w_{rs}s\cdot tsrst\cdot srw_{rs}\cdot y'}\\
&\ \ \ +\xi_{s}^{2}\xi_{r}(-\sum_{z<w_{rs}\atop z\neq w_{rs}s,w_{rs}r}p_{z,w_{rs}}\nt_{x'\cdot w_{rs}r\cdot t\cdot z}\nt_{t\cdot rw_{rs}\cdot y'})+q^{-2L(s)}\xi_{s}\xi_{r}\nt_{x'\cdot w_{rs}r\cdot t\cdot w_{rs}s\cdot t\cdot rw_{rs}\cdot y'}\\
&\ \ \ +q^{-2L(r)}\xi_{s}^{2}\nt_{x'\cdot w_{rs}r\cdot t\cdot w_{rs}r\cdot t\cdot rw_{rs}\cdot y'}+\xi_{s}\nt_{x'\cdot w_{rs}r\cdot t\cdot w_{rs}rs\cdot t\cdot rw_{rs}\cdot y'}\\
&\ \ \ +\xi_{r}\nt_{x'\cdot w_{rs}r\cdot t\cdot w_{rs}\cdot t\cdot rw_{rs}\cdot y'}+\xi_{s}\nt_{x'\cdot w_{rs}r\cdot t\cdot w_{rs}sr\cdot t\cdot rw_{rs}\cdot y'}\\
&\ \ \ +\xi_{t}\nt_{x'\cdot w_{rs}\cdot tsrst\cdot sw_{rs}\cdot y'}+\nt_{x'\cdot w_{rs}s\cdot tsrst\cdot sw_{rs}\cdot y'}.
 \end{align*}
Using Corollary \ref{cor:est} and the assumption $ L(s)>L(r) $, $ m_{rs}=8 $, we have
$$
\begin{aligned}
\deg (\nt_{x}\nt_w\nt_y)&\leq \max \{3L(s)-6L(r),L(s), 2L(s)-2L(r)\}\\
&<3L(s)-2L(r)\\
&=L'(d_{I})-L'(rs)\\
&=-\deg p_{w,d}.
 \end{aligned}
$$

This completes the proof of the proposition.

\section{Conclusions}\label{sec:9}

\subsection{P1-P15}
Now we arrive at the following main result.

\begin{thm}\label{thm:91}
Conjectures P1-P15 hold for any  weighted Coxeter group of rank 3.
\end{thm}
\begin{proof}
 First note that the boundedness conjecture for Coxeter groups of rank 3 has been proved in \cite{zhou,gao}.
	
Assume $ (W,S) $  is an irreducible Coxeter group with $ S=\{r,s,t\} $. If $ m_{rs} $, $ m_{st} $, $ m_{rt} $ are all odd, then $ (W,S) $ only admits a constant weight function, and P1-P15 hold by \cite[\S16]{lusztig2003}. The only irreducible finite Coxeter group which admits unequal parameters is of type $ B_3 $, and P1-P15 can be proved  directly as noted by \cite[25.1A]{bonnafe2017book}. For affine Weyl groups of type $ \tilde{B}_2 $ and $ \tilde{G}_2 $, P1-P15 are proved by Guilhot and Parkinson in  \cite{guilhot-parkinson2019G2,guilhot-parkinson2019C2}. For the case of $ m_{rs}\neq 2 $, $ m_{st} \neq 2$, $ m_{rt} \neq 2$ and  the case of $ m_{rt}=2 $, $ m_{rs}=m_{st}=\infty $, P1-P15  follow from \cite{xie2019}. The remaining cases are those Coxeter groups  listed in section \ref{sec:intr}, and P1-P15 follow from Theorem \ref{thm:main} and Propositions \ref{prop:bound}, \ref{prop:strictineq} and Lemma \ref{lem:length}.
\end{proof}

\subsection{Kazhdan-Lusztig Cells}
 The cell partitions of affine Weyl groups of type $ \tilde{B}_2 $ and $ \tilde{G}_2 $ are due to  \cite{guilhot2010cells,guilhot2008computedrank2}. The cell partitions for Coxeter groups with complete graph are also clear by \cite[\S8]{xie2019} in certain sense.

 In the remainder of this section we assume that $ S=\{r,s,t\} $ with $ m_{rt}=2 $ and $ \frac{1}{m_{rs}}+\frac{1}{m_{st}}<\frac12 $, $ 2<m_{rs} \leq \infty $, $ 2<m_{st} \leq \infty $, and discuss the cell partitions. 
 
 We use the notations from section \ref{sec:3}. First, we have the following results on $ \af $-functions and cells.
\begin{thm}\label{thm:92}
Let $\mathbb{A}=\{N\in \mathbb{N}\mid D_{N}\neq \emptyset\}$.
\begin{itemize}
\item [(1)] For any $d\in D$, $\af(d)=\af'(d)$ .
\item [(2)] For any $N\in \mathbb{N}$, $W_{N}=\Omega_{N}$.
\item [(3)] For any $N\in \mathbb{A}$, the following are decompositions into right cells
$$\Omega_{N}=\bigsqcup_{d\in D_{N},b\in B_{d}}bdU_{d},$$
$$W=\bigsqcup_{N\in \mathbb{A},d\in D_{N},b\in B_{d}}bdU_{d}.$$
\item [(4)] For any $N\in \mathbb{A}$, the set $ \Omega_{\leq N} $ is $ \prec_{ {LR}} $-closed, and hence $ \Omega_{ N} $ is a union of some two-sided cells. For any $d\in D_{N}$, $ B_d d U_d $ is contained in a two-sided cell.
\end{itemize}
\end{thm}

\begin{proof}
See Theorem \ref{thm:main} and \cite[Thm. 6.13]{xie2019}. Here we only show that $ B_d d U_d $ is contained in a two-sided cell. If $ z\in B_d d U_d $, then $ z\prec_{ {LR}} d $ and $ \af(z)=\af(d) $. By P11, we get $ z\sim_{ {LR}} d $. Thus $ z\sim_{ {LR}} d \sim_{ {LR}} z'$ for any $z,z'\in B_d d U_d $.
\end{proof}

\begin{lem}\label{lem:2cell}
If $ B_d =U_d^{-1}$ with $ d\in D $, then $ B_ddU_d $ is a two-sided cell.
\end{lem}

\begin{proof}
By Theorem \ref{thm:92}, $ B_d d U_d $ is contained in a two-sided cell. Assume $z\in \Omega_{\af(d)}$ satisfies $ z\sim_{ {LR}} d $. By the definition of $\sim_{ {LR}}$, there exists $z_{1},\ldots,z_{n}\in W$ such that $z_{1}=z$, $z_{n}=d$, and $z_{i}\prec_{ {L}} z_{i+1}$ or $z_{i}\prec_{ {R}} z_{i+1}$ for any $1\leq i\leq n-1$. By P4, we have
$\af(z_{1})\geq\ldots\geq\af(z_{n})$ and hence $\af(z_{1})=\ldots=\af(z_{n})$. By P9, P10, we get $z_{i}\sim_{ {L}} z_{i+1}$ or $z_{i}\sim_{ {R}} z_{i+1}$ for any $1\leq i\leq n-1$. Since $ B_d=U_d^{-1} $, $ B_d d U_d $ is stable by taking inverse. Thus $ B_ddU_d $ is a union of some right cells and in the same time a union of some left cells. Hence all the $ z_i $ and in particular $ z $ belong to $  BdU_d $.  This proves the lemma.
\end{proof}

Contrary to the complete graph case,  it is possible that $ \Omega_{ N} $ contains more than one two-sided cells.
\begin{ex}\label{ex:two}
Assume that $ m_{rt}=2 $, $ m_{rs}=4 $, $ m_{st}=5$, and  $ L(r)=5$, $L(s)=L(t)=1 $. Then we have
\[D_{0}=\{e\}, \quad D_1 =\{s,t\}, \quad D_5=\{r,w_{st}\},\]
\[D_6=\{rt\}, \quad D_9=\{rsr\},\quad D_{12}=\{w_{rs}\}.\]

Since $L(rt)=6>5$, we have $B_{w_{st}}=U_{w_{st}}=\{e\}$. By Lemma \ref{lem:2cell}, $\{w_{st}\}$ is a two-sided cell.  Therefore, $\Omega_{5}$ contains 2 two-sided cells: $\{w_{st}\}$ and $ B_{r}rU_{r} $.
\end{ex}

\begin{lem}\label{lem:connect}
Assume that $ d_1,d_2\in D_N $ and there exists some $ w $ in $ \Omega_{ N} $ such that $ w=d_1\cdot x =y\cdot d_2$ for some $ x,y\in W $. Then we have $ d\sim_{ {LR}} d_2 $.
\end{lem}

\begin{proof}
We have $ w\prec_{ {L}} d_2 $, $ w\prec_{ {R}} d_1 $ and $ \af(d_1)=\af(d_2)=\af(w) =N$. By P9, P10,  $ d_1\sim_{ {R}} w\sim_{ {L}} d_2 $.
\end{proof}

When $ m_{rs} $ (resp. $ m_{st}  $)  is  even, we often set $ m_{rs}=2m $ (resp. $ m_{st}=2n $). Let $ a=L(r) $, $ b=L(s) $, $ c=L(t) $.
Keep assumptions that $ m_{rt}=2 $ and $ \frac{1}{m_{rs}}+\frac{1}{m_{st}}<\frac12 $, where $ 2<m_{rs} \leq \infty $, $ 2<m_{st} \leq \infty $.
\begin{thm}\label{thm:twocells}
Assume that $ d_1, d_2$ belong to  $ D_N  $ for some $ N $ and lie in different two-sided cells. Then this happens if and only if   $ \{d_1,d_2\} $ is in one of the following cases, or the ones with  $ r , t $ (resp. $ a,c $) exchanged:
\begin{itemize}
\item[(1)] $ \{w_{rs},t \}$;
\item [(2)] $ \{r,t\} $ with $ b>a $.
\item[(3)] $ \{sw_{rs}, t \} $ with $a>b $;
\item [(4)]$ \{ sw_{rs},tw_{st} \} $ with $ a>b>c $ and $ a+c>N $;
\item [(5)]$ \{ sw_{rs}, sw_{st} \} $ with $ a>b<c $, and $ a+c>N $;
\item [(6)]$ \{rw_{rs} ,rt\}$ with $ a<b $, $ m_{st}=3 $.
\end{itemize}
Note that $ \af(d_1)=\af(d_2)  $ is a hidden restriction on  $ a,b,c $.
\end{thm}

\begin{proof}
The set $ D$ is a  union of  $A_1$ and $ A_2 $ with \[
A_1=\{e,r,s,w_{rs} \}\cup\{t, w_{st} \}\cup\{rt \},\quad A_2=\{d_{rs}, d_{st} \}
\]
where $ d_{rs}=\begin{cases}
rw_{rs}& \text{ if }  b>a \\
sw_{rs}& \text{ if }  b<a
\end{cases} $,
$ d_{st}=\begin{cases}
sw_{st}& \text{ if }  b<c \\
tw_{st}& \text{ if }  b>c
\end{cases} $.
When $ m_{rs}$ (resp. $ m_{st}$) is  $\infty $,  we ignore $ w_{rs} $ (resp. $ w_{st}$). When $a=b$ (resp. $b=c$),  we ignore $ d_{rs} $ (resp. $ d_{st}$). This does not affect the following proof.

In the following, we take any $ d_1,d_1\in D $, assume that $ d_1,d_2\in D_N $ for some $ N $, and check whether
\begin{itemize}
\item [(i)] there is an element $ w\in \Omega_{ N} $ such that $ w=d_1\cdot y=x\cdot d_2$ for some $ x, y\in W $, which will imply that $ d_1,d_2 $ are in the same two-sided cell by Lemma \ref{lem:connect},
\item [(ii)] or $ d_1,d_2 $ lie in different two-sided cells.
\end{itemize}

If $ d_1,d_2\in A_1 \cap D_N $, then  (i) holds except cases $ \{d_1,d_2\}= \{w_{rs},t \} ,\{r,w_{st} \} $ and $ \{r,t\} $ with $ b>a+c $. For example, if $ d_1=w_{rs},d_2=w_{st}$, then $ L(r)+\frac{L(s)}{2} \leq \frac N2$ and $ L(t)+\frac{L(s)}{2} \leq \frac N2$, which implies that $ rt\in D_{< N} $ and $ w=d_1sd_2 \in \Omega_{ N} $.
For the exceptional cases, (ii) holds by Lemma \ref{lem:2cell}.

It remains to consider the cases where one of $ d_1, d_2 $, say $ d_1 $, belongs to $ A_2 $. By the symmetric role of $ r,t $, we may only consider the case of $ d_1=d_{rs} $. Accordingly  possible choices of $ d_2$ are $ t, w_{st}, rt, d_{st} $ due to $ \af(d_1)=\af(d_2)$.

\noindent\textbf{Case(1)} $ a>b$.
\begin{itemize}
\item If $ d_2=t $, we have $B_{d_1}= U_{d_1} =\{e\}$. Thus by Lemma \ref{lem:connect}  $ d_1 $ and $ t $ are always in different two-sided cells.
\item  Assume  $ d_2=w_{st} $. In this case, we will prove $ \af(rt)<N=\af(d_1)=\af(d_2) $ and that (i) holds.
\begin{itemize}
\item If $ m_{rs}\neq 4 $ and $ m_{st}\neq 3 $, then we take $ w=d_1tsrd_2 $. Since $ 3a-2b\leq N $ and $ 2b+2c\leq N $, we have $ a+c<N $. Then we have $ w\in \Omega_N$.
\item  If $ m_{st}=3 $ ($ m_{rs}\geq 8 $), then we take $w= d_1 tsrsr d_2$. Since $ 4a-3b\leq N $ and $ 3b=N $, we have $ a+c=a+b<2a\leq N $. Hence  $ w\in \Omega_N$.
\item If $ m_{rs}=4 $ ($ m_{st}\geq 5 $), we take $ w= d_1tstsrstsrd_2$. Since $ 2a-b=N $, $ 3b+2c\leq N$, we have $ a+c<a+b+c\leq N $, and hence $ w\in \Omega_{ N} $.
\end{itemize}
\item If $ d_2=rt $, then we take $ w=d_1t $ and (i) holds.
\item Assume $ d_2=d_{st} $ with $ b<c $. If $ a+c\leq N $, then we take $ w=d_1d_2 $ and (i) holds.
Otherwise, $ B_{d_2}=U_{d_2}=\{e\} $, and hence $ \{d_2\} $ is a two-sided cell.
Since
$ N=ma-(m-1)b=nc-(n-1)b $, then $ a+c>N $ is equivalent to $ 1<a/b<\frac{mn-m}{mn-m-n} $. This indeed could  happen.
\item Assume $ d_2=d_{st} $ with $ b>c $. If $ a+c\leq N $, then (i) holds by taking \begin{itemize}
\item $ w=d_1tsrd_2 $ if $ m_{rs}\geq 6 $;
\item $ w= d_1tstsrd_2$ if $ m_{rs}=4 $ ($ m_{st}\geq 6$).
\end{itemize}
Then we consider the situation  $ a+c> N $. It implies that $ \{d_1\} $ is a  two-sided cell and hence $ d_1,d_2 $ are in different two-sided cells. Since
$ N=ma-(m-1)b=nb-(n-1)c $,  $ a+c>N $ is equivalent to $ 1<a/b<\frac{mn}{mn-n+1}$. This indeed could  happen.
\end{itemize}

\noindent\textbf{Case(2)} $ a<b$.  Then $ d_1=rw_{rs} $.
\begin{itemize}
\item If $ d_2=t $, then we take $ w=d_1t $.
\item Let $ d_2=w_{st} $.  We have $ 2b-a\leq N $ and $ 2c+b\leq N $, and hence $ a+c<\frac{3}{2}b-\frac{1}{2}a+c\leq N $. Take $ w=d_{1}sd_2 $ and (i) follows.
\item Let $ d_2=rt $.
\begin{itemize}
\item If $ m_{st}=3 $. Then one can check that $ B_{d_1}^{-1}=U_{d_1}=\{e,t\} $. Then $ B_{d_1} d_1 U_{d_1}$ is a two-sided cell by Lemma \ref{lem:connect},  and  $ d_1,d_2 $ lie in different two-sided cells.
\item If $ m_{st}=4 $, then $ m_{rs}\geq 6 $, and we take $ w= d_1tsrsd_2 $.
\item If $ m_{st}\geq 5 $, we take $ w=d_1tsd_2 $.
\end{itemize}
\item Let $ d_2=d_{st} $.
\begin{itemize}
\item If $ b>c $, then we take $ w=d_1sd_2 $.
\item If $ b<c $, then this  can be reduced to  the last one of  Case(1)  by exchanging $ r $, $ t $.
\end{itemize}
\end{itemize}
This completes the proof.
\end{proof}

\begin{cor}
Assume that $(W,S)$ is an irreducible hyperbolic Coxeter group of rank 3 with $ L $ constant. Then for any $ N\in \mathbb{A}$, $ \Omega_{N} $ is precisely a two-sided cell.
\end{cor}

This corollary confirms \cite[Conj. 3.1]{Belolipetsky2014hyperbolic_cells} in the case of rank 3, see the homepage of Paul Gunnells for some beautiful figures about cells partitions in the equal parameter case.

\begin{thm}\label{thm:3cells}
There are at most three elements in $ D_{N} $, and $ |D_N|=3 $ occurs in one of the following situations:
\begin{itemize}
\item [(1)] $ D_N=\{r,s,t\} $, $ a=b=c $;
\end{itemize}
(For other cases, $ m_{rs}=2m $, $ m_{st}=2n $ for some $ m,n \in \mathbb{N}$.)
\begin{itemize}
\item [(2)] $ D_N=\{rw_{rs},rt,sw_{st} \} $,  $(a/b,c/b)=\left (\frac{(m-1)(n-1)}{mn-m+1},\frac{mn}{mn-m+1} \right )$;
\item [(3)] $ D_N=\{rw_{rs},rt,tw_{st} \} $,  $(a/b,c/b)=\left (\frac{(m-1)n}{mn-1},\frac{m(n-1)}{mn-1} \right )$;
\item [(4)] $ D_N=\{sw_{rs},rt,sw_{st} \} $,  $(a/b,c/b)=\left (\frac{m(n-1)}{mn-m-n},\frac{(m-1)n}{mn-m-n} \right )$;
\item [(5)] $ D_N=\{sw_{rs},rt,tw_{st} \} $,  $(a/b,c/b)=\left (\frac{mn}{mn-n+1},\frac{(m-1)(n-1)}{mn-n+1} \right )$.
\end{itemize}
If $ |D_N| =3$,  $ D_N $ is contained in a two-sided cell.
\end{thm}
\begin{proof}
The proof is straightforward by choosing three elements $ d_1,d_2,d_3 $ from the set \[
\{ r,s,rw_{rs}(a<b), sw_{rs}(a>b), w_{rs}\}\cup \{t, sw_{st}(b<c), tw_{st}(b>c) ,w_{st} \}\cup \{ rt\},
\]
and then using $ \af(d_1)=\af(d_2)=\af(d_3) $ to determine the $ a/b, c/b $. Note that when $ m_{rs}$ or $ m_{st} $ is $\infty $, we need to ignore some elements. We omit the details.
\end{proof}

\subsection{Examples, I}
Assume that $ b=c $, $ m_{rs} =2m$ and $ k=m_{st} $ for some $ m,k\in \mathbb{N} $.
The cell partition is determined by $ a/b \in \mathbb{R}_{>0}$.

 Let $ d_1,d_2\in D $ with $ \af(d_1)=\af(d_2) $. This equation on $ a,b $ gives a value $ h $ of $ a/b $.  We call $ h $  a critical value  when $ d_1 ,d_2$ are  in the same two-sided cell.  By definition of $ \Omega_{ N} $ and Theorem \ref{thm:92}, the cell partition is completely determined by the relative position of $ a/b $ with all such critical values.

By Theorem \ref{thm:twocells}, the possible critical values are given by equations\[
\af(r)=\af(s),
\]
\[
\af(rw_{rs})=\af(w_{st})\text{ with }a<b,
\]\[
\af(sw_{rs})=\af(w_{st})\text{ with }a>b,
\]\[
\af(w_{rs})=\af(w_{st}),
\]
\[
\af(rw_{rs})=\af(rt) \text{ with }a<b, m_{st}\neq 3,
\]
\[
\af(sw_{rs})=\af(rt)\text{ with }a>b,
\]\[
 \af(w_{st})=\af(rt).
\]
Thus possible critical values of $ a/b $ are\[
1,\frac{m-k}{m-1}, \frac{m+k-1}{m},\frac{m-1}{m} (\text{for } k\neq 3), \frac{k-m}{m},\frac{m}{m-1}, k-1.
\]

For the hyperbolic Coxeter group $ 245 $, i.e. $ m=2 $, $ k=5 $, the critical values are
$ \frac12,1,\frac32,2,3,4. $
\begin{center}
\begin{tikzpicture}
\draw[->] (0,0)--(6,0) node[right] {$ a/b $};
\foreach \x/\xtext in {0, 0.5/\frac12, 1,1.5/\frac32,2,3,4}
\draw (\x cm,2pt) -- (\x cm,0) node[anchor=north] {$\xtext$};
\end{tikzpicture}
\end{center}
The critical values for the hyperbolic Coxeter group $ 238 $ are
$ \frac13,1, \frac32,\frac43,2. $
\begin{center}
\begin{tikzpicture}
\draw[->] (0,0)--(6,0) node[right] {$ a/b $};
\foreach \x/\xtext in {0, 0.3333/\frac13,1,1.5/ \frac32,2}
\draw (\x cm,2pt) -- (\x cm,0) node[anchor=north] {$\xtext$};
\draw (1.3333cm,2pt) -- (1.3333cm,0) node[anchor=south] {$\frac43$};
\end{tikzpicture}
\end{center}


\subsection{Examples, II}

Assume that $ m_{rs}=2m $, $ m_{st}=2n $ for some $ m $, $ n \in\mathbb{N}$.
Then $ (W,S) $ admits three parameters.

We can draw  the points $ (a/b,c/b) \in \mathbb{R}_{>0}^2 $ such that there exists some elements $ d_1\neq d_2$  of $ D $ with $ \af(d_1) =\af( d_2) $ and $ d_1,d_2 $ are in the same two-sided cell. These form a partition of $ \mathbb{R}_{>0}^2 $  by some linear hyperplanes (see Figure \ref{fig:246}). We call them critical hyperplanes.
By definition of $ \Omega_{ N} $ and Theorem \ref{thm:92},
the cell partition of $ (W,S) $ is completely determined by  the relative position of $ (a/b,c/b) $ with critical hyperplanes \footnote{These hyperplanes should be the ones  in Bonnaf{\'e}'s semi-continuity conjecture, see  \cite{bonnafe2009semi-cont}.}. This visualizes how the cell partitions depend on parameters.  Note that, for example, if $ d_1=rw_{rs} $,  we need to require $ a<b $, i.e. $ 0<a/b<1 $.

\tiny

\begin{figure}[h]
     \centering
\newcommand{\m}{2}
\newcommand{\n}{3}
\renewcommand{\d}{20}
\begin{tikzpicture} [domain=0:\d,scale=0.225]
\draw[->](0,0)--(\d+.2,0) node[right]{$ a/b$};
\draw [->](0,0)--(0,\d+0.2) node[above]{$ c/b $};

\foreach \x in {0,2,...,\d}
 \draw (\x,0) node[below]{$ { \x} $};

\foreach \x in {2,4,...,\d}
    \draw (0,\x) node[left]{$ { \x} $};

\clip (0,0) rectangle (\d,\d);

\draw (1,0)--(1,\d)node[above]{$ r,s$}; 
\draw plot(\x, 1)node[right]{$ s,t $}; 
\draw plot(\x, \x)node[right]{$ r,t $};

\draw  plot(\x,1-\x) node[right]{$ rt,s $};
\draw [domain=1:\d]plot(\x, \m*\x-\x-\m+1) node[right]{$ rt,rsr $};
\draw [domain=0:1] plot(\x,\m-\m*\x);
\draw plot(\x,{1+\x*1/(\n-1)})node[right]{$ rt,tst $};
\draw plot(\x,{1-1/\n*\x}) node[right]{$ rt,sts $};
\draw plot(\x,{\m+(\m-1)*\x}) node [right]{$ rt,w_{rs} $};
\draw plot(\x,{1/(\n-1)*(\x-\n)}) node [right]{$ rt,w_{st} $};

\draw[domain=1:\d,densely dotted] plot(\x,{\m*\x-(\m-1)}) node [right]{$ t,rsr$};
\draw [domain=0:1] plot(\x,{\m-(\m-1)*\x}); 
\draw [domain=0:\d/2,densely dotted] plot(\x,{\m*(1+\x)}) node [right]{$ t,w_{rs}$};
\draw [domain=1:\d]plot(\x,{1/(\n-1)*(\n-\x)}) node [right]{$ r,sts $};
\draw[domain=1:\d,densely dotted]  plot(\x,{1/\n*(\n-1+\x)}) node [right]{$ r,tst $};
\draw[densely dotted]  plot(\x,{1/\n*\x-1}) node [right]{$ r,w_{st} $};

\draw [densely dotted,domain=1:\m*\n/(\m*\n+1-\n)]  plot(\x,{1/(\n-1)*(\m+\n-1-\m*\x)});
\draw [domain=\m*\n/(\m*\n+1-\n):\d]  plot(\x,{1/(\n-1)*(\m+\n-1-\m*\x)}) node [right]{$ sts,rsr$};
\draw [domain=0:1]plot(\x,{1/(\n-1)*(\n-\m+(\m-1)*\x)}); 

\draw [densely dotted,domain=1:(\m*\n-\m)/(\m*\n-\m-\n)]plot(\x,{1/\n*(\m*\x+\n-\m)}); 
\draw [domain=(\m*\n-\m)/(\m*\n-\m-\n):\d]plot(\x,{1/\n*(\m*\x+\n-\m)});

\draw [domain=0:(\m-1)*(\n-1)/(\m*\n+1-\m)]plot(\x,{1/\n*(\n+\m-1-(\m-1)*\x)}) ;
\draw [densely dotted,domain=(\m-1)*(\n-1)/(\m*\n+1-\m):1]plot(\x,{1/\n*(\n+\m-1-(\m-1)*\x)}); 

\draw plot(\x,{1/(\n-1)*(\n-\m-\m*\x)}) node [right] {$ w_{rs},sts $};
\draw plot(\x,{1/\n*(\m*\x+\m+\n-1)}) node [right] {$ w_{rs},tst $};
\draw [domain=0:1]plot(\x,{1/\n*(\m-\n-(\m-1)*\x)}); 
\draw [domain=1:\d]plot(\x,{1/\n*(\m*\x-\m-\n+1)}) node [right] {$ w_{st},rsr $};

\draw plot(\x,{1/\n*(\m*\x+\m-\n)})node [right] {$ w_{rs},w_{st} $};
\draw (14,8)node[above=1pt]{$ F $};
\end{tikzpicture}
\renewcommand{\m}{2}
\renewcommand{\n}{3}
\renewcommand{\d}{5}
\begin{tikzpicture} [domain=0:\d,scale=0.9]
\draw[->](0,0)--(\d+.06,0) node[right]{$ a/b$};
\draw [->](0,0)--(0,\d+0.06) node[above]{$ c/b $};

\foreach \x in {0,2,...,\d}
 \draw (\x,0) node[below]{$ { \x} $};

\foreach \x in {2,4,...,\d}
    \draw (0,\x) node[left]{$ { \x} $};

\clip (0,0) rectangle (\d,\d);

\draw (1,0)--(1,\d)node[above]{$ r,s$}; 
\draw plot(\x, 1)node[right]{$ s,t $}; 
\draw plot(\x, \x)node[right]{$ r,t $};

\draw  plot(\x,1-\x) node[right]{$ rt,s $};
\draw [domain=1:\d]plot(\x, \m*\x-\x-\m+1) node[right]{$ rt,rsr $};
\draw [domain=0:1] plot(\x,\m-\m*\x);
\draw plot(\x,{1+\x*1/(\n-1)})node[right]{$ rt,tst $};
\draw plot(\x,{1-1/\n*\x}) node[right]{$ rt,sts $};
\draw plot(\x,{\m+(\m-1)*\x}) node [right]{$ rt,w_{rs} $};
\draw plot(\x,{1/(\n-1)*(\x-\n)}) node [right]{$ rt,w_{st} $};

\draw[domain=1:\d,densely dotted] plot(\x,{\m*\x-(\m-1)}) node [right]{$ t,rsr$};
\draw [domain=0:1] plot(\x,{\m-(\m-1)*\x}); 
\draw [domain=0:\d/2,densely dotted] plot(\x,{\m*(1+\x)}) node [right]{$ t,w_{rs}$};
\draw [domain=1:\d]plot(\x,{1/(\n-1)*(\n-\x)}) node [right]{$ r,sts $};
\draw[domain=1:\d,densely dotted]  plot(\x,{1/\n*(\n-1+\x)}) node [right]{$ r,tst $};
\draw[densely dotted]  plot(\x,{1/\n*\x-1}) node [right]{$ r,w_{st} $};

\draw [densely dotted,domain=1:\m*\n/(\m*\n+1-\n)]  plot(\x,{1/(\n-1)*(\m+\n-1-\m*\x)});
\draw [domain=\m*\n/(\m*\n+1-\n):\d]  plot(\x,{1/(\n-1)*(\m+\n-1-\m*\x)}) node [right]{$ sts,rsr$};
\draw [domain=0:1]plot(\x,{1/(\n-1)*(\n-\m+(\m-1)*\x)}); 

\draw [densely dotted,domain=1:(\m*\n-\m)/(\m*\n-\m-\n)]plot(\x,{1/\n*(\m*\x+\n-\m)}); 
\draw [domain=(\m*\n-\m)/(\m*\n-\m-\n):\d]plot(\x,{1/\n*(\m*\x+\n-\m)});

\draw [domain=0:(\m-1)*(\n-1)/(\m*\n+1-\m)]plot(\x,{1/\n*(\n+\m-1-(\m-1)*\x)}) ;
\draw [densely dotted,domain=(\m-1)*(\n-1)/(\m*\n+1-\m):1]plot(\x,{1/\n*(\n+\m-1-(\m-1)*\x)}); 

\draw plot(\x,{1/(\n-1)*(\n-\m-\m*\x)}) node [right] {$ w_{rs},sts $};
\draw plot(\x,{1/\n*(\m*\x+\m+\n-1)}) node [right] {$ w_{rs},tst $};
\draw [domain=0:1]plot(\x,{1/\n*(\m-\n-(\m-1)*\x)}); 
\draw [domain=1:\d]plot(\x,{1/\n*(\m*\x-\m-\n+1)}) node [right] {$ w_{st},rsr $};

\draw plot(\x,{1/\n*(\m*\x+\m-\n)})node [right] {$ w_{rs},w_{st} $};

\draw(2/5,6/5) node[left=2pt]{$ A $};
\draw(3/5,4/5) node[left=2pt]{$ B $};
\draw(3/2,1/2) node[below]{$ C $};
\draw(2,1) node[below]{$ D$};
\draw(4,3) node[below]{$ E$};
\draw(1,1) node[below]{$ O$};
\end{tikzpicture}

        \caption{Hyperplanes $ \af(d_1)=\af(d_2) $ for the Coxeter group 246. Real lines are critical hyperplanes while  dotted lines are $ \af(d_1)=\af(d_2) $ with  $ d_1, d_2$  in different two-sided cells.}
        \label{fig:246}
\end{figure}
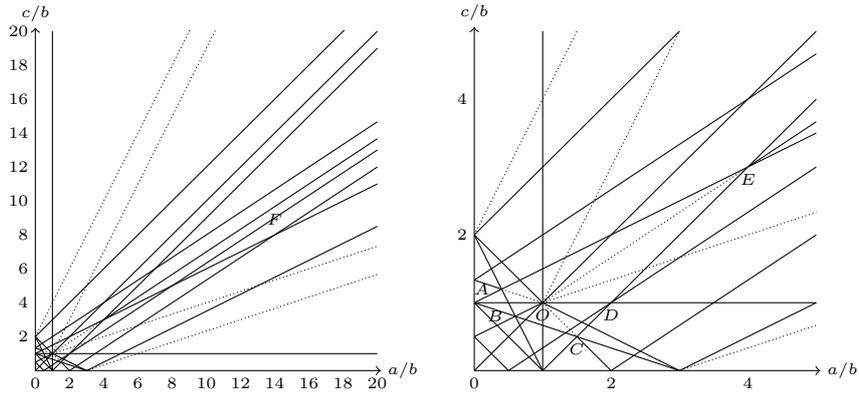

\newcommand{\sca}{0.7}
\newcommand{\picwid}{0.3}
\newcommand{\pla}{(-3,-2)}

\begin{figure}[h]
\centering
\begin{tikzpicture}[scale=\sca]
 \node at \pla {\includegraphics[width=\picwid\linewidth]{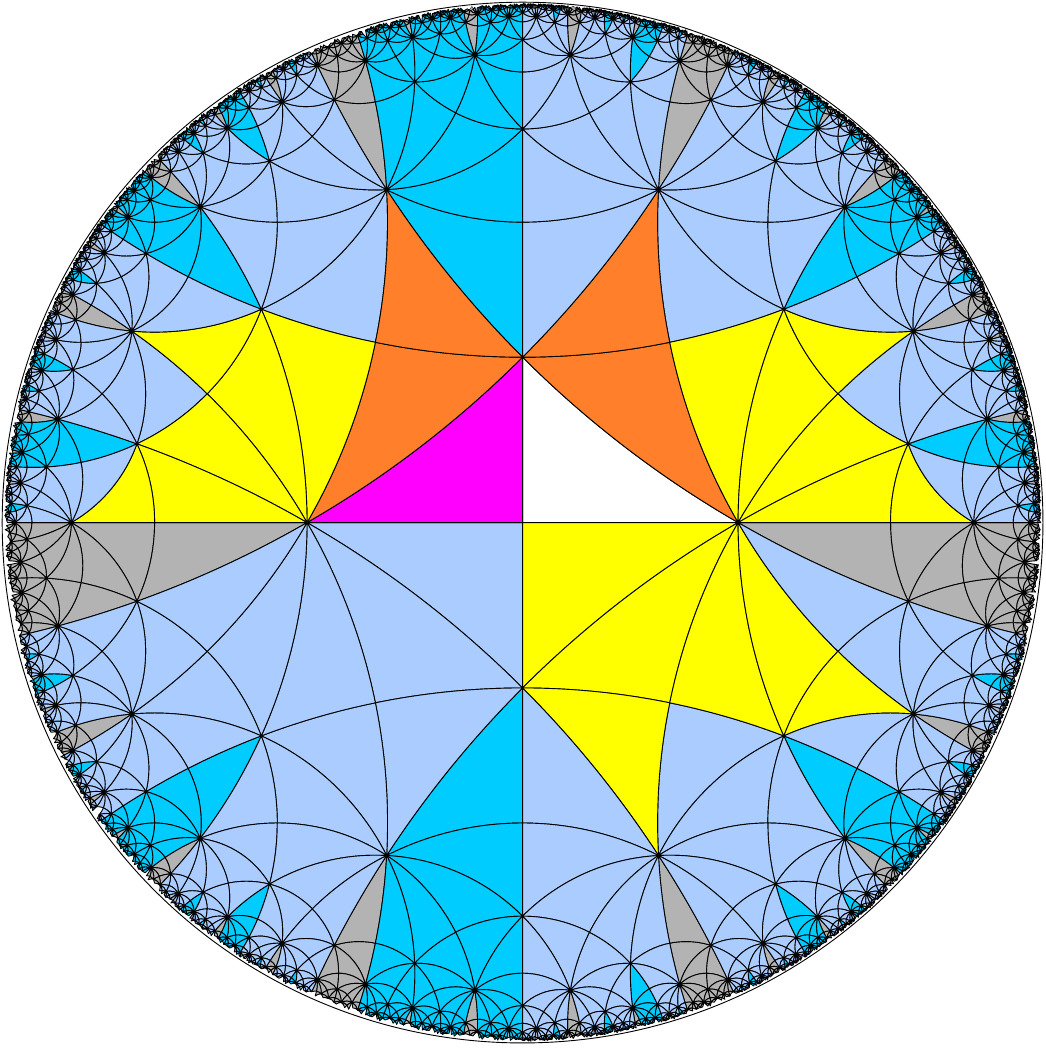}};
\draw (0,0.05) node[above=1pt]{$ r $} circle [radius=0.05];\draw (0.05,0.05)--++(0.65,0); \draw (0.75,0.05) node[above=1pt]{$ s $} circle[radius=0.05];  \draw (0.8,0.05)--++(0.65,0); \draw  (1.45,0) ++(0.05,0.05) node[above=1pt]{$ t $} circle[radius=0.05];
\draw (0.35,0.05)node [above=3pt]{$ 4 $}  ++(0.8,0) node[above=3pt]{$ 6 $};
\draw (0,0.05) node[below=1pt]{$ 2 $};
\draw (0.75,0.05) node[below=1pt]{$ 5 $};
\draw  (1.45,0) ++(0.05,0.05) node[below=1pt]{$ 6 $};

\filldraw[fill={rgb,255:red,255; green,255; blue,255}][xshift=0 cm,yshift=-1cm] (0,0) rectangle (0.2,0.2) +(0,-0.1);
\draw [yshift=-1cm] (0.3,0.1)node[right] {$ e $};

\filldraw[fill={rgb,255:red,255; green,0; blue,255}][xshift=0 cm,yshift=-1.5cm] (0,0) rectangle (0.2,0.2) +(0,-0.1);
\draw [yshift=-1.5cm] (0.3,0.1)node[right] {$ r$};

\filldraw[fill={rgb,255:red,255; green,127; blue,42}][xshift=0 cm,yshift=-2cm] (0,0) rectangle (0.2,0.2) +(0,-0.1);
\draw [yshift=-2cm] (0.3,0.1)node[right] {$ s$};

\filldraw[fill={rgb,255:red,255; green,255; blue,0}][xshift=0 cm,yshift=-2.5cm] (0,0) rectangle (0.2,0.2) +(0,-0.1);
\draw [yshift=-2.5cm] (0.3,0.1)node[right] {$ t$};

\filldraw[fill={rgb,255:red,170; green,204; blue,255}][xshift=0 cm,yshift=-3cm] (0,0) rectangle (0.2,0.2) +(0,-0.1);
\draw [yshift=-3cm] (0.3,0.1)node[right] {$ rw_{rs}$, $ rt $, $ sw_{st} $};

\filldraw[fill={rgb,255:red,0; green,204; blue,255}][xshift=0 cm,yshift=-3.5cm] (0,0) rectangle (0.2,0.2) +(0,-0.1);
\draw [yshift=-3.5cm] (0.3,0.1)node[right] {$ w_{rs}$};

\filldraw[fill={rgb,255:red,179; green,179; blue,179}][xshift=0 cm,yshift=-4cm] (0,0) rectangle (0.2,0.2) +(0,-0.1);
\draw [yshift=-4cm] (0.3,0.1)node[right] {$ w_{st}$};
\end{tikzpicture}
\begin{tikzpicture}[scale=\sca]
 \node at \pla {\includegraphics[width=\picwid\linewidth]{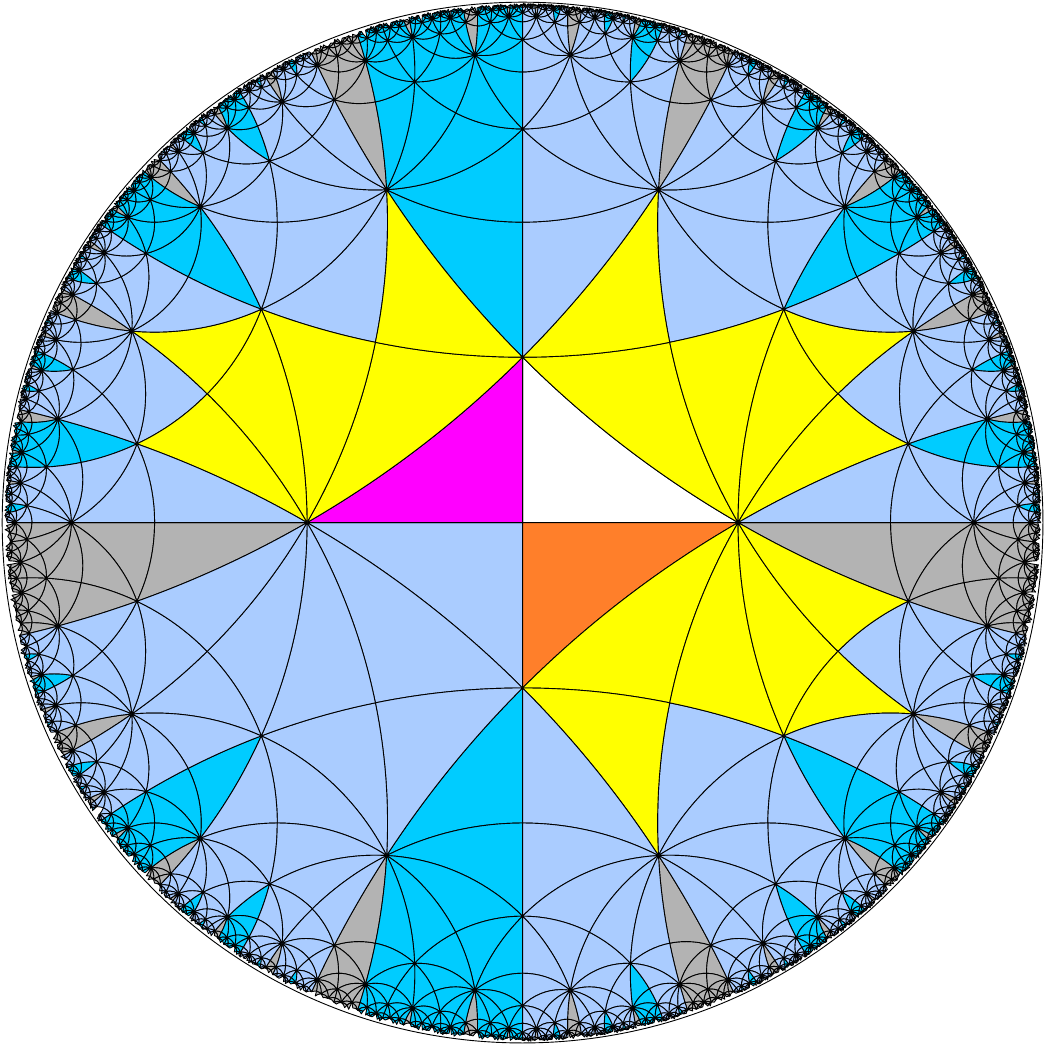}};
\draw (0,0.05) node[above=1pt]{$ r $} circle [radius=0.05];\draw (0.05,0.05)--++(0.65,0); \draw (0.75,0.05) node[above=1pt]{$ s $} circle[radius=0.05];  \draw (0.8,0.05)--++(0.65,0); \draw  (1.45,0) ++(0.05,0.05) node[above=1pt]{$ t $} circle[radius=0.05];
\draw (0.35,0.05)node [above=3pt]{$ 4 $}  ++(0.8,0) node[above=3pt]{$ 6 $};
\draw (0,0.05) node[below=1pt]{$ 3$};
\draw (0.75,0.05) node[below=1pt]{$ 5 $};
\draw  (1.45,0) ++(0.05,0.05) node[below=1pt]{$ 4$};

\filldraw[fill={rgb,255:red,255; green,255; blue,255}][xshift=0 cm,yshift=-1cm] (0,0) rectangle (0.2,0.2) +(0,-0.1);
\draw [yshift=-1cm] (0.3,0.1)node[right] {$ e $};

\filldraw[fill={rgb,255:red,255; green,0; blue,255}][xshift=0 cm,yshift=-1.5cm] (0,0) rectangle (0.2,0.2) +(0,-0.1);
\draw [yshift=-1.5cm] (0.3,0.1)node[right] {$ r$};

\filldraw[fill={rgb,255:red,255; green,127; blue,42}][xshift=0 cm,yshift=-2cm] (0,0) rectangle (0.2,0.2) +(0,-0.1);
\draw [yshift=-2cm] (0.3,0.1)node[right] {$ t$};

\filldraw[fill={rgb,255:red,255; green,255; blue,0}][xshift=0 cm,yshift=-2.5cm] (0,0) rectangle (0.2,0.2) +(0,-0.1);
\draw [yshift=-2.5cm] (0.3,0.1)node[right] {$ s$};

\filldraw[fill={rgb,255:red,170; green,204; blue,255}][xshift=0 cm,yshift=-3cm] (0,0) rectangle (0.2,0.2) +(0,-0.1);
\draw [yshift=-3cm] (0.3,0.1)node[right] {$ rw_{rs}$, $ rt $, $ tw_{st} $};

\filldraw[fill={rgb,255:red,0; green,204; blue,255}][xshift=0 cm,yshift=-3.5cm] (0,0) rectangle (0.2,0.2) +(0,-0.1);
\draw [yshift=-3.5cm] (0.3,0.1)node[right] {$ w_{rs}$};

\filldraw[fill={rgb,255:red,179; green,179; blue,179}][xshift=0 cm,yshift=-4cm] (0,0) rectangle (0.2,0.2) +(0,-0.1);
\draw [yshift=-4cm] (0.3,0.1)node[right] {$ w_{st}$};
\end{tikzpicture}
\begin{tikzpicture}[scale=\sca]
 \node at \pla {\includegraphics[width=\picwid\linewidth]{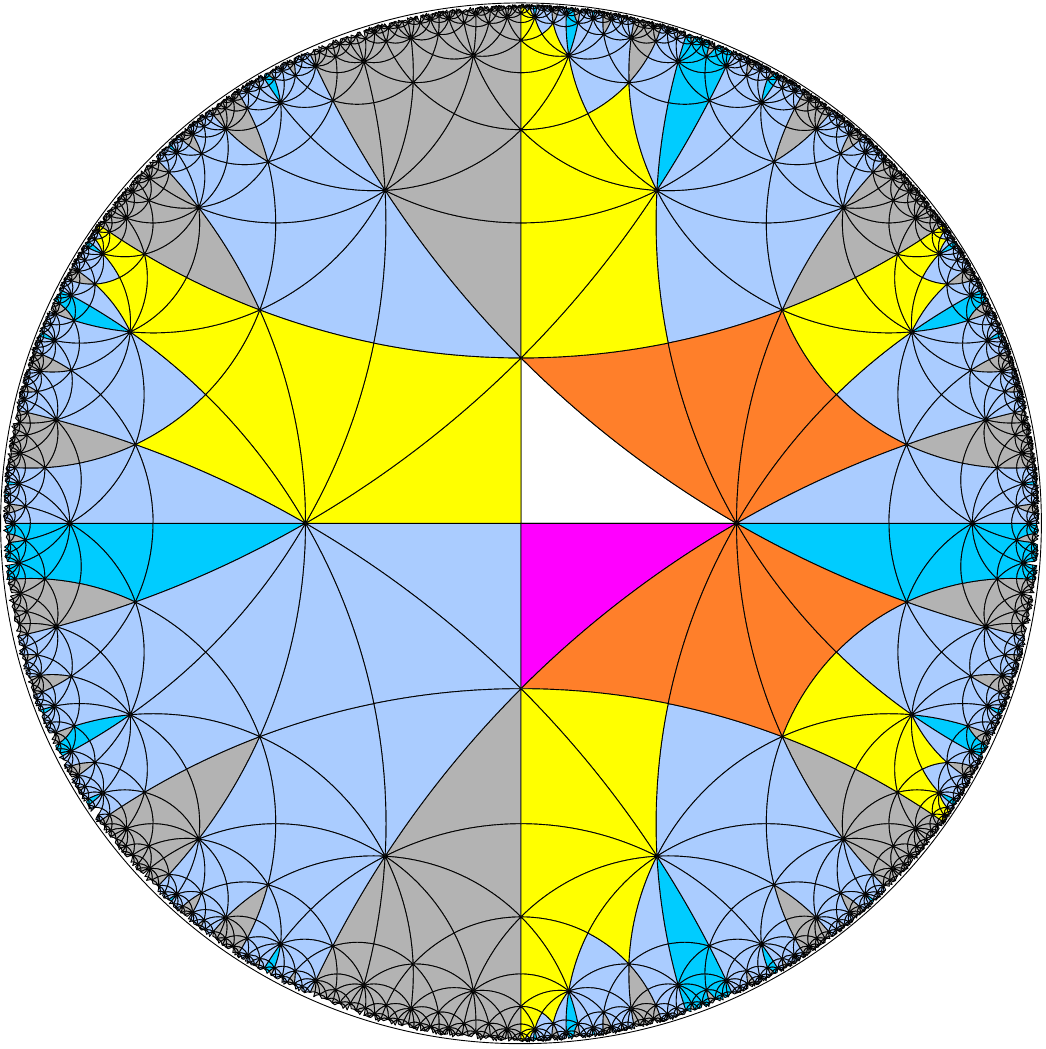}};
\draw (0,0.05) node[above=1pt]{$ r $} circle [radius=0.05];\draw (0.05,0.05)--++(0.65,0); \draw (0.75,0.05) node[above=1pt]{$ s $} circle[radius=0.05];  \draw (0.8,0.05)--++(0.65,0); \draw  (1.45,0) ++(0.05,0.05) node[above=1pt]{$ t $} circle[radius=0.05];
\draw (0.35,0.05)node [above=3pt]{$ 4 $}  ++(0.8,0) node[above=3pt]{$ 6 $};
\draw (0,0.05) node[below=1pt]{$ 3$};
\draw (0.75,0.05) node[below=1pt]{$ 2 $};
\draw  (1.45,0) ++(0.05,0.05) node[below=1pt]{$ 1$};

\filldraw[fill={rgb,255:red,255; green,255; blue,255}][xshift=0 cm,yshift=-1cm] (0,0) rectangle (0.2,0.2) +(0,-0.1);
\draw [yshift=-1cm] (0.3,0.1)node[right] {$ e $};

\filldraw[fill={rgb,255:red,255; green,0; blue,255}][xshift=0 cm,yshift=-1.5cm] (0,0) rectangle (0.2,0.2) +(0,-0.1);
\draw [yshift=-1.5cm] (0.3,0.1)node[right] {$ t$};

\filldraw[fill={rgb,255:red,255; green,127; blue,42}][xshift=0 cm,yshift=-2cm] (0,0) rectangle (0.2,0.2) +(0,-0.1);
\draw [yshift=-2cm] (0.3,0.1)node[right] {$ s$};

\filldraw[fill={rgb,255:red,255; green,255; blue,0}][xshift=0 cm,yshift=-2.5cm] (0,0) rectangle (0.2,0.2) +(0,-0.1);
\draw [yshift=-2.5cm] (0.3,0.1)node[right] {$ r$};

\filldraw[fill={rgb,255:red,170; green,204; blue,255}][xshift=0 cm,yshift=-3cm] (0,0) rectangle (0.2,0.2) +(0,-0.1);
\draw [yshift=-3cm] (0.3,0.1)node[right] {$ sw_{rs}$, $ rt $, $ tw_{st} $};

\filldraw[fill={rgb,255:red,0; green,204; blue,255}][xshift=0 cm,yshift=-3.5cm] (0,0) rectangle (0.2,0.2) +(0,-0.1);
\draw [yshift=-3.5cm] (0.3,0.1)node[right] {$ w_{st}$};

\filldraw[fill={rgb,255:red,179; green,179; blue,179}][xshift=0 cm,yshift=-4cm] (0,0) rectangle (0.2,0.2) +(0,-0.1);
\draw [yshift=-4cm] (0.3,0.1)node[right] {$ w_{rs}$};
\end{tikzpicture}
\begin{tikzpicture}[scale=\sca]
 \node at \pla {\includegraphics[width=\picwid\linewidth]{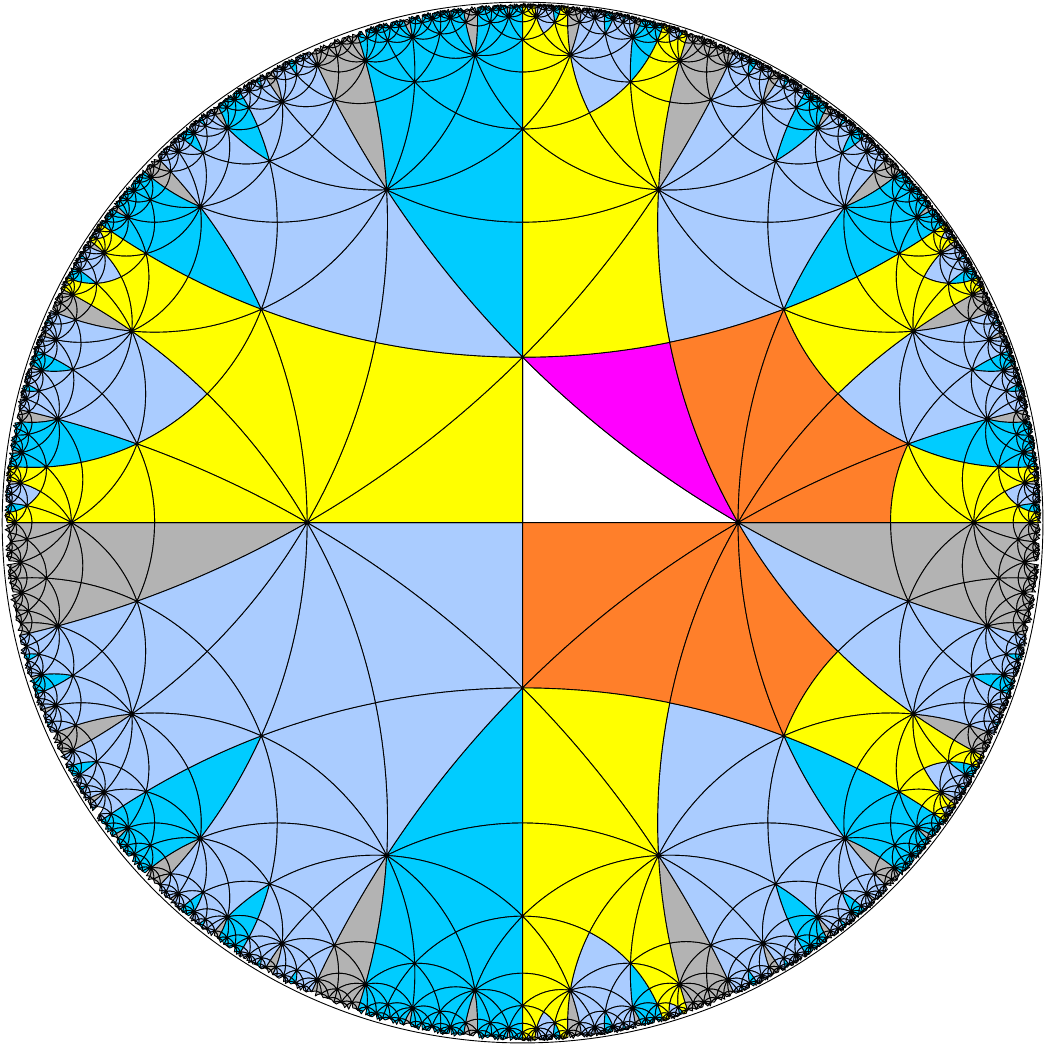}};
\draw (0,0.05) node[above=1pt]{$ r $} circle [radius=0.05];\draw (0.05,0.05)--++(0.65,0); \draw (0.75,0.05) node[above=1pt]{$ s $} circle[radius=0.05];  \draw (0.8,0.05)--++(0.65,0); \draw  (1.45,0) ++(0.05,0.05) node[above=1pt]{$ t $} circle[radius=0.05];
\draw (0.35,0.05)node [above=3pt]{$ 4 $}  ++(0.8,0) node[above=3pt]{$ 6 $};
\draw (0,0.05) node[below=1pt]{$ 4$};
\draw (0.75,0.05) node[below=1pt]{$ 1 $};
\draw  (1.45,0) ++(0.05,0.05) node[below=1pt]{$ 3$};

\filldraw[fill={rgb,255:red,255; green,255; blue,255}][xshift=0 cm,yshift=-1cm] (0,0) rectangle (0.2,0.2) +(0,-0.1);
\draw [yshift=-1cm] (0.3,0.1)node[right] {$ e $};

\filldraw[fill={rgb,255:red,255; green,0; blue,255}][xshift=0 cm,yshift=-1.5cm] (0,0) rectangle (0.2,0.2) +(0,-0.1);
\draw [yshift=-1.5cm] (0.3,0.1)node[right] {$ s$};

\filldraw[fill={rgb,255:red,255; green,127; blue,42}][xshift=0 cm,yshift=-2cm] (0,0) rectangle (0.2,0.2) +(0,-0.1);
\draw [yshift=-2cm] (0.3,0.1)node[right] {$ t$};

\filldraw[fill={rgb,255:red,255; green,255; blue,0}][xshift=0 cm,yshift=-2.5cm] (0,0) rectangle (0.2,0.2) +(0,-0.1);
\draw [yshift=-2.5cm] (0.3,0.1)node[right] {$ r$};

\filldraw[fill={rgb,255:red,170; green,204; blue,255}][xshift=0 cm,yshift=-3cm] (0,0) rectangle (0.2,0.2) +(0,-0.1);
\draw [yshift=-3cm] (0.3,0.1)node[right] {$ sw_{rs}$, $ rt $, $ sw_{st} $};

\filldraw[fill={rgb,255:red,0; green,204; blue,255}][xshift=0 cm,yshift=-3.5cm] (0,0) rectangle (0.2,0.2) +(0,-0.1);
\draw [yshift=-3.5cm] (0.3,0.1)node[right] {$ w_{rs}$};

\filldraw[fill={rgb,255:red,179; green,179; blue,179}][xshift=0 cm,yshift=-4cm] (0,0) rectangle (0.2,0.2) +(0,-0.1);
\draw [yshift=-4cm] (0.3,0.1)node[right] {$ w_{st}$};
\end{tikzpicture}
\caption{Cells of 246 with unequal parameters such that $ |D_N| \geq 3$ for some $ N $}
\label{fig:246-E}
\end{figure}
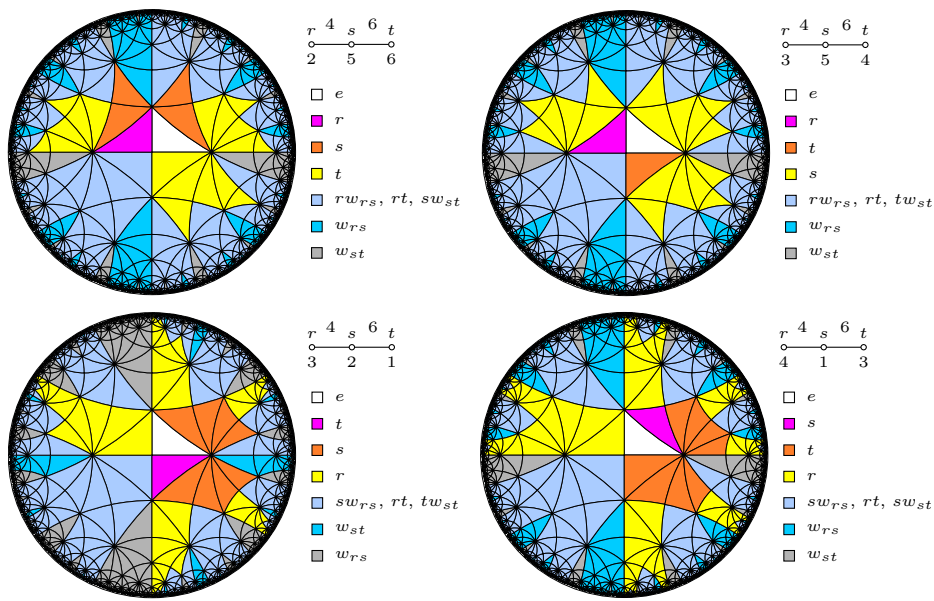

\normalsize

Let us consider the example with $ m_{rs}=4 $ and $ m_{st}=6 $. See  Figure \ref{fig:246} for critical hyperplanes.
The points with more than 3 hyperplanes passing are
$ O(1,1),$  $A(\frac25,$  $\frac65),$  $B(\frac35,$  $\frac45),$  $C(\frac32,$  $\frac12),$  $D(2,1),$  $E(4,3)$.
The furthest point with two hyperplanes passing is $ F(14,8) $.
By writing down  $ D_N $, one can see that only for points $ O,A,B,C,E $  there exists $ N $ such that $ |D_N|\geq 3 $, which is consistent with Theorem \ref{thm:3cells}. See Figure \ref{fig:246-E}  for cell partitions for points $ A,B,C,E $. For  the point $ D $, taking $ (a,b,c)=(2,1,1) $, we have $ D_1=\{s,t\} $, $ D_2=\{r\} $, $ D_3=\{rt,rsr\} $, $ D_6=\{rsrs,ststst\} $. For the point $ F $, taking $ (a,b,c)=(14,1,8) $,  we have $ D_1=\{s \} $, $ D_{8}=\{t\} $, $ D_{14}=\{r\} $, $ D_{22} =\{tstst\}$, $ D_{27}=\{rsr,ststst \} , $ $ D_{30}=\{rsrs\} $.


\subsection*{Acknowledgments}
We would like to thank Gaston Burrull for helping us color cells in alcoves,  thank Paul Gunnells for the beautiful pictures on his homepage about cells, and thank Mikhail Belolipetsky for helpful email correspondence.

The  second named author  would like to thank  the University of Sydney and Sydney Mathematical Research Institute for hospitality during his visit, and he is  supported by NSFC Grant No.11601116,  No. 11801031 and Beijing Institute of Technology Research Fund Program for Young Scholars.

\medskip
\bibliography{hyperbolic}

\end{document}